\documentclass[11pt, a4paper]{amsart}

\makeatletter
\newcommand*{\rom}[1]{\expandafter\@slowromancap\romannumeral #1@}
\makeatother
\usepackage{color}
\usepackage{amsmath,amssymb,hyperref}
\usepackage{amsfonts}
\usepackage{enumerate}
\usepackage{amsmath, amsthm}
\usepackage{amscd}
\usepackage[titletoc,title]{appendix}
\input xy
\xyoption{all}

      \theoremstyle{plain}
      \newtheorem{theorem}{Theorem}[section]
      \newtheorem{lemma}[theorem]{Lemma}
      \newtheorem{corollary}[theorem]{Corollary}
      \newtheorem{proposition}[theorem]{Proposition}
      \theoremstyle{definition}
      \newtheorem{definition}[theorem]{Definition}
      \newtheorem*{claim}{Claim}
      \newtheorem*{conjecture}{Conjecture}

      \theoremstyle{remark}
      \newtheorem{remark}[theorem]{Remark}

\newcommand{\rsl}{{\mathrm{SL}}}
\newcommand{\rsp}{{\mathrm{Sp}}}
\newcommand{\rso}{{\mathrm{SO}}}
\newcommand{\s}{{\mathrm{S}}}
\newcommand{\ru}{{\mathrm{U}}}
\newcommand{\su}{{\mathrm{SU}}}

\newcommand{\br}{{\mathbb R}}
\newcommand{\bq}{{\mathbb Q}}

\newcommand{\ra}{\rightarrow}

\newcommand{\tits}{\smash{\overline{\Gamma \backslash X\vphantom{()}}}^T}
\newcommand{\borelserre}{\smash{\overline{\Gamma \backslash X\vphantom{()}}}^{BS}}

      \makeatletter
      \def\@setcopyright{}
      \def\serieslogo@{}
      \makeatother

\begin{document}

%



   \author{Sungwoon Kim}
   \address{Center for Mathematical Challenges,
   Korea Institute for Advanced Study, Hoegiro 85, Dongdaemun-gu,
   Seoul, 130-722, Republic of Korea}
   \email{sungwoon@kias.re.kr}

   \author{Inkang Kim}
   \address{School of Mathematics, Korea Institute for Advanced Study, Hoegiro 85, Dongdaemun-gu,
   Seoul, 130-722, Republic of Korea}
   \email{inkang@kias.re.kr}






   \title[]{Simplicial volume, Barycenter method, and Bounded cohomology}


\begin{abstract}
We show that codimension one dimensional Jacobian of the barycentric straightening map is uniformly bounded for most of the higher rank symmetric spaces. As a consequence, we prove that the locally finite simplicial volume of most $\bq$-rank $1$ locally symmetric spaces is positive, which has been open for many years. Finally we improve the degree theorem for $\bq$-rank $1$ locally symmetric spaces of Connell and Farb. We also address the issue of surjectivity of the comparison map in real rank 2 case.
\end{abstract}


\footnotetext[2]{\sl{The first author was supported by the Basic Science Research Program through the National Research Foundation of Korea(NRF) funded by the Ministry of Education, Science and Technology(NRF-2012R1A1A2040663). The second author gratefully acknowledges the partial support of grant(NRF-2014R1A2A2A01005574) and a warm support of MSRI
during his stay.}}



   \thanks{}
   \thanks{}

   \dedicatory{}

   \date{}


   \maketitle



\section{Introduction}

Besson, Courtois and Gallot \cite{BCG91}  developed the barycenter method in geometry, which was first constructed by Daudy and Earle \cite{DE} for the hyperbolic plane. Their pioneering work on the barycenter method  enabled them to prove a number of long standing problems, in particular, the minimal entropy rigidity conjecture for compact locally symmetric spaces of rank $1$ and the minimal volume conjecture for compact real hyperbolic spaces. The barycenter method was subsequently extended to higher rank locally symmetric spaces by Connell and Farb \cite{CF03} in proving the degree theorem, which is the generalization of Gromov's volume comparison theorem for compact locally symmetric spaces of higher rank.

One of the important applications of the barycenter method in geometry is the proof of Gromov's conjecture that the simplicial volume of a compact locally symmetric space is strictly positive. The proof was given by Lafont and Schmidt \cite{LS06}.
The simplicial volume of a connected oriented manifold $M$, denoted by $\|M\|$, is an invariant associated to $M$ and depends only on the topology of $M$. Gromov \cite{Gr82} first introduced the simplicial volume in studying the minimal normalized entropy and the minimal volume of a differentiable manifold. He verified that both the minimal normalized entropy and the minimal volume of a differentiable manifold are bounded below by its simplicial volume up to positive constant. For this reason the question was naturally raised as to which manifolds have positive simplicial volumes and Gromov's conjecture has been an important issue in the study of the minimal normalized entropy and the minimal volume of a compact locally symmetric space of noncompact type.

Lafont and Schmidt \cite{LS06} used the barycenter method to define a new straightening procedure with universally bounded volume in top dimension that works for locally symmetric spaces of noncompact type with no local direct factors locally isometric to $\mathbb H^2$ or $\rsl_3 \mathbb R / \rso_3 \mathbb R$. This kind of straightening procedure is essential to Thurston's argument in proving that the simplicial volume of a compact real hyperbolic space is positive. In \cite{CF03}, the uniform upper bound for the volumes of top-dimensional straightened simplices follows from the uniform bound for the Jacobian of the natural map. Together with the work of Thurston \cite{Th} in the case of $\mathbb H^2$ and of Bucher \cite{Bu} for $\rsl_3 \mathbb R / \rso_3 \mathbb R$, Lafont and Schmidt \cite{LS06} could prove Gromov's conjecture for all compact locally symmetric spaces of noncompact type.

Besson, Courtois and Gallot \cite{BCG99} analyzed the $k$-Jacobian of the natural map for negatively curved targets as well as its top dimensional Jacobian. Their result implies that in  rank $1$ case, the volumes of  straightened $k$-simplices in the new straightening procedure given by Lafont and Schmidt are uniformly bounded from above for $k\geq 3$.
The same question also arises naturally in the higher rank case. However there has been no advance in analyzing the $k$-Jacobian of the natural map in higher rank case since the work of Connell and Farb \cite{CF03}. This analysis is strongly related to Dupont's question in  \cite{Du78} of whether for a connected semisimple Lie group, every continuous cohomology class is bounded or not. This question {is} known as \emph{Dupont's conjecture}.
The question is still open and considered one of the major problems in the theory of the continuous bounded cohomology of Lie groups. Once it is proved that the $k$-Jacobian is uniformly bounded, it follows that every continuous cohomology class in degree $k$ is bounded. One aim of the present paper is to provide a way to analyze the $k$-Jacobian of the natural map in the higher rank case. We will give a detailed analysis for the codimension $1$ Jacobian and its applications. As far as we know, this is the first estimate for the $k$-Jacobian in  higher rank case.

\begin{theorem}\label{thm:codim1Jac}
Let $X$ be a symmetric space of noncompact type with no direct factors isometric to $\mathbb R$, $\mathbb H^2$, $\mathbb H^3$, $\mathrm{SL}_3 \mathbb R / \rso(3)$, $\mathrm{SL}_4 \mathbb R / \rso(4)$ or $\mathrm{Sp}_4 \mathbb R /\mathrm U(2)$. Then there exists a constant $C>0$ depending only on $X$ and the chosen Riemannian metric on $\Delta^{\dim(X)-1}$ such that for any $f \in C^0(\Delta^{\dim(X)-1},X)$ there exists a uniform upper bound on the $(\dim(X)-1)$-Jacobian of $st_{\dim(X)-1}(f)$: $$|\mathrm{Jac}(st_{\dim(X)-1}(f))(\delta) | \leq C,$$
where $\delta \in  \Delta^{\dim(X)-1}$ is arbitrary and the Jacobian is computed relative to the fixed Riemannian metric on $\Delta^{\dim(X)-1}$.
\end{theorem}

Let $G$ be a connected semisimple Lie group of noncompact type and $X$ the associated symmetric space of dimension $n$.
Then it is well known, and due to van Est \cite{Van} that the continuous cohomology $H^*_c(G,\mathbb R)$ of $G$ is isomorphic to the set of $G$-invariant differential forms $\Omega^*(X)^G$ on $X$.
Dupont \cite{Du78} gave an explicit description of the van Est isomorphism $\mathcal J : \Omega^*(X)^G \rightarrow H_c^*(G,\mathbb R)$ on the level of cochains as follows.
For a $G$-invariant $k$-form $\omega \in \Omega^k(X)^G$, a $k$-cocycle representing $\mathcal J(\omega) \in H^k_c(G,\mathbb R)$ is given by \begin{eqnarray}\label{dupont} \mathcal J(\omega)(g_0,\ldots, g_k) = \int_{[g_0,\ldots,g_k]} \omega \end{eqnarray}
where $[g_0,\ldots,g_k]$ is the geodesic simplex with ordered vertices $g_0 x,\ldots,g_k x$ for $x\in X$.
Dupont provided these explicit cocycles as candidates for bounded cocycles.
In fact it can be easily seen in rank $1$ case that the cocycle $\mathcal J(\omega)$ is bounded if $k\geq 2$ since the volumes of geodesic simplices of dimension at least $2$ are uniformly bounded from above.
Dupont proved the degree $2$ case of his conjecture in \cite{Du78}. Hartnick and Ott then proved it for certain semisimple Lie groups, including Hermitian semisimple Lie groups, in \cite{HO11}. Note that Theorem \ref{thm:codim1Jac} implies the following corollary immediately.

\begin{corollary}\label{cor1.2}
Let $X$ be a symmetric space of noncompact type with no direct factors isometric to $\mathbb R$, $\mathbb H^2$, $\mathbb H^3$, $\mathrm{SL}_3 \mathbb R / \rso(3)$, $\mathrm{SL}_4 \mathbb R / \rso(4)$ or $\mathrm{Sp}_4 \mathbb R /\mathrm U(2)$. Then the volumes of the barycentrically straightened simplices of codimension $1$ are uniformly bounded from above.
\end{corollary}

Since every $G$-invariant differential form on $X$ is closed and bounded, it can be seen that if we use ``barycentrically" straightened simplices instead of geodesic simplices to define the continuous cocycle $\mathcal J(\omega)$, then $\mathcal J(\omega)$ is actually bounded in degree $\dim(X)-1$. Hence it follows that the comparison map of $G$ is surjective in degree $\dim(X)-1$, which gives an affirmative answer for Dupont's conjecture in degree $\dim(X)-1$. Unfortunately it is known that the continuous cohomology of $G$ in degree $\dim(X)-1$ is always trivial. Hence Corollary \ref{cor1.2} does not give anything new in Dupont's conjecture. However we can show the surjectivity of comparison map in a more general context as follows.

\begin{theorem}\label{codim1bounded}
Let $G$ be a connected noncompact semisimple Lie group with no direct factors isogenous to $\mathbb R$, $\rsl_2 \mathbb R$, $\rsl_2 \mathbb C$, $\mathrm{SL}_3 \mathbb R $, $\mathrm{SL}_4 \mathbb R$ or $\mathrm{Sp}_4 \mathbb R$. Let $X$ be the associated symmetric space and $M$ be a complete locally symmetric space covered by $X$. Then the comparison map $$H^*_{cpt,b}(M,\mathbb R)\rightarrow H^*_{cpt}(M,\mathbb R)$$ is surjective in degree $\dim(X)-1$
where $H^*_{cpt}(M,\mathbb R)$ denotes the compactly supported cohomology of $M$.
\end{theorem}

When $\Gamma$ is a cocompact lattice in a connected semisimple Lie group of rank $1$, it is well known that the comparison map $H^d_b (\Gamma,\mathbb R) \rightarrow H^d(\Gamma,\mathbb R)$ is surjective for any $d\geq 2$. This is due to the fact that the geodesic simplices in the corresponding symmetric space have uniformly bounded volume. In a more general setting, Mineyev \cite{Min1} proved that if $\Gamma$ is a nonelementary Gromov-hyperbolic group, then the comparison map $H^d_b (\Gamma,E) \rightarrow H^d(\Gamma,E)$ is surjective for every $d\geq 2$ and every Banach $\Gamma$-module $E$. Furthermore, it turns out that this surjectivity characterizes nonelementary Gromov-hyperbolic groups \cite{Min2}.
In higher rank case,  Lafont suggested us a conjecture on this problem as follows.

\begin{conjecture}
Let $G$ be a connected noncompact simple Lie group and $X$ an irreducible symmetric space. Given a geodesic $\gamma$ in $X$, let $F(\gamma)$ denote the union of all geodesics in $X$ that are parallel to $\gamma$. Let $t_X = \max_{\gamma} \dim(F(\gamma))$ where the maximum is taken over all geodesics $\gamma$ in $X$. Let $\Gamma$ be a cocompact lattice in $G$.
Then the comparison map $$H^d_b (\Gamma,\mathbb R) \rightarrow H^d(\Gamma,\mathbb R)$$ is surjective in $d \geq t_X+1$.
\end{conjecture}

According to \cite[Proposition 2.11.4]{Eb}, for any geodesic $\gamma$ in $X$, the set $F(\gamma)$ is a complete, totally geodesic submanifold of $X$ which is isometric to $\mathbb R^k \times Y$ where $Y$ is a symmetric space of noncompact type.
Note that this conjecture is true for rank $1$ simple Lie groups since $t_X=1$ if $X$ is an irreducible symmetric space of rank $1$. We study further the Jacobian of straightened $d$-simplices for $d \geq t_X+1$ in rank $2$ case, with applications to the above conjecture  in mind.

\begin{theorem}\label{rank2jacobian}
Let $X$ be an irreducible symmetric space of noncompact type of rank $2$.
Let $t_X$ be the maximal dimension of a totally geodesic subspace of $X$ which splits isometrically as a non-trivial product.
Then for any $d \geq \max \{6, t_X+2 \}$, there exists a constant $C>0$ depending only on $X$ and the chosen Riemannian metric on $\Delta^d$ such that for any $f \in C^0(\Delta^d,X)$ there exist a uniform upper bound on the $d$-Jacobian of $st_d(f)$: $$|\mathrm{Jac}(st_d(f))(\delta) | \leq C,$$
where $\delta \in  \Delta^d$ is arbitrary and the Jacobian is computed relative to the fixed Riemannian metric on $\Delta^d$.
\end{theorem}

Theorem \ref{rank2jacobian} implies

\begin{theorem}\label{rank2thm2}
Let $G$, $X$, $t_X$ be as in Theorem \ref{rank2jacobian}. Let $M$ be a complete locally symmetric space covered by $X$.
Then the comparison map $$H^d_{c,b}(G,\mathbb R)\rightarrow H^d_c(G,\mathbb R)$$ is surjective in all degrees $d \geq \max \{6, t_X+2\}$. Furthermore, the comparison map $$H^d_{cpt,b} (M,\mathbb R) \rightarrow H^d_{cpt}(M,\mathbb R)$$ is also surjective in all degrees $d \geq \max \{6, t_X+2 \}$.
\end{theorem}

Note that Hartnick and Ott \cite{HO11} did not prove the surjectivity of the comparison map for some Lie groups such as complex simple Lie groups and $\rsl_{n+1}\mathbb R$. In rank $2$ case, Theorem \ref{rank2thm2} gives a partial affirmative answer to Dupont's conjecture even in the case that Hartnick and Ott did not cover.

Applying Theorem \ref{rank2thm2} to a cocompact lattice $\Gamma$ in $G$, it easily follows that the comparison map  $$H^d_b (\Gamma,\mathbb R) \rightarrow H^d(\Gamma,\mathbb R)$$ is surjective in all degrees $d \geq \max \{6, t_X+2 \}$.
Indeed we show that $\max \{6, t_X+2 \} =t_X+2$ in all rank $2$ cases except $\rsl_3 \mathbb R$, $\rsp_4 \mathbb R$ and $G_{2(2)}$. Hence this implies that the conjecture is true except for the only one case $d = t_X+1$ in almost all rank $2$ cases.
As far as we know, this is the first evidence supporting the conjecture. This might be a starting point for studying the surjectivity of the comparison map for $\mathrm{CAT}(0)$ groups.

As another important application of Theorem \ref{thm:codim1Jac}, we show that the simplicial volumes of most $\mathbb Q$-rank $1$ locally symmetric spaces are strictly positive. Together with the compact locally symmetric spaces of noncompact type, the simplicial volumes of noncompact locally symmetric spaces of finite volume and noncompact type have been studied. Thurston \cite{Th} proved that the simplicial volume of a noncompact $\mathbb R$-rank $1$ locally symmetric space of finite volume is positive.
In the higher rank case, since every lattice is arithmetic, the $\mathbb Q$-rank of a locally symmetric space of finite volume and noncompact type is well defined. For instance, the $\mathbb Q$-rank of a compact locally symmetric space is $0$ and
a noncompact locally symmetric space of finite volume has $\mathbb Q$-rank at least $1$.
The $\mathbb Q$-rank seems to be an essential factor in determining the positivity of the simplicial volume of a noncompact locally symmetric space of finite volume. L\"{o}h and Sauer \cite{LS09-1} gave the first examples of noncompact higher rank locally symmetric spaces of finite volume having positive simplicial volumes. In fact the examples were Hilbert modular varieties that are one examples of $\mathbb Q$-rank $1$ locally symmetric spaces. Furthermore Kim and Kim \cite{KK12} proved that every $\mathbb Q$-rank $1$ locally symmetric space covered by a product of $\mathbb R$-rank $1$ symmetric spaces has positive simplicial volume.
In accordance with the results so far obtained, it has been conjectured that every $\mathbb Q$-rank $1$ locally symmetric space has positive simplicial volume. We prove that this conjecture is true for most $\mathbb Q$-rank $1$ locally symmetric spaces.

\begin{theorem}\label{thm:Qrank1}
Let $X$ be a symmetric space of noncompact type with no direct factors isometric to $\mathbb R$, $\mathbb H^2$, $\mathbb H^3$, $\mathrm{SL}_3 \mathbb R / \rso(3)$, $\mathrm{SL}_4 \mathbb R / \rso(4)$ or $\mathrm{Sp}_4 \mathbb R /\mathrm U(2)$. Then the simplicial volume of a $\mathbb Q$-rank $1$ locally symmetric space covered by $X$ is positive.
\end{theorem}

Unlike the $\mathbb Q$-rank $1$ case, there has been no evidence so far that higher $\mathbb Q$-rank locally symmetric spaces have positive simplicial volumes. Indeed L\"{o}h and Sauer \cite{LS09} proved that the simplicial volume of a locally symmetric space of $\mathbb Q$-rank at least $3$ vanishes, based on Gromov's vanishing-finiteness theorem \cite[Corollary A, p.58]{Gr82}. Now $\mathbb Q$-rank $2$ locally symmetric spaces are the only possible class of higher $\mathbb Q$-rank locally symmetric spaces which might have positive simplicial volumes. The $\mathbb Q$-rank $2$ case still remains open.

If we restrict ourself to the irreducible symmetric spaces, we exclude four irreducible symmetric spaces in Theorem \ref{thm:Qrank1}.
Due to Thurston's work, $\mathbb Q$-rank $1$ locally symmetric spaces covered by $\mathbb H^2$ or $\mathbb H^3$ have positive simplicial volumes. Furthermore since the minimal rational parabolic subgroup of $\rsl_3 \mathbb R$ is amenable, it can be seen that the simplicial volume of a $\mathbb Q$-rank $1$ locally symmetric space covered by $\mathrm{SL}_3 \mathbb R / \rso(3)$ is equal to its Lipschitz simplicial volume \cite{LS09-1, BKK, KK12, Michelle} and thus is positive \cite{LS09}.

\begin{theorem}\label{irred}
The simplicial volume of a $\mathbb Q$-rank $1$ locally symmetric space covered by an irreducible symmetric space of noncompact type which is neither $\mathrm{SL}_4 \mathbb R / \rso(4)$ nor $\mathrm{Sp}_4 \mathbb R /\mathrm U(2)$ is positive.
\end{theorem}


Theorem \ref{thm:Qrank1} together with a simple degree argument establishes a kind of Gromov's volume comparison theorem for $\mathbb Q$-rank $1$ locally symmetric spaces.

\begin{theorem}[Degree theorem]\label{thm:degree}
Let $M$ be an $n$-dimensional $\mathbb Q$-rank $1$ locally symmetric space with no local direct factors locally isometric to $\mathbb R$, $\mathbb H^2$, $\mathbb H^3$, $\mathrm{SL}_3 \mathbb R / \rso(3)$, $\mathrm{SL}_4 \mathbb R / \rso(4)$ or $\mathrm{Sp}_4 \mathbb R /\mathrm U(2)$. Then for any complete Riemannian $n$-manifold $N$ of finite volume with Ricci curvature bounded below by $-(n-1)$ and any proper map $f : N\rightarrow M$,  $$\mathrm{deg}(f)\leq C \frac{\mathrm{Vol}(N)}{\mathrm{Vol}(M)},$$ where $C$ depends only on $n$.
\end{theorem}

We remark that Connel and Farb proved the degree theorem for general locally symmetric spaces of finite volume in \cite{CF03}, but only by assuming that $f$ is a coarse Lipschitz map. Note that we do not assume the Lipschitz condition on $f$ in Theorem \ref{thm:degree}. In a sense, $\mathbb Q$-rank $1$ locally symmetric spaces are the first class of noncompact higher rank locally symmetric spaces of finite volume for which Gromov's volume comparison theorem holds without a Lipschitz condition.
This is one advantage of the simplicial volume.

Corollary \ref{cor1.2} plays an important role in the proof of Theorem \ref{thm:Qrank1}. In fact this is our first  motivation for the estimation of the codimension $1$ Jacobian of straightened simplices. Roughly speaking, the key idea of the proof is to use the rational Tits building and the Tits compactification of a $\mathbb Q$-rank $1$ locally symmetric space. Let $\Gamma$ be a $\mathbb Q$-rank $1$ lattice in a semisimple Lie group $G$ and $X$ the associated symmetric space. Let $\| \Gamma \backslash X \|_\mathrm{lf}$ denote the simplicial volume of $\Gamma \backslash X$.
As the first step, we observe that
$$\left\| \borelserre,\partial \borelserre \right\|(1) \leq \left\| \Gamma \backslash X \right\|_\mathrm{lf}$$
where $\borelserre$ denotes the Borel--Serre compactification of $\Gamma \backslash X$ and $\| \cdot \|(1)$ is the seminorm on relative homology that was defined by Gromov in \cite[Section 4.1]{Gr82}. Hence in order to show $\| \Gamma \backslash X \|_\mathrm{lf}>0$, it suffices to prove that $$\left\| \borelserre,\partial \borelserre \right\| (1) >0.$$

Note that $\Gamma \backslash X$ is the interior of the compact manifold $\borelserre$.
To each relative fundamental cycle $c$ of $\borelserre$, we associate a fundamental cycle $\hat c$ of the Tits compactification $\tits$ of $\Gamma \backslash X$ with the same norm $\|c\|(1)$.
This is possible only for $\mathbb Q$-rank $1$ locally symmetric spaces, which is the key property distinguishing them from higher $\mathbb Q$-rank locally symmetric spaces. Then we can straighten $\hat c$ by using the barycenter method as Lafont and Schmidt did. Since $\hat c$ has ideal simplices with only one ideal point, we need to extend the straightening procedure to ideal simplices. This can be solved by taking the geodesic cone over the straightened simplices of codimension $1$ with ideal top point.
Here we need a uniform upper bound on the volumes of the straightened simplices of codimension $1$ in order to obtain a uniform bound on the volumes of the straightened simplices occurring in $\hat c$. Then by Thurston's argument, the norm of $\hat c$ will be uniformly bounded below away from $0$. Thus we will finish the proof.

This paper is organized as follows. We will start by recalling how to construct the straightening procedure via the barycenter method. Assuming the validity of Theorem \ref{thm:codim1Jac}, we will give a proof of Theorems \ref{thm:Qrank1}, \ref{irred} and \ref{thm:degree}. Then the rest of the paper will be dedicated to proving Theorems \ref{thm:codim1Jac}, \ref{codim1bounded}, \ref{rank2jacobian} and \ref{rank2thm2}. In particular, for a proof of Theorem \ref{thm:codim1Jac}, we will first give a proof  in the case of $\rsl_{n+1} \mathbb R$ and then the other cases will be treated in the Appendix.

\section{Barycentric straightening}\label{barycentric}

Let $G$ be a noncompact semisimple Lie group and $X$ the associated symmetric space of noncompact type of dimension $n$.
Lafont and Schmidt \cite{LS06} introduced a new straightening $$st_* : S_*(X)\rightarrow S_*(X)$$ associated with a closed locally symmetric space covered by $X$ in proving that the simplicial volume of a closed locally symmetric space is strictly positive.
{This} straightening is called \emph{barycentric straightening}.
In fact, their idea is based on the barycenter method first developed by Besson, Courtois and Gallot \cite{BCG91} in order to prove the minimal entropy rigidity conjecture for rank $1$ locally symmetric spaces and then extended to higher rank symmetric spaces by Connell and Farb \cite{CF03}. We begin with a quick review of the barycentric straightening. The main reference for this is \cite{LS06}.

Recall that a singular $k$-simplex in $X$ is a continuous map $\sigma : \Delta^k \rightarrow X$, where $\Delta^k$ is the standard Euclidean $k$-simplex defined by \begin{center} $\Delta^k= \left\{(a_1,\ldots,a_{k+1}) \in \mathbb R^{k+1} \ \big| \ \sum_{i=1}^{k+1} a_i=1 \text{ and } a_i \geq 0 \right\}.$ \end{center}
For convenience, we use the spherical $k$-simplex $\Delta^k_s$ instead of $\Delta^k$, which is defined by \begin{center} $\Delta^k_s = \left\{(a_1,\ldots,a_{k+1}) \in \mathbb R^{k+1} \ \big| \ \sum_{i=1}^{k+1} a_i^2=1 \text{ and } a_i \geq 0 \right\}$,\end{center} equipped with the Riemannian metric induced from $\mathbb R^{k+1}$. Let $\{e_1,\ldots,e_{k+1}\}$ denote the standard basis vectors for $\mathbb R^{k+1}$.
Given a lattice $\Gamma$ in $G$, a singular $k$-simplex $\sigma :\Delta^k_s \rightarrow X$ is straightened to a simplex $st_k(\sigma) : \Delta^k_s \rightarrow X$ by the following procedures: First associate to each singular simplex $\sigma :\Delta^k_s \rightarrow X$ its ordered vertex set $V=(\sigma(e_1),\ldots,\sigma(e_{k+1}) )$. Let $\partial X$ be the visual boundary of $X$ and  $\mathcal M(\partial X)$ be the space of atomless probability measures on $\partial X$.
Albuquerque \cite{Al99} constructed an $h(g_0)$-conformal density $\nu : X \rightarrow \mathcal M(\partial X)$ associated with $\Gamma$ by generalizing the construction of Patterson-Sullivan measure where $h(g_0)$ is the volume entropy of $X$ with respect to the symmetric metric $g_0$ on $X$.
Note that $\nu$ is $\Gamma$-equivariant and each measure $\nu(x)$ is a probability measure fully supported on $G\cdot b^+(\infty)$
where $b$ denotes the barycenter of a positive Weyl chamber (see \cite[Section 2.3]{CF03} for more details) and
$b^+(\infty)$ is the boundary point of $\partial X$ determined by the barycenter $b$.
The Furstenberg boundary $\partial_F X=G/P$ is identified with the orbit $G\cdot b^+(\infty)$ where $P$ is a minimal parabolic subgroup of $G$. Hence we set $\partial_F X := G\cdot b^+(\infty)$ throughout the present paper. For more details, see \cite{Al99}.

By using this conformal density, it is possible to embed the spherical $k$-simplex $\Delta^k_s$ into $\mathcal M(\partial X)$. More precisely, define a map $\widehat V : \Delta^k_s \rightarrow \mathcal M (\partial X)$ by $$\widehat V \left( \sum_{i=1}^{k+1} a_i e_i \right) = \sum_{i=1}^{k+1} a_i^2 \nu(\sigma(e_i)).$$
In the definition of $\widehat V$ above, only the vertex set $V$ of $\sigma$ is involved. Hence the embedding $\widehat V$ depends only on the vertex set of a given singular simplex.

Now we recall the concept of the barycenter of a measure in $\mathcal M(\partial X)$.
Given a measure $\mu \in \mathcal M (\partial X)$, define a function $g_\mu : X \rightarrow \mathbb R$ by
\begin{eqnarray}\label{eqn:g} g_\mu (x) = \int_{\partial X} B(p,x,\theta )d\mu(\theta) \end{eqnarray}
where $B(p,x,\theta)$ is the Busemann function on $X$. For points $p, x \in X$ and $\theta \in \partial X$, the Busemann function $B:X \times X \times \partial X \rightarrow \mathbb R$ is defined by
$$B(p,x,\theta)=\lim_{t \rightarrow \infty} (d(x, l_\theta (t))-t),$$
where $l_\theta$ is the unique geodesic ray starting at $p$ with endpoint $\theta$.
The \emph{barycenter} $bar(\mu)$ of $\mu$ is defined to be the unique point of $X$ where $g_\mu$ is minimized. From the properties of the Busemann function, it can be easily seen that $g_\mu$ changes by an additive constant when one changes the base point $p$. Thus the barycenter is independent of the choice of basepoints.
The barycenter is not always defined for all measures of $\mathcal M (\partial X)$. However, Connell and Farb \cite{CF03} showed that the barycenter is well defined for a finite weighted sum of Patterson-Sullivan measures on $\partial_F X$. Since all measures in the image of $\widehat V$ are such measures, the following procedures make sense.
$$ \xymatrixcolsep{2pc}\xymatrix{
\Delta^k_s \ar[r]^-{\widehat V} &
\mathcal M (\partial X) \ar[r]^-{bar} & X.
}$$

\begin{definition}[Barycentric straightening]
Given a singular $k$-simplex $\sigma \in C^0(\Delta^k_s, X)$, with corresponding vertex set $V$, define $st_k(\sigma) \in C^0(\Delta^k_s, X)$ by $st_k(\sigma)(\delta)=bar(\widehat V(\delta))$ for $\delta \in \Delta^k_s$.
\end{definition}

A simplex of the image of $st_k$ is called a \emph{straight simplex} in $X$. For a $k$-simplex $\sigma$, note that $$st_k(\sigma)(e_i)=bar(\nu(\sigma(e_i)))=\sigma(e_i)$$ for all $i=1,\ldots,k+1$. In other words, the vertex set remains fixed pointwise under the barycentric straightening.
Indeed $st_k(\sigma)$ depends only on the vertex set $V$ and hence we use the notation $st_V(\delta):=st_k(\sigma)(\delta)$ for convenience.
Lafont and Schmidt \cite{LS06} proved that if $X$ is a symmetric space of noncompact type with no direct factors isometric to $\mathbb H^2$ or $\mathrm{SL}_3\mathbb R / \rso(3)$, the collection of maps $st_* : S_*(X)\rightarrow S_*(X)$ satisfies the following properties:
\begin{itemize}
\item[(a)] the maps $st_k$ are $\Gamma$-equivariant,
\item[(b)] the maps $st_k$ induce a chain map $st_* : C_*(X,\mathbb R) \rightarrow C_*(X,\mathbb R)$ which is $\Gamma$-equivariantly chain homotopic to the identity,
\item[(c)] the image of $st_n$ lies in $C^1(\Delta^n,X)$, i.e., straightened top-dimensional simplices are $C^1$,
\item[(d)] there exists a constant $C>0$, depending only on $X$ and the chosen Riemannian metric on $\Delta^n$, such that for any $\sigma \in S_n(X)$, with corresponding straightened simplex $st_n(\sigma):\Delta^n \rightarrow X$, there is a uniform upper bound on the Jacobian of $st_n(\sigma)$: $$|\mathrm{Jac}(st_n(\sigma))(\delta)|\leq C,$$ where $\delta \in \Delta^n$ is arbitrary and the Jacobian is computed relative to the fixed Riemannian metric on $\Delta^n$.
\end{itemize}

By using the barycentric straightening, Lafont and Schmidt \cite{LS06}  proved that the simplicial volume of any closed locally symmetric space is strictly positive. The key steps in the proof were to verify properties (c) and (d). In fact, properties (c) and (d) directly follow from the detailed analysis in the paper of Connell and Farb \cite[Section 4]{CF03}, which is used to obtain a uniform bound for the Jacobian of the natural map.

\begin{theorem}[Connell-Farb]\label{thm:connell-farb}
Let $X$ be a symmetric space of noncompact type with no direct factors isometric to $\mathbb H^2$ or $\rsl_3 \mathbb R/\rso(3)$. Let $\mu \in \mathcal M (\partial X)$ be a probability measure fully supported on $\partial_F X$ and let $x\in X$. Consider the endomorphisms $K_x(\mu)$ and $H_x(\mu)$ defined on $T_xX$ by
$$\langle K_x(v),v \rangle =\int_{\partial_F X} DdB_{(x,\theta)}(v,v) \ d \mu(\theta)$$ and
$$\langle H_x(v),v \rangle=\int_{\partial_F X} dB^2_{(x,\theta)}(v) \ d\mu(\theta).$$
Then $\det K_x(\mu)>0$ and there is a positive constant $C>0$ depending only on $X$ such that
$$\mathrm{Jac}_x(\mu):=\frac{\det (H_x(\mu))}{\det^2 (K_x(\mu))}\leq C.$$
Furthermore, the constant $C$ is explicitly computable.
\end{theorem}

\begin{remark}
The $\Gamma$-equivariances in properties (a) and (b) come from the $\Gamma$-equivariance of the conformal density $\nu : X \rightarrow \mathcal M (\partial X)$. From \cite[Proposition 7.5]{Al99}, the conformal density $\nu$ is actually $G$-equivariant.
Hence, the barycentric straightening $st_* : C_*(X,\mathbb R) \rightarrow C_*(X,\mathbb R)$ is actually $G$-equivariant and $G$-equivariantly chain homotopic to the identity.
\end{remark}

It follows from property (d) that the volumes of straight $n$-simplices are uniformly bounded from above. The question has been naturally raised as to whether the $k$-dimensional volumes of straight $k$-simplices are uniformly bounded from above or not for $k < n$.
In fact this question is related to the surjectivity of the comparison map in  bounded cohomology theory as follows:
Recall that the continuous cohomology $H^*_c(G,\mathbb R)$ of $G$ is the cohomology of the complex $C^*_c(G,\mathbb R)^G$ endowed with the homogeneous coboundary operator, where  $$C^k_c(G,\mathbb{R})=\{ f : G^{k+1} \rightarrow \mathbb{R}\text{ }|\text{ }f \text{ is continuous} \},$$
and $C^k_c(G,\mathbb R)^G$ denotes the subspace of $G$-invariant cochains. The action of $G$ on $C^k_c(G,\mathbb R)$ is given by $$(g\cdot f)(g_0,\ldots,g_k)=f(g^{-1}g_0,\ldots,g^{-1}g_k).$$
According to the van Est isomorphism \cite[Proposition IX.5.5]{BW}, for a connected semisimple Lie group $G$ with finite center, its continuous cohomology $H^*_c(G,\mathbb{R})$ is isomorphic to the set of $G$-invariant differential forms on the associated symmetric space $X$. In particular,
the continuous cohomology of $G$ in top degree is generated by the $G$-invariant volume form on $X$.

For a cochain $f: G^{k+1} \rightarrow \mathbb{R}$, define its sup norm by
$$\|f\|_\infty = \sup \{ |f(g_0,\ldots,g_k)|\text{ }|\text{ } (g_0,\ldots,g_k)\in G^{k+1}\}.$$
The sup norm turns $C^*_c(G,\mathbb{R})$ into normed real vector spaces. The continuous bounded cohomology $H^*_{c,b}(G,\mathbb{R})$ of $G$ is defined as the cohomology of the subcocomplex $C^*_{c,b}(G,\mathbb{R})^G$ of
$G$-invariant continuous bounded cochains in $C^*_c(G,\mathbb{R})^G$.
The inclusion of $C^*_{c,b}(G,\mathbb{R})^G \subset C^*_c(G,\mathbb{R})^G$ induces a comparison map $ H^*_{c,b}(G,\mathbb{R}) \rightarrow H^*_c(G,\mathbb{R})$.

One of the major problems in the theory of continuous bounded cohomology is the question of whether the comparison map is an isomorphism for any connected semisimple Lie group \cite[Conjecture A]{Mo06}. In particular Dupont \cite[Remark 3]{Du78} conjectured that the comparison map is surjective as mentioned in the introduction.
Existing proofs for the surjectivity fall into two classes, namely, explicit methods and indirect methods. Explicit methods construct explicit bounded cocycles which are obtained by integrating invariant simplices in the corresponding symmetric space. The surjectivity of the comparison map for rank $1$ simple Lie groups can be obtained by this method.
Hartnick and Ott \cite{HO11} used an indirect method, employing Gromov's theorem \cite{Gr82} on the boundedness of the primary characteristic classes of flat bundles, to show that the comparison map is surjective if each of the simple factors of $G$ is either Hermitian or a Lie group presented in \cite[Thereom 1.5]{HO11}.

Once it is proved that the volumes of straight $k$-simplices are uniformly bounded from above, it can be proved that the comparison map $H^k_{c,b}(G,\mathbb R) \rightarrow H^k_c(G,\mathbb R)$ is surjective by using the explicit method as above.
Assuming Theorem \ref{thm:codim1Jac} at this moment, we can give a proof of Theorem \ref{codim1bounded} as follows.

\begin{proof}[Proof of Theorem \ref{codim1bounded}]
First, recall the de Rham isomorphism between de Rham cohomology with compact support and singular cohomology with compact support, which is defined at the cochain level by
$$\Psi : \Omega^*_{cpt}(M,\mathbb R) \rightarrow C^*_{cpt}(M,\mathbb R), \ \Psi(\omega) (\sigma)= \int_{\sigma} \omega$$
where $\omega$ is a $k$-form on $M$ with compact support and $\sigma$ is a singular $k$-simplex in $M$.
In addition, we define another cocycle $\Psi_b$ by $$\Psi_b(\omega) (\sigma)= \Psi(\omega)(st_k(\sigma)).$$
Then since $st_* : S_*(M) \rightarrow S_*(M)$ is chain homotopic to the identity, the cocycle $\Psi_b(\omega)$ represents the same cohomology class as $\Psi(\omega)$.
Furthermore since $\omega$ has compact support and the volume of straightened $(\dim(X)-1)$-simplex is uniformly bounded, it is not difficult to see that $\Psi_b(\omega)$ is actually bounded if $\omega$ is a differential form of degree $\dim(X)-1$. This completes the proof.
\end{proof}

\section{Seminorms on relative homology}

This section is devoted to proving Theorem \ref{thm:Qrank1} assuming the validity of Theorem \ref{thm:codim1Jac}.
As seen in the previous section, Theorem \ref{thm:codim1Jac} implies Theorem \ref{codim1bounded}.
Let $M$ be an oriented connected compact $n$-manifold with boundary. Gromov introduced a one parameter family of seminorms on the real singular relative homology $H_*(M,\partial M)$ in \cite[Section 4.1]{Gr82}. Let $\| \cdot \|$ denote the obvious $\ell^1$-norm on the real singular chain complex $C_*(M)$ of $M$. For every $\theta \geq 0$, Gromov defined a norm $\| \cdot \|(\theta)$ on $C_*(M)$ by putting $$\|c\|(\theta)=\|c\|+\theta \|\partial c \|.$$
Then, the norm $\| c' \|(\theta)$ of $c'\in C_*(M,\partial M)$ is defined by $$\|c' \|(\theta)=\inf_{c} \|c \|(\theta)$$ where the infimum is taken over all singular chains in $C_*(M)$ representing $c'$. This norm induces a seminorm on $H_*(M,\partial M)$, which is still denoted by $\|\cdot \|(\theta)$. From the definition, it immediately follows that $$ \| \cdot \| (\theta_1) \geq \| \cdot \| (\theta_2) \text{ for } \theta_1 \geq \theta_2.$$

The usual $\ell^1$-seminorm on $H_*(M,\partial M)$ is equal to $\| \cdot \|(0)$. Hence the relative simplicial volume $\|M,\partial M\|$ of $(M,\partial M)$ is equal to $\|[M,\partial M]\|(0)$ where $[M,\partial M] \in H_n(M,\partial M)$ denotes the relative fundamental class of $(M,\partial M)$. In particular, if the fundamental class of each connected component of $\partial M$ is amenable, all the seminorms $\| \cdot \| (\theta)$ are equal by Gromov's equivalence theorem, \cite[Section 4.1]{Gr82}, cf. also \cite{Michelle}, which follows.
\begin{theorem}[Gromov's equivalence theorem]
Let $(Y,B)$ be a pair of topological spaces. If the fundamental groups of the path-connected components of $B$ are amenable, then the seminorms $\| \cdot \|(\theta)$ on $H_i(Y,B)$, for $i \geq 2$, are equal for every $\theta \in [0,\infty]$.
\end{theorem}

\subsection{Disjoint mapping cone}\label{sec:cone}

We describe the topological meaning of the seminorm $\| \cdot \|(1)$, which will play a key role in proving that the simplicial volumes of most $\mathbb Q$-rank $1$ locally symmetric spaces are strictly positive.

Let $\partial_1 M,\ldots,\partial_s M$ be the connected components of $\partial M$. Recall that for each $i$, the cone of $\partial_i M$, denoted by $Cone(\partial_i M)$, is defined as the quotient space $(\partial_i M \times I) / (\partial_i M \times \{1\})$ of the product of $\partial_i M$ with the unit interval $I = [0, 1]$.
We define the \emph{disjoint mapping cone} of $(M,\partial M)$, denoted by $Dcone(M,\partial M)$, as the space obtained by gluing the $\coprod_{i=1}^s Cone(\partial_i M)$ to $M$ along $\coprod_{i=1}^s\partial_i M$.
From this viewpoint, it can be easily seen that $$H_*(Dcone(M,\partial M)) \cong H_*(M,\partial M)$$ in degree $* \geq 2$.

We now define $cone(\tau) \in S_k(Dcone(M,\partial M))$ for a given singular simplex $\tau \in S_{k-1}(\partial M)$ in the following way. We may assume that $\tau : \Delta^{k-1} \rightarrow \partial_i M$ for some $i$. Then the continuous map $\tau$ can be extended naturally to a continuous map $cone(\tau):Cone(\Delta^{k-1}) \rightarrow Cone(\partial_i M)$ by setting $$cone(\tau)(\delta,t)=(\tau(\delta),t) \text{ for all } \delta \in \Delta^{k-1} \text{ and }t \in [0,1].$$ Noting that $Cone(\Delta^{k-1})$ is homeomorphic to $\Delta^{k}$, one may define a singular $k$-simplex $cone(\tau): \Delta^k \rightarrow Cone(\partial_i M)$. In this definition of $cone(\tau)$, one has used an identification $Cone(\Delta^{k-1}) \cong \Delta^k$. In order that the cone operator $cone(\cdot)$ is compatible with the boundary operator $\partial_*$, we define the identifications $Cone(\Delta^{k-1}) \cong \Delta^k$ as follows: First, identify $\Delta^{k-1} \times \{0\}$ with $\Delta^{k-1}$ by the projection onto the first factor. Then for each $p \in \Delta^{k-1}$, identify the line segment $p \times [0,1]$ in $\Delta^{k-1}\times [0,1]$ with the line segment connecting $p$ and $e_{k+1}=(0,\ldots,0,1)\in \mathbb R^{k+1}$ linearly. This identification gives rise to a continuous map $j_k : \Delta^{k-1} \times [0,1] \rightarrow \Delta^k$ and moreover $j_k$ maps all points of $\Delta^{k-1} \times \{1\}$ to $e_{k+1}$. Hence $j_k$ defines a homeomorphism $j_k : Cone(\Delta^{k-1}) \rightarrow \Delta^k$ for each $k\in \mathbb N$. It is not difficult to see that these identifications are compatible with the boundary operator on $\Delta^*$, that is,
\begin{eqnarray}\label{eqn:coneboundary} \partial_k cone(\tau) = cone (\partial_{k-1} \tau) +(-1)^{k}\tau.\end{eqnarray}

We define the subspace $Cone(S_{k-1}(\partial M)) \subset S_k(Dcone(M,\partial M))$ by
 $$Cone(S_{k-1}(\partial M)) =\{ cone(\tau) \ | \ \tau \in S_{k-1}(\partial M) \}$$
and let $Cone(C_{k-1}(\partial M))$ be the real vector space generated by the basis of singular simplices in $Cone(S_{k-1}(\partial M))$.
We set $$C^{conic}_*(Dcone(M,\partial M))=C_*(M) \oplus Cone(C_{*-1}(\partial M)).$$
Then it is easy to check that $(C^{conic}_*(Dcone(M,\partial M)), \partial_*)$ is a subcomplex of $C_*(Dcone(M,\partial M))$. Denote its homology by $H^{conic}_*(Dcone(M,\partial M))$. Note that $H^{conic}_*(Dcone(M,\partial M))$ is isomorphic to $H_*(Dcone(M,\partial M))$ in degree $*\geq 2$ (see \cite[Lemma 5.2]{KKT} for a proof). Hence there is a fundamental class $[Dcone(M,\partial M)]_{conic}$ in $H_n^{conic}(Dcone(M,\partial M))$.

The usual $\ell^1$-norm is defined on $C^{conic}_*(Dcone(M,\partial M))$ and induces a seminorm on $H^{conic}_*(Dcone(M,\partial M))$. We will use the same notation $\| \cdot \|$ for both norms. For each chain $z \in C^{conic}_*(Dcone(M,\partial M))$, denote by $z_M$ and $z_{cone}$ the projections of $z$ onto $C_*(M)$ and $Cone(C_{*-1}(\partial M))$ respectively.

\begin{lemma}\label{lem:ell1norm}
Let $z$ be a fundamental cycle in $C^{conic}_n(Dcone(M,\partial M))$.
Then, $$\| z \| = \|z_M \|(1).$$
\end{lemma}
\begin{proof}
Since $z$ is a cycle, it follows that $\partial z =\partial z_M + \partial z_{cone}=0$. This implies that $\partial z_{cone}=-\partial z_M$. Noting that $z_M$ is a relative fundamental cycle of $(M,\partial M)$, it follows that $\partial z_M$ is a fundamental cycle of $\partial M$ and so is $\partial z_{cone}$ up to sign. On the other hand, $z_{cone}$ is uniquely determined by its boundary $\partial z_{cone}$, that is, $z_{cone}=(-1)^n cone(\partial z_{cone})$. Hence,  $$\| z_{cone}\|=\|cone(\partial z_{cone}) \| =\| \partial z_{cone}\|.$$
Thus we have that $$z=z_M+(-1)^n cone(\partial z_{cone})=z_M+(-1)^{n+1} cone(\partial z_M),$$
and $$\|z\|=\|z_M\|+\|z_{cone}\|=\|z_M\|+\|\partial z_{cone}\|=\|z_M\|+\|\partial z_M\|=\|z_M\|(1),$$
which completes the proof.
\end{proof}

As a corollary, $\|M,\partial M\|(1)$ can be viewed as the $\ell^1$-seminorm of the fundamental class $[Dcone(M,\partial M)]_{conic}$ on $H^{conic}_n(Dcone(M,\partial M))$.

\begin{corollary}\label{cor:conicnorm}
Let $M$ be an oriented connected compact manifold with boundary. Then $$\| [Dcone(M,\partial M)]_{conic}\| = \| M, \partial M \|(1) \geq \|M,\partial M\| .$$
Furthermore, the equality holds if the fundamental groups of the connected components of $\partial M$ are amenable.
\end{corollary}

\begin{proof}
It follows from the proof of Lemma \ref{lem:ell1norm} that $z$ is a fundamental cycle in $C^{conic}_n(Dcone(M,\partial M))$ if and only if $z$ is of the form $$z=w+(-1)^{n+1}cone(\partial w)$$ for a relative fundamental cycle $w$ of $(M,\partial M)$. This implies the first equality in the corollary. The second inequality is clear and the sufficient condition for the equality comes from Gromov's equivalence theorem in \cite[Section 4.1]{Gr82}
\end{proof}

\subsection{Simplicial volume of a noncompact manifold}

Let $N$ be a topological space and let $S_k(N)$ be the set of all continuous maps from the standard $k$-simplex $\Delta^k$ to $N$. A subset $A$ of $S_k(N)$ is called \emph{locally finite} if each compact subset of $N$ intersects only finitely many elements of $A$. Let us denote by $S^\mathrm{lf}_k(N)$ the set of all locally finite subsets of $S_k(N)$.

\begin{definition}
Let $N$ be a topological space and let $k\in\mathbb{N}$. The \emph{locally finite chain complex} of $N$ is the chain complex $C^\mathrm{lf}_*(N)$ consisting of the real vector spaces $$ C^{\mathrm{lf}}_k(N)=\left\{ \sum_{\sigma \in A} a_\sigma\cdot\sigma \ \bigg| \ A\in S^{\mathrm{lf}}_k(N) \text{ and }a_\sigma \in \mathbb{R}, {\sigma \in A} \right\}$$
equipped with the boundary operator given by the alternating sums of the $(k-1)$-faces. \emph{The locally finite homology} $H^\mathrm{lf}_*(N)$ of $N$ is the homology of the locally finite chain complex $C^\mathrm{lf}_*(N)$.
\end{definition}

The $\ell^1$-norm $\| \cdot \|$ on the locally finite chain complex of $N$ is defined using the canonical basis of singular simplices. Note that the $\ell^1$-norm of a locally finite chain can be infinite. This $\ell^1$-norm gives rise to a semi-norm on the locally finite homology of $N$. It is well known that any oriented connected manifold $N$ possesses a fundamental class, which is a distinguished generator of the top dimensional locally finite homology $H^{\mathrm{lf}}_n(N;\mathbb{Z})\cong\mathbb{Z}$ with integral coefficients \cite{LS07}. The existence of a fundamental class of a noncompact oriented connected manifold is one advantage of locally finite homology against singular homology.

\begin{definition}
Let $N$ be an oriented connected $n$-dimensional manifold without boundary. Then \emph{the simplicial volume of $N$} is defined as $$\| N\|_\mathrm{lf}=\inf \left\{ \| z \| \ \mid \ z \in C^{\mathrm{lf}}_n(N)\text{ is a fundamental cycle of } N \right\}.$$
\end{definition}

Note that $\| N\|_\mathrm{lf}$ might be infinite. For a tame manifold $N$,  L\"{o}h gives a necessary and sufficient condition for $\| N\|_\mathrm{lf}$ to be finite in terms of $\ell^1$-homology. See \cite[Theorem 6.4]{Loh08} for further details.

\begin{lemma}\label{lem:keyinequal}
Let $M$ be an oriented connected compact manifold with boundary and $\mathring{M}$ denote the interior of $M$. If the simplicial volume of $\mathring{M}$ is finite, then
$$\| M,\partial M \|(1) \leq \|\mathring{M}\|_\mathrm{lf}.$$
\end{lemma}
\begin{proof}
Choose an exhausting sequence $(M_k)_{k\in \mathbb N}$ of compact cores of $\mathring{M}$. Each $M_k$ is homeomorphic to $M$ and moreover, $M_k$ is a deformation retract of $\mathring{M}$. Let $z=\sum_{i=1}^\infty a_i \sigma_i$ be a locally finite fundamental cycle of $\mathring{M}$ with finite $\ell^1$-norm. Set $z_k = \sum_{\mathrm{Im}(\sigma_i) \cap M_k \neq \emptyset} a_i \sigma_i$. Then it is verified in  \cite[Theorem 5.4]{LS07} that $z_k$ is a relative fundamental cycle representing $[M,M-M_k]$. Hence for each retraction $r_k : \mathring{M} \rightarrow M_k$, $(r_k)_*z_k$ is a relative fundamental cycle of $(M_k,\partial M_k)$ and we have
\begin{eqnarray}\label{eqn:A} \| (r_k)_*z_k \| \leq \|z_k\|.\end{eqnarray}
Noting that any simplex occurring in $\partial ((r_k)_*z_k)$ comes from simplices occurring in $\partial \sigma_i$ for $\sigma_i$ with $\mathrm{Im}(\sigma_i) \cap \partial M_k \neq \emptyset$, it follows that
\begin{eqnarray}\label{eqn:B} \| (r_k)_*z_k \|(1) \leq  \| (r_k)_*z_k \| + (n+1)\left\| \sum_{\mathrm{Im}(\sigma_i)\cap \partial M_k \neq \emptyset}a_i \sigma_i \right\|. \end{eqnarray}

By combining inequalities (\ref{eqn:A}) and (\ref{eqn:B}), we have
\begin{eqnarray}\label{eqn:C} \| M_k, \partial M_k \|(1) \leq \| z_k \| + (n+1)\sum_{\mathrm{Im}(\sigma_i)\cap \partial M_k \neq \emptyset} |a_i|.\end{eqnarray}
Since every $M_k$ is homeomorphic to $M$, the value $\| M_k, \partial M_k \|(1)$ is equal to $\| M, \partial M \|(1)$ for every $k \in \mathbb N$. Since the $\ell^1$-norm of $z$ is finite and $(M_k)_{k\in \mathbb N}$ is an exhausting sequence of $\mathring{M}$, it follows that
$$\lim_{k\rightarrow \infty} \|z_k\|=\|z\| \text{ and } \lim_{k\rightarrow \infty} \sum_{\mathrm{Im}(\sigma_i)\cap \partial M_k \neq \emptyset} |a_i|=0.$$
Hence, taking the limit of inequality (\ref{eqn:C}) as $k$ goes to infinity, we have
\begin{eqnarray}\label{eqn:D} \| M,\partial M \|(1) \leq \|z\|.\end{eqnarray}
Inequality (\ref{eqn:D}) holds for any locally finite fundamental cycle $z$ with finite $\ell^1$-norm. Thus taking the infimum over all locally finite fundamental cycles with finite $\ell^1$-norms, we finally get
$$\| M,\partial M \|(1) \leq \|\mathring{M}\|_\mathrm{lf},$$
which completes the proof.
\end{proof}

Let $\Gamma$ be a $\mathbb Q$-rank $1$ lattice of a semisimple $\mathbb Q$-group $\mathbf G$ with trivial center and noncompact factors (see the definition and properties of $\mathbb Q$-rank $1$ lattices in the next section). Let $X=\mathbf G (\mathbb R)/K$ be the associated symmetric space where $K$ is a maximal compact subgroup of $\mathbf G(\mathbb R)$. Let $\borelserre$ denote the Borel--Serre compactification of $\Gamma \backslash X$ (see \cite[Chapter III.5]{BJ06} for more details of the Borel--Serre compactification). Then $\borelserre$ is a compact manifold with boundary and $\Gamma \backslash X$ is homeomorphic to the interior of $\borelserre$. Applying Lemma \ref{lem:keyinequal} to this situation, we have
$$\| \borelserre,\partial \borelserre \|(1) \leq \|\Gamma \backslash X \|_\mathrm{lf}.$$
Hence it is sufficient to prove $\| \borelserre,\partial \borelserre \|(1) >0$ in order to show that the simplicial volume of $\Gamma \backslash X$ is positive.

\section{$\mathbb Q$-rank $1$ lattices and Tits compactification}

Before giving a proof of $\| \borelserre,\partial \borelserre \|(1) >0$, we first collect some definitions and results about $\mathbb Q$-rank $1$ locally symmetric spaces in this section.
Let $G$ be a noncompact, semisimple Lie group with trivial center and no compact factors.
Then one may define arithmetic lattices in the following way.

\begin{definition}\label{def:arithmetic}
A lattice $\Gamma$ in $G$ is called \emph{arithmetic} if there are
\begin{itemize}
\item[(1)] a semisimple algebraic group $\mathbf{G} \subset \mathrm{GL}_n \mathbb{C}$ defined over $\bq$ and
\item[(2)] an isomorphism $\varphi : \mathbf{G}(\br)^0 \ra G$
\end{itemize}
such that $\varphi(\mathbf{G}(\mathbb{Z}) \cap \mathbf{G}(\br)^0)$ and $\Gamma$ are commensurable.
\end{definition}

It is a well-known result of Margulis \cite{Ma91} that all irreducible
lattices in higher rank Lie groups are arithmetic. The $\bq$-rank
of a semisimple algebraic group $\mathbf{G}$ is defined as the
dimension of a maximal $\bq$-split torus of $\mathbf{G}$. For an
arithmetic lattice $\Gamma$ in $G$, $\bq$-rank$(\Gamma)$ is
defined by the $\bq$-rank of $\mathbf{G}$ where $\mathbf{G}$ is an
algebraic group as in Definition \ref{def:arithmetic}.

A closed subgroup $\mathbf{P} \subset \mathbf{G}$ defined over
$\bq$ is called a \emph{rational parabolic subgroup} if $\mathbf{P}$
contains a maximal connected solvable subgroup of $\mathbf{G}$.
For any rational parabolic subgroup $\mathbf{P}$ of $\mathbf{G}$,
one obtains the \emph{rational Langlands decomposition} of
$P=\mathbf{P}(\br)$:
$$P = N_{\mathbf{P}} \times A_\mathbf{P} \times M_\mathbf{P},$$
where $N_\mathbf{P}$ is the real locus of the unipotent radical
$\mathbf{N}_\mathbf{P}$ of $\mathbf{P}$, $A_\mathbf{P}$ is a
stable lift of the identity component of the real locus of the
maximal $\bq$-torus in the Levi quotient $\mathbf{P} /
\mathbf{N}_\mathbf{P}$, and $M_\mathbf{P}$ is a stable lift of the
real locus of the complement of the maximal $\bq$-torus in
$\mathbf{P} / \mathbf{N}_\mathbf{P}$.

Let $X=G/K$ be the associated symmetric space of noncompact type
with $K$ a maximal compact subgroup of $G$. Write
$X_\mathbf{P}=M_\mathbf{P}/(K \cap M_\mathbf{P})$. Let us denote
by $\tau :  M_\mathbf{P} \ra X_\mathbf{P}$ the canonical
projection. Fix a base point $x_0 \in X$ whose stabilizer group is
$K$. Then we have an analytic diffeomorphism
$$\mu : N_{\mathbf{P}} \times A_\mathbf{P} \times X_\mathbf{P} \ra X, \ (n, a, \tau(m)) \ra nam \cdot x_0,$$
which is called the \emph{rational horocyclic decomposition} of $X$.
For more details, see \cite[Section III.2]{BJ06}.

\subsection{Precise reduction theory}
Let $\mathfrak{g}$ and $\mathfrak{a}_\mathbf{P}$ denote the Lie
algebras of the Lie groups $G$ and $A_\mathbf{P}$ defined above.
Then the adjoint action of $\mathfrak{a}_\mathbf{P}$ on
$\mathfrak{g}$ gives a root space decomposition: $$
\mathfrak{g}=\mathfrak{g}_0 + \sum_{\alpha \in
\Lambda} \mathfrak{g}_\alpha,$$
where $$\mathfrak{g}_\alpha = \{Z \in \mathfrak{g} \ | \
ad(A)(Z)=\alpha(A)Z \text{ for all }A\in \mathfrak{a}_\mathbf{P}
\},$$ and $\Lambda$ consists of
those nontrivial characters $\alpha$ such that
$\mathfrak{g}_\alpha \neq 0$. It is known that
$\Lambda$ is a root system. Fix
an order on $\Lambda$ and
denote by $\Lambda^+$ the
corresponding set of positive roots. Define $$\rho_\mathbf{P} =
\sum_{\alpha \in \Lambda^+}
(\dim \mathfrak{g}_\alpha) \alpha.$$

Let $\Pi$ be the set of
simple positive roots. Since we consider only $\bq$-rank $1$ arithmetic lattices, we
restrict ourselves from now on to the case
$\bq$-rank$(\mathbf{G})=1$. Then the following hold:
\begin{itemize}
\item[(1)] All proper rational parabolic subgroups of $\mathbf{G}$ are minimal.
\item[(2)] For any proper rational parabolic subgroup $\mathbf{P}$ of $\mathbf{G}$, $\dim A_\mathbf{P}=1$.
\item[(3)] The set $\Pi$ of simple positive $\bq$-roots contains only a single element.
\end{itemize}

For any proper rational parabolic subgroup $\mathbf{P}$ of
$\mathbf{G}$ and any $t>1$, define $$A_{\mathbf{P},t}=\{ a\in
A_\mathbf{P} \ | \ \alpha(a)>t \},$$ where $\alpha$ is the unique
root in $\Pi$. For
bounded sets $U\subset N_\mathbf{P}$ and $V\subset X_\mathbf{P}$,
the set $$\mathcal{S}_{\mathbf{P},U,V,t} = U \times
A_{\mathbf{P},t} \times V \subset N_{\mathbf{P}} \times
A_\mathbf{P} \times X_\mathbf{P}$$ is identified with the subset
$\mu(U \times A_{\mathbf{P},t} \times V)$ of $X=G/K$ by the
horospherical decomposition of $X$ and called a \emph{Siegel set}
in $X$ associated with the rational parabolic subgroup
$\mathbf{P}$. Given a $\bq$-rank $1$ lattice $\Gamma$ in $G$, it
is a well known result due to Borel and Harish-Chandra that
there are only finitely many $\Gamma$-conjugacy classes of
rational parabolic subgroups. We next recall the precise reduction theory
in $\bq$-rank $1$ case (see \cite[Proposition III.2.21]{BJ06}).

\begin{theorem}\label{thm:reduction theory}
Let $\Gamma$ be a $\bq$-rank $1$ lattice in $G$. Let $\mathbf{G}$
denote a semisimple algebraic group defined over $\bq$ with
$\bq$-rank$(\mathbf{G})=1$ as in Definition \ref{def:arithmetic}.
Denote by $\mathbf{P}_1,\ldots,\mathbf{P}_s$ representatives of
the $\Gamma$-conjugacy classes of all proper rational parabolic
subgroups of $\mathbf{G}$. Then there exist a bounded set
$\Omega$ in $\Gamma\backslash G/K$ and Siegel sets $U_i \times
A_{\mathbf{P}_i,t_i} \times V_i$, $i=1,\ldots,s$, in $X=G/K$ such
that
\begin{itemize}
\item[(1)] each Siegel set $U_i \times A_{\mathbf{P}_i,t_i}
    \times V_i$ is mapped injectively into $\Gamma \backslash
    X$ by the projection $\pi : X \ra \Gamma\backslash X$,
\item[(2)] the image of $U_i \times V_i$ in $(\Gamma \cap P_i)
    \backslash N_{\mathbf{P}_i} \times X_{\mathbf{P}_i}$ is
    compact,
\item[(3)] $\Gamma \backslash X$ admits the disjoint
    decomposition $$\Gamma \backslash X = \Omega \cup
    \coprod_{i=1}^s \pi(U_i \times A_{\mathbf{P}_i,t_i} \times
    V_i).$$
\end{itemize}
\end{theorem}

Geometrically, $B_\mathbf{P}(t)=\mu(N_\mathbf{P} \times
A_{\mathbf{P},t} \times X_\mathbf{P})$ is a horoball for any
proper minimal rational parabolic subgroup $\mathbf{P}$ of
$\mathbf{G}$. Hence each $\mu(U_i \times A_{\mathbf{P}_i,t_i}
\times V_i)$ is a fundamental domain of the cusp group $\Gamma_i =
\Gamma \cap \mathbf{P}(\br)$ acting on the horoball
$B_{\mathbf{P}_i}(t_i)$ and each $\mu(U_i \times V_i)$ is a
bounded domain in the horosphere that bounds the horoball
$B_{\mathbf{P}_i}(t_i)$. Hence each set $\pi(U_i \times
A_{\mathbf{P}_i,t_i} \times V_i)$ corresponds to a cusp of the
locally symmetric space $\Gamma \backslash X$.

Each Siegel set $U_i \times
A_{\mathbf{P}_i,t_i} \times V_i$ meets $\partial X$ at only one point $\xi_i$ since $\dim A_{\mathbf P_i,t_i}=1$. In fact $\xi_i$ is the point at infinity corresponding to the half-geodesic $A_{\mathbf{P}_i,t_i}$.

\subsection{Rational horocyclic coordinates}
Let $\mathbf{P}$ be a proper minimal rational parabolic subgroup
of $\mathbf{G}$ with $\bq$-rank$(\mathbf{G})=1$. The pullback
$\mu^*g$ of the metric $g$ on $X$ to $N_{\mathbf{P}} \times
A_\mathbf{P} \times X_\mathbf{P}$ is given by
$$ds^2_{(n,a,\tau(m))}= \sum_{\alpha \in
\Lambda^+} e^{-2\alpha(\log a)}
h_\alpha \oplus da^2 \oplus d(\tau(m))^2,$$ where
$h_\alpha$ is some metric on $\mathfrak{g}_\alpha$ that smoothly
depends on $\tau(m)$ but is independent of $a$. Choosing
orthonormal bases $\{ N_1,\ldots,N_r\}$ of
$\mathfrak{n}_\mathbf{P}$, $\{Z_1,\ldots,Z_l\}$ of some tangent
space $T_{\tau(m)}X_\mathbf{P}$ and $A \in
\mathfrak{a}_\mathbf{P}$ with $\| A\|=1$, one can obtain
\emph{rational horocyclic coordinates} $ \eta : N_{\mathbf{P}}
\times A_\mathbf{P} \times X_\mathbf{P} \ra \mathbb{R}^r \times
\mathbb{R} \times \mathbb{R}^l$ defined by $$\eta\left(\exp
\left(\sum_{i=1}^r x_iN_i\right), \exp (yA), \exp
\left(\sum_{i=1}^lz_iZ_i\right)\right)=(x_1,\ldots,x_r,y,z_1,\ldots,z_l).$$
We abbreviate $(x_1,\ldots,x_r,y,z_1,\ldots,z_l)$ by $(x,y,z)$.
Then the $G$-invariant Riemannian volume form $\omega_X$ on $X\cong
N_{\mathbf{P}} \times A_\mathbf{P} \times X_\mathbf{P}$ is given in terms of the rational horocyclic coordinates by
\begin{eqnarray}\label{eqn:volumeformula} \omega_X=h(x,z) \exp^{-2\|\rho_\mathbf{P} \| y} dx dy dz,\end{eqnarray} where
$h(x,z)$ is a smooth function that is independent of $y$. See
\cite[Corollary 4.4]{Bo3}.

Note that all proper rational minimal parabolic subgroups are
conjugate under $\mathbf{G}(\bq)$. Hence the respective root
systems are canonically isomorphic \cite{Bo2} and moreover, one
can conclude $\|\rho_\mathbf{P}\|=\|\rho_{\mathbf{P}'}\|$ for any
two proper minimal rational parabolic subgroups
$\mathbf{P},\mathbf{P}'$ of $\bq$-rank $1$ algebraic group
$\mathbf{G}$.

\subsection{Tits compactification and disjoint mapping cone}

We recall here the Tits compactification of a $\mathbb Q$-rank $1$ locally symmetric space.
For any proper rational parabolic subgroup $\mathbf P$ of $\mathbf G$, define the positive Weyl chamber $$A^+_{\mathbf P} = \{ a\in A_{\mathbf P} \ | \ \alpha(a)>0, \ \alpha \in  \Lambda \},$$ and let $A^+_{\mathbf P}(\infty)$ be the set of points at infinity of $A_{\mathbf P}^+$. Then the rational Tits building for $\mathbf G$ is realized by a disjoint union of the open simplexes $$\Delta_{\mathbb Q}(X) = \coprod_{\mathbf P} A^+_{\mathbf P}(\infty),$$ where $\mathbf P$ runs over all the rational parabolic subgroups of $\mathbf G$.
In the case that $\mathbf G$ has $\mathbb Q$-rank $1$, its Tits building $\Delta_{\mathbb Q}(X)$ is a disjoint union of uncountably many points and $\Gamma \backslash \Delta_{\mathbb Q}(X)$ consists of finitely many points, which are in canonical one-to-one correspondence with the ends of $\Gamma \backslash X$. Then the Tits compactification $\tits$ of $\Gamma \backslash X$ is defined by the quotient $\Gamma \backslash (X\cup \Delta_{\mathbb Q}(X))$. We refer the reader to \cite[Chapter III.12]{BJ06} for more details.

By using the precise reduction theory, the Tits compactification of a $\mathbb Q$-rank $1$ locally symmetric space is obtained as follows. Let $\Gamma \backslash X = \Omega \cup \coprod_{i=1}^s \pi(U_i \times A_{\mathbf{P}_i,t_i} \times V_i)$ be the disjoint decomposition as in Theorem \ref{thm:reduction theory}. Then $\Omega$ is a compact submanifold with boundary and is homeomorphic to $\borelserre$. Moreover, each connected component of the boundary of $\Omega$ consists of the projection of a horosphere in $X$.
The set $(U_i \times A_{\mathbf P_i,t_i} \times V_i)(\infty)$ of points at infinity for each Siegel set $U_i \times A_{\mathbf P_i,t_i} \times V_i$ consists of exactly one point $A_{\mathbf P_i}^+(\infty)=\{\xi_i\}$. Hence the Tits compactification $\tits$ is
$$\tits = \Omega \cup \coprod_{i=1}^s \pi(U_i \times A_{\mathbf P_i,t_i} \times V_i \cup \{\xi_i\}).$$
Thus, in the $\mathbb Q$-rank $1$ case, the Tits compactification is the end compactification, which is obtained by adding one point to each end.

We now identify $\tits$ with the disjoint mapping cone $Dcone(\Omega,\partial \Omega)$ of $(\Omega,\partial \Omega)$.
In fact one can easily see that there is a natural identification between them: first of all, we define a homeomorphism $\varphi : [0,1] \rightarrow \mathbb R_{\geq 0}\cup \{+\infty\}$ by $$\varphi (t)=\tan \frac{\pi}{2}t.$$
Geometrically, $U_i \times A_{\mathbf P_i,t_i} \times V_i \cup \{\xi_i\}$ can be viewed as the geodesic cone over $U_i \times \{a_i\}\times V_i$ with top point $\xi_i$ where $a_i \in A_{\mathbf P_i}$ with $\alpha(a_i)=t_i$. Then $\pi(U_i \times \{a_i\} \times V_i)$ is a connected component of the boundary of $\Omega$, denoted by $\partial_i \Omega$.
We now define a homeomorphism $f_i : \partial_i \Omega \times [0,1] \rightarrow U_i \times \overline{A_{\mathbf P_i,t_i}} \times V_i \cup \{\xi_i\}  $ by $$f_i(\pi(n,a_i,m),t)=(n,a(t),m) \text{ for }t\in [0,1) \text{ and } f_i(\pi(n,a_i,m),1)=\xi_i$$ where $\overline{A_{\mathbf P_i,t_i}}=\{ a \in A_{\mathbf P_i} \ | \ \alpha(a)\geq t_i\}$ and $a(t)$ is the unique element of $A_{\mathbf P_i,t_i}$ with $\alpha(a(t))=t_i +\varphi(t)$. In other words, $a(t)$ is the unique point on $A_{\mathbf P_i,t_i}$ with $d(a_i,a(t))=\varphi(t)$. Since $f_i$ maps $\partial_i \Omega \times \{1\}$ to $\xi_i$ and the set $U_i \times \overline{A_{\mathbf P_i,t_i}} \times V_i \cup \{\xi_i\}$ is mapped injectively into $\tits$ by the projection $\pi : X \cup \Delta_{\mathbb Q}(X) \rightarrow \tits$, the map $f_i$ induces a homeomorphism $$f_i : Cone(\partial_i \Omega) \rightarrow \pi(U_i\times \overline{A_{\mathbf P_i,t_i}} \times V_i \cup \{\xi_i\}).$$
Then we obtain a homeomorphism $f : Dcone(\Omega,\partial \Omega) \rightarrow \tits$ with $f|_{\Omega}=id$ and $f|_{Cone(\partial_i \Omega)}=f_i$. From now on, we think of $Dcone(\Omega,\partial \Omega)$ as $\tits$ via this homeomorphism.

For $\tau \in S_{k-1}(\partial_i \Omega)$, recall that we define a singular $k$-simplex $cone(\tau)$ in $Dcone(\Omega,\partial \Omega)$ in Section \ref{sec:cone}. Via the identification between $Dcone(\Omega,\partial \Omega)$ and $\tits$, it can be easily seen that $cone(\tau)$ is the geodesic cone over $\tau$ with top point $\pi(\xi_i)$. Here we give another description for the construction of $cone(\tau)$ in $\tits$. First, regard $\tau$ as a simplex $\tau :\Delta^{k-1} \rightarrow \Gamma \backslash X$ and then choose a lift $\tilde \tau : \Delta^{k-1} \rightarrow X$.
Since $\partial \Omega$ is the projection of a horosphere in $X$, we can assume that the image of $\tilde \tau$ is contained in a horosphere $\tilde S$. Note that such horospheres are in one-to-one correspondence with the vertices of the Tits building $\Delta_{\mathbb Q}(X)$. Let $\tilde c$ be the vertex of $\Delta_{\mathbb Q}(X)$ associated with $\tilde S$. We define the cone of $\tilde \tau$, denoted by $cone(\tilde \tau)$, as the geodesic cone over $\tilde \tau$ with top point $\tilde c$. Then we obtain a continuous map $cone(\tilde \tau):\Delta^k \rightarrow X\cup \Delta_{\mathbb Q}(X)$. Since the correspondence between the vertices of $\Delta_{\mathbb Q}(X)$ and horospheres is $\Gamma$-equivariant, $cone(\tilde \tau)$ is well defined up to the $\Gamma$-action on $X\cup \Delta_{\mathbb Q}(X)$ and so we get a singular simplex in $\tits$. The resulting simplex is actually $cone(\tau)$ as defined in Section \ref{sec:cone} up to a reparametrization.

\begin{lemma}\label{lem:conevol}
Let $\omega_X$ be the Riemannian volume form on $X$. Then for any $\tau \in S_{n-1}(\partial \Omega)$,
$$\left| \int_{cone(\tau)} \omega_X \right| \leq \frac{\mathrm{Vol}_{n-1}(\tau)}{2 \|\rho_{\mathbf P}\|}.$$
\end{lemma}
\begin{proof}
As observed above, $cone(\tau)$ is obtained by the geodesic cone $cone(\tilde \tau)$ over $\tilde \tau$ with top point $\tilde c$ where $\tilde \tau : \Delta^{n-1} \rightarrow N_{\mathbf P_i} \times \{a_i\} \times X_{\mathbf P_i}$ is a lift of $\tau$. Then a direct computation with the formula in (\ref{eqn:volumeformula}) for the Riemannian volume form $\omega_X$ in the rational horocyclic coordinates $N_{\mathbf P_i} \times A_{\mathbf P_i} \times X_{\mathbf P_i}$ gives the inequality in the lemma.
\end{proof}

\section{Bounded volume class}

Let us define the set of conic singular simplices in $\tits$ by $$Cone(S_{k-1}(\partial \Omega))= \{ cone(\tau) \ | \ \tau \in S_{k-1}(\partial \Omega) \}.$$ Let $C^{conic}_*(\tits)$ be the set of real singular chains generated by $S_*(\Omega)$ and $Cone(S_{*-1}(\partial \Omega))$.
Then the Riemannian volume form $\omega_X$ on $X$ gives rise to a map $\Theta : C_n^{conic}(\tits) \rightarrow \mathbb R$  defined as follows: $$\Theta(\sigma)= \omega_X (\sigma):=\int_\sigma \omega_X \text{ for }\sigma \in S_n(\Omega) $$ and $$\Theta(cone(\tau))= \omega_X (cone(\tau)):=\int_{cone(\tau)}\omega_X \text{ for }\tau \in S_{n-1}(\partial \Omega).$$
This definition makes sense due to Lemma \ref{lem:conevol} even though $cone(\tau)$ is noncompact in $\Gamma \backslash X$.

\begin{lemma}\label{lem:cocycle}
The map $\Theta : C_n^{conic}(\tits) \rightarrow \mathbb R$ is a cocycle.
\end{lemma}
\begin{proof}
From Stokes' theorem, it follows that $\delta \Theta(\sigma')=0$ for $\sigma'\in S_{n+1}(\Omega)$. Hence it is sufficient to show that $\delta \Theta(cone(\tau'))=\omega_X(\partial_{n+1} cone(\tau'))=0$ for any $\tau' \in S_n(\partial \Omega)$. Recall that $cone(\tau')$ is regarded as a map $cone(\tau') : \Delta^n \times [0,1] \rightarrow \tits$. Moreover, for $t \in [0,1)$, $cone(\tau')$ defines a map $\Delta^n \times \{t\} \rightarrow \Gamma \backslash X-\Omega.$
Applying Stokes' theorem to the compact set $cone(\tau')(\partial (\Delta^n \times [0,t]))$,
$$\int_{\partial (\Delta^n \times [0,t])} cone(\tau')^*\omega_X = \int_{\Delta^n \times [0,t]}  cone(\tau')^*d\omega_X=0.$$
On the other hand, $\partial (\Delta^n \times [0,t])$ can be decomposed into three pieces, $\partial \Delta^n \times [0,t]$, $\Delta^n\times \{0\}$ and $\Delta^n \times \{t\}$. As $t \rightarrow 1$, it can be easily seen that the integral over $\partial \Delta^n \times [0,t]$ converges to $\omega_X(cone(\partial_n \tau'))$ and the integral over $\Delta^n \times \{t\}$ decays exponentially to $0$. The integral over $\Delta^n \times \{0\}$ is $(-1)^{n+1} \omega_X(\tau')$ and hence we have
\begin{eqnarray*}
0&=&\lim_{t\rightarrow 1} \int_{\partial (\Delta^n \times [0,t])} cone(\tau')^*\omega_X \\
&=&\omega_X(cone(\partial_n \tau')+(-1)^{n+1}\tau')=\omega_X(\partial_{n+1} cone(\tau'))
\end{eqnarray*}
Therefore, we conclude that $\Theta$ is a cocycle.
\end{proof}

\begin{lemma}
Let $[\tits]_{conic}$ be the fundamental class in $H_n^{conic}(\tits)$. Then $$ \langle [\Theta], [\tits]_{conic} \rangle= \mathrm{Vol}(\Gamma \backslash X).$$
\end{lemma}
\begin{proof}
Let $z$ be a relative fundamental cycle of $(\Omega,\partial \Omega)$ obtained from a triangulation of $(\Omega,\partial \Omega)$.
Then $z+(-1)^{n+1}cone(\partial z)$ is a fundamental cycle representing $[\tits]_{conic}$. It is clear that $$\omega_X(z)=\mathrm{Vol}(\Omega) \text{ and } \omega_X((-1)^{n+1}cone(\partial z))=\mathrm{Vol}(\Gamma \backslash X - \Omega).$$ Hence we have $$\langle \Theta, z+(-1)^{n+1} cone(\partial z) \rangle=\mathrm{Vol}(\Gamma\backslash X),$$
which implies the lemma.
\end{proof}

Now we will define a bounded cocycle $\Theta_b : C_n^{conic}(\tits) \rightarrow \mathbb R$ by using the barycentric straightening $st_* : C_*(\Gamma \backslash X) \rightarrow C_*(\Gamma \backslash X)$. For a simplex $\sigma \in S_k(\Omega)\subset S_k(\Gamma \backslash X)$, it has been already seen that $st_k(\sigma)$ is well defined. We need to define the barycentric straightening process for simplices in $Cone(S_{*-1}(\partial \Omega))$.
Let $\tau \in S_{k-1}(\partial \Omega)$ and $\tilde c \in \Delta_{\mathbb Q}(X)$ be the vertex associated with a lift $\tilde \tau : \Delta^{k-1} \rightarrow X$ of $\tau$.
Then we define $cone(st_{k-1}(\tilde \tau))$ by the geodesic cone over $st_{k-1}(\tilde \tau)$ with top point $\tilde c$. It is not difficult to check that $cone(st_{k-1}(\tilde \tau))$ is well defined $\Gamma$-equivariantly. Hence $cone(st_{k-1}(\tau)) \in S_k(\tits)$ is well defined.
We use this process as the straightening process of $cone(\tau)$, i.e.,
$$st_k(cone(\tau)) := cone (st_{k-1}(\tau)).$$
Notice that $st_k(cone(\tau))$ is not in $Cone(S_{k-1}(\partial \Omega))$ but in $S_k(\tits)$.
We extend the barycentric straightening process on $Cone(S_{*-1}(\partial \Omega))$ linearly to $Cone(C_{*-1}(\partial \Omega))$.

The straightening process on $C^{conic}_*(\tits)$ commutes with the boundary operator on $C_*(\tits)$:
\begin{eqnarray*}
(st_{k-1} \circ \partial_k)(cone(\tau)) &=& st_{k-1}(cone(\partial_{k-1} \tau) +(-1)^k\tau) \\
&=& cone (st_{k-2}(\partial_{k-1} \tau)) +(-1)^k st_{k-1}(\tau) \\
&=& cone (\partial_{k-1} (st_{k-1}(\tau)))+(-1)^k st_{k-1}(\tau) \\
&=& \partial_k cone(st_{k-1}(\tau)) \\
&=& (\partial_k \circ st_k) (cone(\tau)).
\end{eqnarray*}
Using this straightening, define a cochain $\Theta_b : C_n^{conic}(\tits) \rightarrow \mathbb R$ by
$$\Theta_b(\sigma)= \omega_X (st_n(\sigma)) \text{ for }\sigma \in S_n(\Omega)$$  and $$\Theta_b(cone(\tau))= \omega_X (st_n(cone(\tau))) \text{ for }\tau \in S_{n-1}(\partial \Omega).$$

\begin{lemma}\label{lem:boundedcocycle}
The cochain $\Theta_b$ is a cocycle. Furthermore, if $X$ is an $n$-dimensional symmetric space of noncompact type with no direct factors isometric to $\mathbb R$, $\mathbb H^2$, $\mathbb H^3$, $\mathrm{SL}_3 \mathbb R / \rso(3)$, $\mathrm{SL}_4 \mathbb R / \rso(4)$ or $\mathrm{Sp}_4 \mathbb R /\mathrm U(2)$, then $\Theta_b$ is bounded.
\end{lemma}
\begin{proof}
It is easy to check that for a simplex $\sigma' \in S_{n+1}(\Omega)$, $$\delta \Theta_b (\sigma') = \omega_X( st_n(\partial_{n+1} \sigma'))=\omega_X(\partial_{n+1} st_{n+1}(\sigma'))=d\omega_X (st_{n+1}(\sigma'))=0.$$
For a simplex $\tau' \in S_n(\partial \Omega)$, we have
$$\delta \Theta_b (cone(\tau'))=\omega_X(st_{n+1}(\partial_{n+1} cone(\tau')))=\omega_X(\partial_{n+1} cone( st_n(\tau'))).$$
As in the proof of Lemma \ref{lem:cocycle}, we can deduce that $\Theta_b$ is a cocycle. If $X$ is a symmetric space of noncompact type with no direct factors isometric to $\mathbb R$, $\mathbb H^2$, $\mathbb H^3$, $\mathrm{SL}_3 \mathbb R / \rso(3)$, $\mathrm{SL}_4 \mathbb R / \rso(4)$ or $\mathrm{Sp}_4 \mathbb R /\mathrm U(2)$, it follows from Corollary \ref{cor1.2} that the volumes of straightened $(n-1)$-simplices are uniformly bounded from above. By Lemma \ref{lem:conevol}, for any $\tau \in S_{n-1}(\partial \Omega)$,
$$|\Theta_b(cone(\tau)) |=\left| \int_{cone(st_{n-1}(\tau))} \omega_X \right| \leq \frac{\mathrm{Vol}_{n-1}(st_{n-1}(\tau))}{2 \|\rho_{\mathbf P}\|}.$$
Since $\| \rho_{\mathbf P} \|$ is a constant depending only on $\mathbf G$,
the boundedness of $\Theta_b$ follows immediately.
\end{proof}

\begin{lemma}\label{lem:bounded}
The cocycle $\Theta_b$ represents the volume class $[\Theta] \in H^n_{conic}(\tits)$.
\end{lemma}
\begin{proof}
It suffices to find an $(n-1)$-cochain $\beta : C_{n-1}^{conic}(\tits) \rightarrow \mathbb R$ satisfying $\Theta_b - \Theta = \delta \beta$.
Let $H_* : C_*(\Gamma\backslash X) \rightarrow C_{*+1}(\Gamma\backslash X)$ be the chain homotopy from the barycentric straightening to the identity, that is constructed by the straight line homotopy between a simplex and its straight simplex. This homotopy satisfies $\partial_{k+1} H_k +H_{k-1} \partial_k = st_k-id$. For each $\tau \in S_{k-1}(\partial \Omega)$, define $H_k(cone(\tau)) \in C_{k+1}(\tits)$ by
$$H_k(cone(\tau)):= cone (H_{k-1}(\tau)).$$
Note that for any simplex occurring in $H_{k-1}(\tau)$, its cone is well defined in the same way that the cone of $st_{k-1}(\tau)$ is defined.
Hence $cone (H_{k-1}(\tau))$ is well defined. Then
\begin{eqnarray*}
\lefteqn{(H_{k-1} \partial_k + \partial_{k+1} H_k)(cone(\tau))} \\
&=& H_{k-1} (\partial_k cone(\tau)) + \partial_{k+1} (H_k (cone(\tau))) \\
&=& H_{k-1} (cone(\partial_{k-1} \tau) + (-1)^k \tau) + \partial_{k+1} cone( H_{k-1}(\tau)) \\
&=& cone(H_{k-2}(\partial_{k-1} \tau)) +(-1)^k H_{k-1}(\tau) + cone(\partial_k H_{k-1}(\tau)) + (-1)^{k+1}H_{k-1}(\tau) \\
&=& cone((H_{k-2}\partial_{k-1} + \partial_k H_{k-1})(\tau)) \\
&=& cone(st_{k-1}(\tau))-cone(\tau) \\
&=& (st_k-id)(cone(\tau)).
\end{eqnarray*}

Now we define $\beta : C_{n-1}^{conic}(\tits) \rightarrow \mathbb R$ by
$$\beta(\zeta) = \omega_X (H_{n-1} (\zeta)) \text{ and } \beta(cone(\eta))= \omega_X ( H_{n-1}(cone(\eta)))$$
for $\zeta \in S_{n-1}(\Omega)$ and $\eta \in S_{n-2}(\partial \Omega)$. Then for $\tau \in S_{n-1}(\partial \Omega)$,
\begin{eqnarray*}
(\Theta_b - \Theta) (cone(\tau)) &=& \omega_X ( (st_n-id)(cone(\tau))) \\
&=& \omega_X ( (H_{n-1} \partial_n + \partial_{n+1} H_n)(cone(\tau))) \\
&=& \omega_X ( H_{n-1}(\partial_n cone(\tau)))\\
&=& \beta (\partial_n cone(\tau)) \\
&=& \delta \beta( cone(\tau)).
\end{eqnarray*}
In the same way, it follows that for any $\sigma \in S_n(\Omega)$, $(\Theta_b - \Theta) (\sigma) =\delta \beta(\sigma)$.
Therefore, we conclude that $\Theta_b - \Theta =\delta \beta$.
\end{proof}

Now we are ready to prove Theorem \ref{thm:Qrank1}.

\begin{proof}[Proof of Theorem \ref{thm:Qrank1}]
Lemmas \ref{lem:boundedcocycle} and \ref{lem:bounded} imply that the comparison map $$H^n_{conic,b}(\tits) \rightarrow H^n_{conic}(\tits)$$ is surjective.
From the duality of the $\ell^1$ and $\ell^\infty$-norms, the $\ell^1$-seminorm of $[\tits]_{conic}$ is strictly positive. By Corollary \ref{cor:conicnorm},
$$0< \| [\tits]_{conic}\| =\|[Dcone(\Omega,\partial \Omega)]_{conic}\|=\|\Omega,\partial \Omega\|(1).$$
Since $\Omega$ is homeomorphic to $\borelserre$, we have $$0<\|\Omega,\partial \Omega\|(1)=\|\borelserre,\partial \borelserre\|(1),$$
Furthermore, by Lemma \ref{lem:keyinequal},
$$0< \|\borelserre,\partial \borelserre\|(1) \leq \| \Gamma \backslash X \|_\mathrm{lf}$$
This completes the proof.
\end{proof}

\begin{proof}[Proof of Theorem \ref{irred}]
We only need to prove the theorem for the cases of $\mathbb H^2$, $\mathbb H^3$ and $\mathrm{SL}_3 \mathbb R / \rso(3)$.
It is well known that any proper rational parabolic subgroup of $\mathrm{SL}_2 \mathbb R$, $\mathrm{SL}_2 \mathbb C$ or $\mathrm{SL}_3 \mathbb R$  is conjugate to a Borel subgroup. Hence in either case, the fundamental group of each connected component of $\partial \borelserre$ is amenable.
According to \cite{LS09-1, BKK, KK12, Michelle}, we have $$\| \borelserre,\partial \borelserre \|= \| \Gamma \backslash X \|_\mathrm{lf}=\| \Gamma \backslash X \|_\mathrm{Lip} =\frac{\mathrm{Vol}(\Gamma\backslash X)}{\|\omega_X\|_\infty}>0,$$
which completes the proof.
\end{proof}

\begin{proof}[Proof of Theorem \ref{thm:degree}]
Gromov \cite{Gr82} proved for an $n$-dimensional smooth manifold $N$ with $\text{Ricci}_N \geq -(n-1)$,
$$\| N\|_\mathrm{lf} \leq (n-1)^n n! \cdot \mathrm{Vol}(N).$$
Let $\Gamma$ be a $\mathbb Q$-rank $1$ lattice in a semisimple Lie group $G$. Let $M=\Gamma \backslash X$ where $X$ is the associated symmetric space with $G$. From Lemma \ref{lem:bounded}, it can be easily seen that there exists a constant $C$ depending on $X$ such that
$$\mathrm{Vol}(\Gamma\backslash X) \leq C \cdot \|\borelserre,\partial \borelserre\|(1) \leq C \cdot \| \Gamma \backslash X \|_\mathrm{lf}.$$
Hence, we get
$$\frac{\mathrm{deg}(f)}{C} \cdot \mathrm{Vol}(M)\leq \mathrm{deg}(f) \cdot \| M \|_\mathrm{lf} \leq \| N \|_\mathrm{lf} \leq (n-1)^n n! \cdot \mathrm{Vol}(N).$$  Since, for a given dimension $n$, there are only finitely many symmetric spaces of noncompact type, the constant $C$ can be chosen in such a way that it depends only on $n$.
\end{proof}

\section{The Jacobian estimate}\label{sec:jacobianbounded}
From now on we dedicate the rest of the paper to proving Theorems \ref{thm:codim1Jac}, \ref{rank2jacobian} and \ref{rank2thm2}.
To estimate the volume of a straightened $k$-simplex, we start by recalling the analysis in \cite[Section 4]{CF03} given by Connell and Farb. We will denote by $$dB_{(x,\theta)} : T_xX \rightarrow \mathbb R \text{ and } DdB_{(x,\theta)} : T_xX \otimes T_xX \rightarrow \mathbb R$$ the $1$-form and the $2$-form, respectively, obtained by differentiating the Busemann function $B(p,x,\theta) : X \rightarrow X$ at the point $x\in X$. Once we fix a base point $p$, we will use the notation $B(p,x,\theta)=B_{(x,\theta)}$.
 We stick to the notation used in Section \ref{barycentric}. Since $st_V(\delta)$ is the unique critical point of $g_{\widehat V (\delta)}$ for a given ordered set $V=\{x_1,\ldots,x_{k+1}\}$, the point $st_V(\delta)$ is characterized by the $1$-form equation
\begin{eqnarray}\label{eqn:dg}
0=D_{x=st_V(\delta)} (g_{\widehat V (\delta)}) = \int_{\partial_F X} dB_{(st_V(\delta),\theta)}(\cdot) \ d \left( \sum_{i=1}^{k+1} a_i^2 \nu(x_i) \right)(\theta).
\end{eqnarray}
Moreover, differentiating Equation (\ref{eqn:dg}) with respect to the directions in $T_\delta \Delta^k_s$, one obtains
\begin{eqnarray}\label{eqn:ddg}
\lefteqn{0=D_\delta D_{x=st_V(\delta)} (g_{\widehat V (\delta)})} \\ &=& \sum_{i=1}^{k+1} 2a_i \langle \cdot, e_i \rangle_\delta \int_{\partial_F X} dB_{(st_V(\delta),\theta)}(\cdot) \ d \left( \nu(x_i) \right)(\theta) \\ & & +\int_{\partial_F X} DdB_{(st_V(\delta),\theta)}(D(st_V)_\delta(\cdot), \cdot) \ d \left( \widehat V(\delta) \right) (\theta).
\end{eqnarray}
Here $\langle v, e_i\rangle_\delta=v_i$ denotes the standard inner product.

Define bilinear symmetric forms $k_\delta$ and $h_\delta$ on $T_{st_V(\delta)}X$ by
\begin{eqnarray*} k_\delta (v,w) &=& \int_{\partial_F X} DdB_{(st_V(\delta),\theta)}(v,w) \ d \left( \widehat V(\delta) \right) (\theta) \\
h_\delta (v,w) &=& \int_{\partial_F X} dB_{(st_V(\delta),\theta)}(v) \ dB_{(st_V(\delta),\theta)}(w) \ d \left( \widehat V(\delta) \right) (\theta).
\end{eqnarray*}

Then for any unit tangent vector $u \in T_\delta \Delta^k_s$ and any vector $v\in T_{st_V(\delta)}X$, we have
\begin{eqnarray} \label{eqn:kh}
\left| k_\delta ( D(st_V)_\delta(u), v) \right|^2 &=& \left| \sum_{i=1}^{k+1} 2a_i \langle u, e_i \rangle_\delta \int_{\partial_F X} dB_{(st_V(\delta),\theta)}(v) \ d \left( \nu(x_i) \right) (\theta) \right|^2 \nonumber \\
&\leq& 4 \left[\sum_{i=1}^{k+1} \langle u, e_i \rangle^2_\delta \right] \left[ \sum_{i=1}^{k+1} a_i^2 \int_{\partial_F X} dB^2_{(st_V(\delta),\theta)}(v) \ d \left( \nu(x_i) \right) (\theta)\right] \nonumber \\ &\leq& 4 h_\delta (v,v). \end{eqnarray}

Let $W=D(st_V)(T_\delta \Delta^k_s)$ be a subspace of $T_{st_V(\delta)}X$. To estimate the $k$-dimensional volume of a straightened $k$-simplex with ordered vertex set $V$, we may assume that $W$ is a $k$-dimensional subspace of $T_{st_V(\delta)}X$. Then $k_\delta$ induces a bilinear symmetric form $k_\delta^W : W \times W \rightarrow \mathbb R$ on $W$. Similarly, $h_\delta$ induces a bilinear symmetric form $h_\delta^W : W \times W \rightarrow \mathbb R$. Define symmetric endomorphisms $K_\delta^W $ and $H_\delta^W $ of $W$ by
$$\langle K_\delta^W(v), w \rangle = k_\delta^W(v,w) \text{ and } \langle H_\delta^W(v), w \rangle = h_\delta^W(v,w)$$ for $v, w \in W$ where $\langle \cdot , \cdot \rangle$ is the inner product on $T_{st_V(\delta)}X$ induced by the symmetric metric on $X$.
Clearly, for all $v, w \in W$, \begin{eqnarray}\label{eqn:w} k_\delta^W(v,w) =k_\delta(v,w)\text{ and }h_\delta^W(v,w)=h_\delta(v,w).\end{eqnarray} Let $\{w_1,\ldots,w_k\}$ be an orthonormal eigenbasis of $W$ for $H_\delta^W$. Such basis exists because $H_\delta^W$ is a positive definite symmetric endomorphism on $W$.
Consider the composition map
$$ \xymatrixcolsep{3.5pc}\xymatrix{
T_\delta \Delta^k_s \ar[r]^-{D(st_V)_\delta} &
W \ar[r]^-{K_\delta^W} & W.
}$$
Let $\{ \tilde u_1,\ldots, \tilde u_k \}$ be the basis of $T_\delta \Delta^k_s$ obtained by pulling back $\{w_1,\ldots,w_k\}$ via the map $K_\delta^W \circ D(st_V)_\delta$. Let $\{ u_1,\ldots, u_k \}$ be the orthonormal basis of $T_\delta \Delta^k_s$ obtained from the basis $\{ \tilde u_1,\ldots, \tilde u_k \}$ by applying the Gram-Schmidt process. The matrix representation of $K_\delta^W \circ D(st_V)_\delta$ with respect to bases $\{u_i\}_{i=1}^k$ and $\{w_i\}_{i=1}^k$ is an upper triangular $k$ by $k$ matrix. Hence by (\ref{eqn:kh}),
\begin{eqnarray*}
\left| \det \left( K_\delta^W \circ D(st_V)_\delta \right) \right|^2 &=& \prod_{i=1}^k \left| \langle K_\delta^W \circ D(st_V)_\delta (u_i), w_i \rangle \right|^2 \\ &=& \prod_{i=1}^k \left| k_\delta^W (D(st_V)_\delta(u_i),w_i) \right|^2 \\ &\leq& 2^{2k} \prod_{i=1}^k h_\delta^W (w_i,w_i) \\ &=& 2^{2k} \det  H^W_\delta .
\end{eqnarray*}
Therefore,
$$\left| \mathrm{Jac}_k(st_V)_\delta \right|^2 \leq 2^{2k} \frac{\det  H^W_\delta }{\left( \det K^W_\delta \right)^2}.$$
If there exists a constant $C>0$ depending only on $X$ such that for any $k$-dimensional subspace $W$ of $T_{st_V(\delta)}X$,
\begin{equation}\label{kjacobian} \frac{\det  H^W_\delta }{\left( \det K^W_\delta \right)^2} \leq C,\end{equation}
one can conclude that the volumes of the straightened $k$-simplices in $X$ are uniformly bounded from above. Hence it suffices to show the following theorem to prove Theorem \ref{thm:codim1Jac}.

\begin{theorem}\label{thm:jacobiancodim1}
Let $X$ be a symmetric space of noncompact type with no direct factors isometric to $\mathbb R$, $\mathbb H^2$, $\mathbb H^3$, $\mathrm{SL}_3 \mathbb R / \rso(3)$, $\mathrm{SL}_4 \mathbb R / \rso(4)$ or $\mathrm{Sp}_4 \mathbb R /\mathrm U(2)$. Let $\mu \in \mathcal M (\partial X)$ be a probability measure fully supported on $\partial_F X$ and let $x\in X$. Let $W$ be any codimension $1$ subspace of $T_xX$. Consider the endomorphisms $K_x^W$ and $H_x^W$ defined on $W$ by
$$\langle K_x^W(w),w \rangle =\int_{\partial_F X} DdB_{(x,\theta)}(w,w) \ d \mu(\theta)$$ and
$$\langle H_x^W(w),w \rangle=\int_{\partial_F X} dB^2_{(x,\theta)}(w) \ d\mu(\theta).$$
Then there is a positive constant $C>0$ depending only on $X$ such that
$$\mathrm{Jac}_x^W(\mu):=\frac{\det H_x^W}{(\det K_x^W)^2}\leq C.$$
\end{theorem}

The above theorem is an analogue of Connell and Farb's Theorem \ref{thm:connell-farb} for codimension $1$ subspaces.
Before giving a proof of Theorem \ref{thm:jacobiancodim1}, we briefly sketch the proof of Connell and Farb in \cite{CF03} in the next section for reader's convenience.

\section{The strategy of Connell and Farb's Proof}

We stick to the notations used in Theorem \ref{thm:connell-farb}. Set $r=\mathrm{rank}(X)$.
Let $\mathcal F = A x$ be the canonical maximal flat through $x \in X$. Let $K$ be a maximal compact subgroup stabilizing $x$. Choose an orthonormal basis $\{e_i \}_{i=1}^n$ for the tangent space $T_xX$ such that $e_1,\ldots,e_r$ is a basis for $\mathcal F$ with $e_1(\infty)=b^+(\infty)$. Recall that $b$ is the barycenter of the positive Weyl chamber.
Let $v_\theta$ denote the vector at $x$ with $v_\theta(\infty)= \theta$ and let $O_\theta$ be the orthogonal matrix corresponding to the derivative of an isometry in $K$ which sends $e_1$ to $v_\theta$.
Then with respect to the chosen orthonormal basis $\{e_i \}_{i=1}^n$, one may write the terms $\langle K_x(v), v \rangle $ and $\langle H_x(v),v \rangle$ in matrix form as $$\langle K_x(v),v \rangle =\int_{\partial_F X} v^t O_\theta \left( \begin{array}{cc} O_{r \times r} & O_{r \times (n-r)} \\ O_{(n-r)\times r} & D_\lambda \end{array} \right) O_\theta^t v \ d\mu (\theta),$$ and
$$\langle H_x(v),v \rangle=\int_{\partial_F X} v^t O_\theta \left( \begin{array}{cc} 1 & O_{(n-1)\times 1} \\ O_{1\times (n-1)} & O_{(n-1)\times (n-1)} \end{array} \right) O_\theta^t v \ d\mu(\theta).$$
Here, $O_{l\times m}$ denotes the $l$ by $m$ zero matrix and $D_\lambda$ is a diagonal matrix $\mathrm{Diag}(\lambda_1,\ldots,\lambda_{n-r})$. Note that $\{\lambda_1,\ldots,\lambda_{n-r} \}$ is the set of nonzero eigenvalues of $DdB_{(x, b^+(\infty))}$. Furthermore, there is a constant $c>0$ only depending on $X$ such that $$ \frac{h(g_0)}{c} \leq \lambda_i \leq c h(g_0)$$
for all $i=1,\ldots,n-r$. For more details, see \cite[Section 4.3]{CF03}.

\subsection{The Jacobian estimate}
Consider the endomorphism $Q_x$ of $T_xX$ defined by
$$\langle Q_x(v),v \rangle =\int_{\partial_F X} v^t O_\theta \left( \begin{array}{cc} O_{r \times r} & O_{r \times (n-r)} \\ O_{(n-r)\times r}& I_{n-r} \end{array} \right) O_\theta^t v \ d \mu (\theta),$$
where $I_{n-r}$ is the identity matrix of size $n-r$.
Let $D=\frac{h(g_0)}{c}$. Then it is easy to see that \begin{eqnarray}\label{eqn:Q} \langle K_x(v),v \rangle \geq D\langle Q_x(v),v \rangle \end{eqnarray} for all $v\in T_xX$.
From the matrix forms of $K_x$, $H_x$ and $Q_x$ above, it follows that $K_x$, $H_x$ and $Q_x$ are all positive symmetric matrices.

The matrix forms of $H_x$ and $Q_x$ can be reformulated in terms of the angle between two subspaces of $T_xX$. Consider the bi-invariant metric $d_{\mathrm{SO}(n)}$ on $\mathrm{SO}(n)$ such that $\mathrm{SO}(n)$ has diameter $\pi/2$. Then the angle between two subspaces $V, W \subset T_xX$ is defined as
$$\angle(V,W):=\inf \{ d_{\mathrm{SO}(n)}(I, P) \ | \ P\in \mathrm{SO}(n) \text{ with }PV\subset W \text{ or } PW \subset V \}.$$
Then for a unit vector $v$,  $\langle Q_x(v),v \rangle$ is written as
$$\langle Q_x(v),v \rangle=\int_{\partial_F X} \sum_{j=r+1}^n \langle O_\theta^t v, e_i \rangle^2 \ d\mu (\theta) =\int_{\partial_F X} \sin^2 \angle(O_\theta^t v, \mathcal F) \ d\mu(\theta).$$
Furthermore, observing that $$ \langle v, e_1 \rangle^2 \leq \sum_{j=1}^r \langle v, e_j \rangle^2 = \sin^2 \angle (v,\mathcal F^\bot),$$
$\langle H_x(v),v \rangle$ satisfies
$$\langle H_x(v),v \rangle=\int_{\partial_F X} \langle O_\theta^t v, e_1 \rangle^2 \ d\mu (\theta) \leq \int_{\partial_FX} \sin^2 \angle(O_\theta^tv,\mathcal F^\bot) \ d\mu(\theta),$$

Note that $tr(Q_x)=n-r$ and $\langle Q_x(v),v \rangle \leq 1$ for all unit vectors $v \in T_xX$. From these facts, it follows that the number of those eigenvectors of $Q_x$ with eigenvalues strictly less than $1/(r+1)$ can not exceed $r$. Hence, by (\ref{eqn:Q}), the number of those eigenvectors of $K_x$ with eigenvalues strictly less than $D/(r+1)$ can not exceed $r$ either. The key idea of Connell and Farb in proving Theorem \ref{thm:connell-farb} is to prove the weak eigenvalue matching Theorem in \cite{CF14}. Indeed there was a mistake in \cite[Theorem 1.4]{CF03}, which was pointed out by us. Recently
Connell and Farb fixed the mistake in \cite{CF14}.

\begin{theorem}[Weak Eigenvalue Matching, \cite{CF14}] \label{thm:eigenmatching} Let $X$ be a symmetric space of noncompact type with no direct factors isometric to $\mathbb R$, $\mathbb H^2$ or $\mathrm{SL}_3 \mathbb R / \mathrm{SO}(3)$. Then there are constants $C_1$ and $C$ so that the following holds. Given any $\epsilon < 1/(\mathrm{rank}(X)+1)^2$, for any orthonormal $k$-frame $v_1,\ldots,v_k$ in $T_xX$ with $k\leq \mathrm{rank}(X)$, whose span $V$ satisfies $\angle (V,\mathcal F)\leq \epsilon$ there is a $C_1 \epsilon$-orthonormal $2k$-frame given by vectors $v_1', v_1'',\ldots,v_k',v_k''$, such that for $i=1,\ldots,k$: $$\angle (hv_i',\mathcal F^\bot)\leq C\angle (hv_i,\mathcal F)$$ and $$\angle (hv_i'',\mathcal F^\bot)\leq C\angle (hv_i,\mathcal F)$$  for every $h \in K$, where $hv$ is the linear (derivative) action of $K$ on $v\in T_xX$.
\end{theorem}

A set of vectors $\{w_1,\ldots,w_k\}$ is called \emph{$\delta$-orthonormal $k$-frame} if $\langle w_i, w_j\rangle <\delta$ for all distinct $i,j$.
One may assume that $C>1$. Put $\epsilon =1/2(r+1)^2$.
Let $v_1,\ldots,v_k$ be the eigenvectors of $K_x$ with eigenvalues strictly less than $D\epsilon$. As mentioned before, $k\leq \mathrm{rank}(X)$ since $\epsilon < 1/(r+1)$. Then due to the weak eigenvalue matching theorem above, there is an $C_1\epsilon$-orthonormal $2k$-frame given by vectors $v_1',v_1'',\ldots,v_k',v_k''$ of $T_xX$ as in Theorem \ref{thm:eigenmatching}. From the weak eigenvalue matching theorem together with the concavity of $\sin^2 \alpha$ for $0 \leq \alpha \leq \pi/2$ gives,
for all $\theta \in \partial_F X$ and each $v_i'$, that
$$\sin^2 \angle (O_\theta^t v_i', \mathcal F^\bot) \leq \sin^2 C\angle(O_\theta^t v_i,\mathcal F) \leq C^2 \sin^2 \angle(O_\theta^t v_i,\mathcal F).$$
This implies that for each $v_i'$,
\begin{eqnarray}\label{eqn:HK}
\langle H_x (v_i'), v_i' \rangle & \leq &\int_{\partial_FX} \sin^2 \angle(O_\theta^tv_i',\mathcal F^\bot) \ d\mu(\theta) \\
&\leq & \int_{\partial_FX} C^2 \sin^2 \angle(O_\theta^tv_i,\mathcal F) \ d\mu(\theta) \nonumber \\
&=& C^2 \langle Q_x (v_i),v_i \rangle  \nonumber \\
&\leq& \frac{C^2}{D}\langle K_x(v_i), v_i\rangle. \nonumber
\end{eqnarray}
The last inequality comes from (\ref{eqn:Q}). All the above inequalities hold for each $v_i''$ as well.

The above inequality makes it possible to obtain a uniform bound on $\mathrm{Jac}_x(\mu)$ as follows:
Let $v_1,\ldots,v_n$ be an orthonormal eigenbasis of $K_x$ such that $\langle K_x(v_i),v_i \rangle < D\epsilon$ for all $i=1,\ldots,k$ and $\langle K_x(v_i),v_i \rangle \geq D\epsilon$ for all $i=k+1,\ldots, n$.
Noting that $\langle H_x(v),v \rangle \leq 1$ for all unit vectors $v \in T_xX$, we have that \begin{eqnarray*}
\det H_x &\leq& C' \prod_{i=1}^k \langle H_x(v_i'), v_i' \rangle \cdot \langle H_x(v_i''), v_i'' \rangle \\
&\leq& C' C^{4k}D^{-2k} \prod_{i=1}^k \langle K_x(v_i), v_i \rangle^2 \\
&\leq& C' C^{4k}D^{-2k} \prod_{i=k+1}^n \langle K_x(v_i), v_i \rangle^{-2} (\det K_x)^2 \\
&\leq& C' C^{4r}D^{-2n}(\epsilon^{-1})^{2n}  (\det K_x)^2
\end{eqnarray*}
where $C'=1/(1-C_1\epsilon)^{4k}$ and $\epsilon^{-1}=2(1+r)^2$.
Since the constant in front of $(\det K_x)^2$ depends only on $X$, Theorem \ref{thm:connell-farb} follows.

\subsection{Weak Eigenvalue Matching Theorem}

As we saw in the previous section, the weak Eigenvalue Matching Theorem is crucial for the estimate of the Jacobian $\mathrm{Jac}_x(\mu)$. In this section, we briefly describe what are key points in proving the weak eigenvalue matching theorem.

Let $v$ be a vector in $\mathcal F$ and $K_v$ denote the stabilizer subgroup of $v$ in $K$.
Define the subspace $$Q_v := (\mathrm{span}\{K_v \cdot \mathcal F\})^\bot.$$
Then the first key lemma is the following angle inequality.

\begin{lemma}[Angle Inequality, \cite{CF14}]
There exists a constant $C>0$, depending only on $\dim(X)$, so that for any $w \in Q_i$ and any $h \in K$
\begin{eqnarray}\label{angleinequal} \angle (hw, \mathcal F^\bot) \leq C \angle(hv, \mathcal F) \end{eqnarray}
where $h$ acts via the derivative action of $K$ on $v \in T_xX$.
\end{lemma}

From the above lemma, one can easily notice that in the weak eigenvalue matching theorem, $v_i'$ and $v_i''$ will be orthonormal vectors contained in $Q_{v_i}$. In fact, since $\dim(Q_{v}) \geq 2$ for any $v \in T_xX$, one can always choose such two orthonormal vectors $v_i'$ and  $v_i''$. However for a given orthonormal frame $v_1,\ldots,v_k$, it is not clear that it is possible to choose two orthonormal vectors $v_i' , v_i'' \in Q_{v_i}$ for each $v_i$ so that $v_1',v_1'',\ldots, v_k',v_k''$ is an almost orthonormal $2k$-frame. In \cite{CF14}, Connell and Farb first show that this is possible for any frame of $\mathcal F$. More precisely, they prove the following:

\vspace{2.5mm}

\textit{For any $k$-frame $v_1,\ldots,v_k$ of $\mathcal F$, there is an orthonormal $2k$-frame given by $v_1',v_1'',\ldots v_k',v_k''$, such that for $i=1,\ldots,k$, $$\angle (hv_i',\mathcal F^\bot)\leq C\angle (hv_i,\mathcal F) \text{ and } \angle (hv_i'',\mathcal F^\bot)\leq C\angle (hv_i,\mathcal F).$$ }
In fact, it is sufficient to show the above statement in order to prove the weak eigenvalue matching theorem in the following reason.
Briefly speaking, if an orthonormal frame $v_1,\ldots,v_k$ of $V$ with $\angle (V,\mathcal F) \leq \epsilon$ is given, then
one can move each $v_i$ a little to a vector $w_i$ in $\mathcal F$ via an element of $K$ such that $w_1,\ldots,w_k$ is a $k$-frame of $\mathcal F$. Applying the above statement to the $k$-frame $w_1,\ldots,w_k$, one obtain $2k$-orthonormal frame satisfying the angle inequalities for $w_1,\ldots,w_k$. To get the angle inequalities for $v_1,\ldots,v_k$, we need to move each vector of the $2k$-orthonormal frame a little again. Then we finally obtain a $C_1\epsilon$-orthonormal $2k$-frame satisfying the angle inequalities for $v_1,\ldots,v_k$. For a detailed proof, see \cite[Section 2]{CF14}.

\section{Proof of Theorem \ref{thm:jacobiancodim1}.}

In this section we prove Theorem \ref{thm:jacobiancodim1}.
We will see that it suffices to prove the following weak eigenvalue matching theorem for codimension $1$ subspaces.

\begin{theorem}\label{thm:eigenmatching2} Let $X$ be a symmetric space of noncompact type with no direct factors isometric to $\mathbb R$, $\mathbb H^2$, $\mathbb H^3$, $\mathrm{SL}_3 \mathbb R / \rso(3)$, $\mathrm{SL}_4 \mathbb R / \rso(4)$ or $\mathrm{Sp}_4 \mathbb R /\mathrm U(2)$. Then there are constants $C_1$ and $C$ depending only on $X$ so that the following holds. Given any $\epsilon<1/(\mathrm{rank}(X)+1)^2$, for any orthonormal $k$-frame $v_1,\ldots,v_k$ in any codimension $1$ subspace $W$ of $T_xX$ with $k\leq \mathrm{rank}(X)$, whose span $V$ satisfies $\angle (V, \mathcal F)\leq \epsilon$, there is a $C_1\epsilon$-orthonormal $2k$-frame given by vectors $v_1',v_1'',\ldots,v_k',v_k''$ in $W$, such that for $i=1,\ldots, k$, $$\angle (hv_i',\mathcal F^\bot ) \leq C \angle (hv_i, \mathcal F)$$ and $$\angle (hv_i'',\mathcal F^\bot ) \leq C \angle (hv_i, \mathcal F)$$
for every $h \in K$, where $hv$ is the linear (derivative) action of $K$ on $v\in T_x X$.
\end{theorem}

Armed with Theorem \ref{thm:eigenmatching2}, we can prove Theorem \ref{thm:jacobiancodim1} and thus Theorem \ref{thm:codim1Jac} follows.

\begin{proof}[Proof of Theorem \ref{thm:jacobiancodim1}]
Let $r=\mathrm{rank}(X)$.
Recall that the endomorphisms $K_x^W$ and $H_x^W$ on $W$ are, when written in matrix form, as $$\langle K_x^W(v),v \rangle =\int_{\partial_F X} v^t O_\theta \left( \begin{array}{cc} O_{r \times r} & O_{r \times (n-r)} \\ O_{(n-r)\times r} & D_\lambda \end{array} \right) O_\theta^t v \ d\mu (\theta),$$ and
$$\langle H_x^W(v),v \rangle=\int_{\partial_F X} v^t O_\theta \left( \begin{array}{cc} 1 & O_{(n-1)\times 1} \\ O_{1\times (n-1)} & O_{(n-1)\times (n-1)} \end{array} \right) O_\theta^t v \ d\mu(\theta)$$ for $v\in W$.
Define $Q_x^W$ to be the endomorphism of $W$ determined by
$$\langle Q_x^W(v),v \rangle =\int_{\partial_F X} v^t O_\theta \left( \begin{array}{cc} O_{r \times r} & O_{r \times (n-r)} \\ O_{(n-r)\times r}& I_{n-r} \end{array} \right) O_\theta^t v \ d \mu (\theta).$$

Note that $K_x^W$, $H_x^W$ and $Q_x^W$ can be viewed as positive symmetric $k \times k$ matrices.
Let $\{w_1,\ldots,w_{n-1}\}$ be an orthonormal eigenbasis of $W$ for the symmetric matrix $Q^W_x$.
Then the $i$th eigenvalue of the matrix $Q^W_x$ is
$$L_i = \langle Q_x^W w_i, w_i \rangle = \int_{\partial_F X} \sum_{j=r+1}^n \langle O^t_\theta w_i, e_j \rangle^2 \ d \mu(\theta).$$
Note that none of the eigenvalues $L_i$ are $0$ since $Q^W_x$ is a positive symmetric matrix (see \cite[Section 4.4]{CF03} for more details concerning this).
Let $w_n$ be a unit vector perpendicular to $W$. Then we have an orthonormal basis $\{w_1,\ldots,w_n\}$ of $T_xX$. Set $L_n=\langle Q_x w_n, w_n \rangle$.
Note that $0<L_i\leq 1$ for all $i=1,\ldots,n$ and moreover,
$$\sum_{i=1}^n L_i = tr (Q_x) = n-r.$$

Let $\epsilon = 1/2(1+r)^2$. Suppose that $k$ of the eigenvalues of $Q^W_x$ are strictly less than $\epsilon$.
Then since $\epsilon <1/(r+1)$, we have $$n-r =\sum_{i=1}^{n-1}L_i + L_n < \frac{k}{1+r}+n-k.$$
It follows that $k < 1+r$, i.e., $k\leq r$. Then (\ref{eqn:Q}) implies that the number of eigenvectors of $K_x^W$ with eigenvalue strictly less than $D \epsilon$ can not exceed $r$. Let $\{v_1,\ldots,v_{n-1}\}$ be the orthonormal eigenbasis of $W$ for $K_x^W$. Then
we may assume that $\langle K_x(v_i),v_i\rangle < D\epsilon$ for all $i=1,\ldots, k$ and $\langle K_x(v_i),v_i\rangle \geq D\epsilon$ for all $i=k+1,\ldots, n-1$ and $k\leq r$. Furthermore by a similar argument to that in \cite[Section 3]{CF14}, we may assume that the vector space $V$ spanned by $v_1,\ldots,v_k$ satisfies $\angle (V,\mathcal F) \leq \epsilon$.
Applying Theorem \ref{thm:eigenmatching2} to the orthonormal $k$-frame $v_1,\ldots,v_k$, there is a  $C_1\epsilon$-orthonormal $2k$-frame given by $v_1',v_1'',\ldots,v_k',v_k''$ of $W$ such that for all $i=1,\ldots,k$,
$$\langle H_x^W(v_i'), v_i'\rangle \leq \frac{C^2}{D}\langle K_x^W(v_i), v_i \rangle \text{ and }\langle H_x^W(v_i''), v_i''\rangle \leq \frac{C^2}{D}\langle K_x^W(v_i), v_i \rangle.$$

From the fact that $\langle H_x^W(v),v \rangle \leq 1$ for all unit vectors $v \in T_xX$, it follows that
\begin{eqnarray*}
\det H_x^W &\leq& C' \prod_{i=1}^k \langle H_x^W(v_i'), v_i' \rangle \cdot \langle H_x^W(v_i''), v_i'' \rangle \\
&\leq& C' C^{4k}D^{-2k} \prod_{i=1}^k \langle K_x^W(v_i), v_i \rangle^2 \\
&\leq& C' C^{4k}D^{-2k} \prod_{i=k+1}^{n-1} \langle K_x^W(v_i), v_i \rangle^{-2}  (\det K_x^W)^2 \\
&\leq& C' C^{4r}D^{-2n+2}(\epsilon^{-1})^{2n-2}  (\det K_x^W)^2
\end{eqnarray*}
where $C'=1/(1-C_1\epsilon)^{4k}$ and $\epsilon^{-1}=2(1+r)^2$.
The constant $C$ in Theorem \ref{thm:jacobiancodim1} may be taken to be  $$C'C^{4r}D^{-2n+2}(\epsilon^{-1})^{2n-2}.$$
Since this constant depends only on $X$, this completes the proof of Theorem \ref{thm:jacobiancodim1}.
\end{proof}

\section{Weak Eigenvalue Matching Theorem for codimension $1$ subspaces}

We now focus on the proof of the weak eigenvalue matching theorem for codimension $1$ subspaces.
For this, it suffices to show that the space has the following property.

\vspace{0.2cm}

{\bf Property $\bf E$}.
Let $X$ be a symmetric space of noncompact type. $X$ will be said to have property {\bf E} if for any maximal flat $\mathcal F$, and putting $r=\mathrm{rank}(X)$, then for any $r$-frame $\{v_1,\ldots, v_r\}$ of $\mathcal F$, there is an orthonormal $3r$-frame given by vectors $v_1',v_1'',v_1''',\ldots,v_r',v_r'',v_r'''$, such that for $i=1,\ldots, r$,
$$\mathrm{span}\{v_i',v_i'',v_i'''\}=:P_i \subset Q_i:=(\mathrm{span}\{K_i\cdot \mathcal F\})^\bot$$
where $K_i$ is the stabilizer group of $v_i$ in $K$.

\vspace{0.05cm}

\begin{remark}\label{remark:propertyE}
Let $\{v_1,\ldots, v_r\}$ be any $r$-frame of $\mathcal F$.
Firstly, choose a maximally singular line $\ell_1$ closest to $v_1$.
Secondly, choose a maximally singular line $\ell_2$ that is closest to $v_2$ among maximally singular lines not in the $1$-dimensional subspace spanned by $\ell_1$. In this way for each $i=2,\ldots,r$, choose a maximally singular line $\ell_i$ that is closest to $v_i$ among maximally singular lines not in the $(i-1)$-dimensional subspace spanned by the lines already chosen.
Choose a vector $w_i \in \ell_i$ for each $i$. Then it is clear that $\{w_1,\ldots,w_r\}$ is a $r$-frame of $\mathcal F$ and consists of maximally singular vectors. By the construction of the $r$-frame, it follows that $Q_{w_i} \subset Q_{v_i}$ for each $i$. Hence if the statement of property {\bf E} holds for any frame of maximally singular vectors, it holds for arbitrary frame. Consequently in order to show that $X$ has the property {\bf E}, it suffices to prove the statement of property {\bf E} only for frames of maximally singular vectors.
\end{remark}

\begin{theorem}\label{thm:9.1}
Let $X$ be a symmetric space of noncompact type satisfying property $\bf E$. Then Theorem \ref{thm:eigenmatching2} holds.
\end{theorem}

\begin{proof}
It is sufficient to prove Theorem \ref{thm:eigenmatching2} when $k=r:=\mathrm{rank}(X)$.
Given an orthonormal $r$-frame $\{v_1,\ldots,v_r\}$, a $c_1\epsilon$-orthonormal $r$-frame $\{w_1,\ldots,w_r\}$ of $\mathcal F$ can be obtained in the same way as in the proof in \cite[Section 2]{CF14}: choose a closest element $k_i \in K$ to the identity such that $\widehat w_i :=k_i^{-1}v_i$ lies in $\mathcal F$. Then replace $\widehat w_i$ with the most singular vector $w_i$ in the $\epsilon_0:=1/(r+1)^2$ ball about $\widehat w_i$ and that is closest to $\widehat w_i$.
Applying the property $\bf E$ to the $r$-frame $\{w_i\}$ of $\mathcal F$ produces an orthonormal $3r$-frame $\{w_i',w_i'',w_i'''\}$ as described in property $\bf E$. 

Similarly to the proof in the erratum \cite{CF14}, there is a constant $c_1>0$ depending only on $\mathrm{rank}(X)$ such that for $i=1,\ldots, r$,
$$v_i \in k_i' K_i \cdot \mathcal F \text{ for some }k_i' \in K \text{ with } d_K(k_i',id)<c_1\epsilon.$$
where $K_i$ is the stabilizer group of $w_i$ in $K$. Since $w_i$ is more singular than $\widehat w_i$,  $\mathrm{Stab}_K(\widehat w_i) \subset K_i$ and moreover, it follows from $v_i \in k_i' K_i \cdot \mathcal F $ that
$$ \mathrm{Stab}_K(v_i) \cdot k_i\mathcal F =\mathrm{Stab}_K(k_i\widehat w_i) \cdot k_i\mathcal F  = k_i \mathrm{Stab}_K(\widehat w_i) \cdot \mathcal F \subset k_i' K_i \cdot \mathcal F.$$
Hence we have $$k_i' Q_i = (k_i'K_i\cdot \mathcal F)^\bot \subset (\mathrm{Stab}_K(v_i)\cdot k_i\mathcal F)^\bot.$$
This implies that for any $w \in k_i'Q_i$ and any $h\in K$,  $$\angle (hw,k_i\mathcal F^\bot ) \leq C \angle (hv_i, k_i\mathcal F) \text{ and so }\angle (hw,\mathcal F^\bot ) \leq C \angle (hv_i, \mathcal F).$$

Since $W$ is a codimension $1$ subspace and the dimension of $P_i$ is $3$, the dimension of $k_i'P_i \cap W$ is at least $2$.
Hence it is possible to choose an orthonormal $2$-frame $\{u_i', u_i''\}$ of $P_i$ such that $k_i' (\mathrm{span}\{u_i', u_i''\}) \subset k_i'P_i \cap W$. We put $v_i'=k_i'u_i'$ and $v_i''=k_i'u_i''$ for each $i=1,\ldots,r$. Since the $P_i$s are pairwise orthogonal, $\{u_i',u_i''\}$ is an orthonormal $2r$-frame and hence $\{v_i',v_i''\}$ is an $C_1\epsilon$-orthonormal $2r$-frame of $W$. Furthermore, since $v_i'$ and $v_i''$ are vectors in $k_i'P_i \subset k_i'Q_i \subset (\mathrm{Stab}_K(v_i)\cdot k_i\mathcal F)^\bot$, they satisfy the desired angle inequality.
\end{proof}

Now we focus on the problem which symmetric spaces have the property {\bf E}. We will prove

\begin{theorem}\label{propertye}
Let $X$ be a symmetric space of noncompact type with no direct factors isometric to $\mathbb R$, $\mathbb H^2$, $\mathbb H^3$, $\mathrm{SL}_3 \mathbb R/\mathrm{SO}(3)$, $\mathrm{SL}_4 \mathbb R/\mathrm{SO}(4)$, $\mathrm{SL}_5 \mathbb R/\mathrm{SO}(5)$ or $\mathrm{Sp}_4\mathbb R/ \mathrm{U}(2)$. Then it has the property $\bf E$.
\end{theorem}

Theorem \ref{propertye} together with Theorem \ref{thm:9.1} implies Theorem \ref{thm:eigenmatching2} in all cases except $\rsl_5\mathbb R/\rso(5)$. We will give a proof of Theorem \ref{thm:eigenmatching2} in the $\rsl_5\mathbb R/\rso(5)$ case separately.

\section{Reduction to the irreducible case}

To make a reduction to the case when $X$ is irreducible, it suffices to deal with the product of two symmetric spaces.

\begin{proposition}\label{red}
If the property $\bf E$ holds for two symmetric spaces $X_1$ and $X_2$ of noncompact type, then it also holds for $X_1 \times X_2$.
\end{proposition}

We first need the following lemma.

\begin{lemma}\label{redlemma}
For $i=1,2$, let $V_i$ be a real vector space of dimension $n_i$ and $V=V_1 \times V_2$ be the product vector space of $V_1$ and $V_2$. Given a frame $\{v_1,\ldots,v_{n_1+n_2}\}$ of $V$, there exists a permutation $\tau$ of the set $\{1,\ldots,n_1+n_2\}$ such that
\begin{itemize}
\item[-] $\{p_1(v_{\tau(1)}),\ldots,p_1(v_{\tau(n_1)})\}$ is a frame of $V_1$ and,
\item[-] $\{p_2(v_{\tau(n_1+1)}),\ldots,p_2(v_{\tau(n_1+n_2)})\}$ is a frame of $V_2$
\end{itemize}
where $p_i : V_1 \times V_2 \rightarrow V_i$ is the canonical projection onto $V_i$.
\end{lemma}
\begin{proof}
We may assume that $V_1=\mathbb R^{n_1}$, $V_2=\mathbb R^{n_2}$ and $V=\mathbb R^{n_1+n_2}$.
Let $A$ be a square matrix of size $n_1+n_2$. One can write a matrix $A$ in terms of its column vectors
$$A = [a_1, \ \cdots, \ a_{n_1+n_2}].$$
The determinant of $A$ can be computed by
\begin{align}\label{det} \det(A) = \frac{1}{(n_1+n_2)!}\sum_{\sigma \in S_{n_1+n_2}} & sign(\sigma) \det \left( [p_1(a_{\sigma(1)}), \ \cdots, \ p_1(a_{\sigma(n_1)})] \right) \nonumber \\  & \ \cdot \det \left( [p_2(a_{\sigma(n_1+1)}), \ \cdots, \ p_2(a_{\sigma(n_1+n_2)})] \right)
\end{align}
where $S_{n_1+n_2}$ denotes the permutation group of $\{1,\ldots,n_1+n_2\}$.
This can be shown by verifying that the formula on the right hand side of (\ref{det}) satisfies the essential properties characterizing the determinant.

Let $B$ be the matrix of columns $v_1,\ldots, v_{n_1+n_2}$. Since $\{ v_1,\ldots, v_{n_1+n_2} \}$ is a frame of $V$, we have $\det (B) \neq 0$. Applying the formula (\ref{det}) to $B$, there exists a permutation $\tau \in S_{n_1+n_2}$ such that $$\det \left( [p_1(v_{\tau(1)}), \ \cdots, \ p_1(v_{\tau(n_1)})] \right)  \det \left( [p_2(v_{\tau(n_1+1)}), \ \cdots, \ p_2(v_{\tau(n_1+n_2)})] \right) \neq 0.$$
This implies the lemma.
\end{proof}

For $i=1,2$, let $\mathcal F_i$ be a maximal flat of $X_i$ and $G_i$ be the identity component of the isometry group of $X_i$ and $K_i$ be the maximal compact subgroup of $G_i$. Then a maximal flat $\mathcal F$ of $X=X_1 \times X_2$ can be written as a product $\mathcal F=\mathcal F_1 \times \mathcal F_2$ of maximal flats $\mathcal F_i$ of $X_i$, and the identity component $G$ of the isometry group of $X$ is $G_1 \times G_2$. A maximal compact subgroup $K$ of $G$ can be written as $K=K_1 \times K_2$ where $K_i$ are maximal compact subgroups of $G_i$.

Let $v=(v^1,v^2)$ be a vector in $\mathcal F$. Then the stabilizer subgroup $\mathrm{Stab}_K(v)$ of $v$ in $K$ can be written as a product
$$\mathrm{Stab}_K(v)=\mathrm{Stab}_{K_1}(v^1) \times \mathrm{Stab}_{K_2}(v^2).$$
Hence we have
$$ Q_v := (\mathrm{Stab}_K(v) \cdot \mathcal F)^\bot = (\mathrm{Stab}_{K_1}(v^1) \cdot \mathcal F_1 \times \mathrm{Stab}_{K_2}(v^2) \cdot \mathcal F_2 )^\bot.$$
Now we are ready to prove Proposition \ref{red}.

\begin{proof}[Proof of Proposition \ref{red}]
Let $\dim(\mathcal F_i)=n_i$ for $i=1,2$.
Let $\{v_1,\ldots,v_n\}$ be an $n$-frame of $\mathcal F=\mathcal F_1\times \mathcal F_2$ where $n=n_1+n_2$. For each $i=1,\ldots,n$, $v_i$ can be written as a product
$v_i = (v_i^1, v_i^2)$. According to Lemma \ref{redlemma}, we may assume that $\{v_1^1,\ldots,v_{n_1}^1\}$ is an $n_1$-frame of $\mathcal F_1$ and $\{v_{n_1+1}^2, \ldots,v_{n}^2\}$ is an $n_2$-frame of $\mathcal F_2$. Applying the property $\bf E$ to the $n_1$-frame $\{v_1^1,\ldots,v_{n_1}^1\}$ of $\mathcal F_1$ and the $n_2$-frame $\{v_{n_1+1}^2, \ldots,v_{n}^2\}$ of $\mathcal F_2$, we obtain an orthonormal $3n_1$-frame of $T_{x_1}X_1$ given by vectors $$(v^1_1)',(v^1_1)'',(v^1_1)''',\ldots,(v^1_{n_1})',(v^1_{n_1})'',(v^1_{n_1})''' $$ and an orthonormal $3n_2$-frame of $T_{x_2}X_2$ given by vectors $$(v^2_{n_1+1})',(v^2_{n_1+1})'',(v^2_{n_1+1})''',\ldots,(v^2_{n})',(v^2_{n})'',(v^2_{n})''' $$ such that for $i=1,\ldots, n_1$,
$$ \mathrm{span}\{(v^1_i)', (v^1_i)'', (v^1_i)'''\} \subset (\mathrm{Stab}_{K_1}(v^1_i) \cdot \mathcal F_1)^\bot,$$
and for $j=n_1+1,\ldots, n$,
$$ \mathrm{span}\{(v^2_j)', (v^2_j)'', (v^2_j)'''\} \subset (\mathrm{Stab}_{K_2}(v^2_j)\cdot \mathcal F_2)^\bot.$$
For $i=1,\ldots, n_1$, set $$v_i'=((v_i^1)',0), \ v_i''=((v_i^1)'',0), \ v_i'''=((v_i^1)''',0),$$
and for $j=n_1+1,\ldots, n$,
$$v_j'=(0,(v_j^2)'), \ v_j''=(0, (v_j^2)''), \ v_j'''=(0,(v_j^2)''').$$
Then it is clear that $\{v_1',v_1'',v_1''',\ldots, v_n',v_n'',v_n'''\}$ is an orthonormal $3n$-frame of $T_{x}X$.
Furthermore it can be easily seen that for $i=1,\ldots,n$,
$$ \mathrm{span}\{v_i',v_i'',v_i'''\} \subset (\mathrm{Stab}_K(v_i) \cdot \mathcal F)^\bot.$$
Therefore $X$ has the property $\bf E$.
\end{proof}

\section{Proof of Theorem \ref{propertye}}

As seen in the previous section, it suffices to prove Theorem \ref{propertye} for irreducible symmetric spaces of noncompact type. In this section we just deal with rank $1$ irreducible symmetric spaces and $\rsl_{n+1}\mathbb R/\mathrm{SO}(n)$. The proof of Theorem \ref{propertye} for the other higher rank symmetric spaces of noncompact type will be given in the Appendix.

From now on we fix notation as follows: Let $G$ denote a semisimple Lie group and $K$ a maximal compact subgroup of $G$. Denote by $X$ the associated symmetric space of rank $n$. We will denote the Lie algebra of $G$ and $K$ by $\mathfrak g$ and $\mathfrak k$ respectively. Let $\mathfrak g = \mathfrak k \oplus \mathfrak p$ be the Cartan decomposition and $\mathfrak a \subset \mathfrak p$ denote a maximal abelian subalgebra of $\mathfrak p$. Let $\mathfrak a^+$ denote a positive Weyl chamber of $\mathfrak a$. Let $\Lambda$ be the set of roots and $\Lambda^+$ the set of positive roots. Denote by $\Pi$ the basis of $\Lambda^+$.
Note that $T_xX$ and $\mathcal F$ can be canonically identified with $\mathfrak p$ and $\mathfrak a$ respectively.

Similarly to the erratum \cite{CF14}, property {\bf E} boils down to a combinatorial problem. We sketch this briefly.
Let $$\mathfrak g = \mathfrak g_0 +\sum_{\alpha \in \Lambda} \mathfrak g_\alpha$$ be the root space decomposition. By using the Cartan involution $\theta$ of $\mathfrak g$, one can define $\mathfrak p_\alpha = (\mathrm{Id}- \theta)\mathfrak g_\alpha \subset \mathfrak p$ for each $\alpha \in \Lambda$. Note that $\dim(\mathfrak p_\alpha) \geq 1$ and $\mathfrak p_\alpha =\mathfrak p_{-\alpha}$ and the projection map $\mathrm{Id} -\theta : \mathfrak g_\alpha \rightarrow \mathfrak p_\alpha$ is an isomorphism (see \cite[Proposition 2.14.2]{Eb}). Then $\mathfrak p$ admits an orthogonal direct sum
$$\mathfrak p=\mathfrak a + \sum_{\alpha \in \Lambda^+} \mathfrak p_\alpha.$$

Let $v$ be a vector of $\mathfrak a$. Then \begin{eqnarray}\label{eqnQv} Q_v = (\mathrm{span}(K_v \cdot \mathfrak a))^\bot = \sum_{\substack{ \alpha \in \Lambda^+, \\ \alpha(v) \neq 0 }} \mathfrak p_\alpha\end{eqnarray} where $K_v$ is the stabilizer subgroup of $v$ in $K$.
Choosing an orthonormal basis for $\mathfrak p_\alpha$ for each $\alpha \in \Lambda^+$,  we have an orthonormal basis for $\mathfrak a^\bot$. Let $\{b_i\}_{i=1}^m$ denote the basis for $\mathfrak a^\bot$ where $m=\dim(X)-\mathrm{rank}(X)$.
As seen in (\ref{eqnQv}), $Q_v$ is spanned by the collection of $\mathfrak p_\alpha$'s with $\alpha(v) \neq 0$.
Hence to every vector $v \in \mathfrak a$, one can associate a vector $u_v = (u_1^v, \ldots, u_m^v)$ as follows: put $u_i^v =1$ if $b_i \in Q_v$ and otherwise $u_i^v=0$. Denote by $|u_v|$ the number of entries of $u_v$ that are equal to $1$. Then it is clear that $\dim(Q_v)=|u_v|$. In addition, given $v,w \in \mathfrak a$, we denote by $|u_v \cap u_w|$ the number of entries of $u_v+u_w$ that are equal to $2$. Note that $|u_v \cap u_w|=\dim(Q_v \cap Q_w)$.

Given a frame $\{v_1,\ldots,v_n\}$ of $\mathfrak a$, one can associate an $n \times m$ matrix $A$ whose $i$th row consists of $u_{v_i}$. Then property {\bf E} is equivalent to the statement that one can pick three $1$ entries from each row of $A$, so that all of the $3n$ choices are in different columns. We denote the $i, j$th entry of $A$ by $a_{ij}$.

\vspace{3mm}

{\bf Property $\bf E$ restated}.
Let $G$ be a semisimple Lie group.
For any frame $\{v_1,\ldots,v_n\}$ of $\mathfrak a$, the associated matrix $A$ satisfies the following property:
For each $1\leq i\leq n$ there exists $1\leq k_i < l_i < m_i \leq m$ so that each $a_{i k_i}=a_{i l_i}=a_{i m_i}=1$ and $\cup_{i=1}^n \{ k_i, l_i, m_i \}$ has cardinality $3n$.
\smallskip

\subsection{Rank $1$ symmetric spaces}
Let $X$ be a rank $1$ symmetric space of noncompact type. Besson, Courtois and Gallot \cite{BCG99} gave a uniform upper bound on the $k$-Jacobian $\mathrm{Jac}_{k}(\mu)$ as follows. Let $W$ be a $k$-dimensional subspace of $T_x X$ for $k \geq 3$. Then
$$\mathrm{Jac}_x^W(\mu):=\frac{\det H_x^W}{(\det K_x^W)^2}\leq \frac{k^{k/2}}{(k-1)^k}.$$
This implies that $\mathrm{Jac}_x^W(\mu)$ is uniformly bounded from above for $k\geq 3$. Furthermore they gave the exact best upper bound when $k=\dim(X)$. This enabled them to prove the minimal entropy rigidity theorem for rank $1$ symmetric spaces.

From the viewpoint of the weak eigenvalue matching theorem, one can understand why the uniform upper bound on the Jacobian is obtained for $k\geq 3$ as follows. In the case of rank $1$ symmetric spaces of noncompact type, there is at most one eigenvector of $Q_x$ with eigenvalue strictly less than $1/2$, denoted it by $v_1$. Then it is easy to see that $Q_1 = v_1^\bot$. Given a $k$-dimensional subspace $W$ of $T_xX$, one can choose an orthonormal $2$-frame given by $v_1', v_1''$ satisfying the angle inequality (\ref{angleinequal}) if $\dim(W\cap v_1^\bot) \geq 2$. Since $\dim(W\cap v_1^\bot)=k-1$, one can conclude that $\mathrm{Jac}_x^W(\mu)$ is uniformly bounded if $\dim(W)\geq 3$.

If $\dim(X) \geq 4$, then $\dim(v_1^\bot) = \dim(X)-1 \geq 3$. Hence the property $\bf E$ holds for any rank $1$ symmetric space of noncompact type with dimension at least $4$, that is, all rank $1$ symmetric spaces of noncompact type except $\mathbb R$, $\mathbb H^2$ and $\mathbb H^3$.

\subsection{$\mathrm{SL}_{n+1}\mathbb R$}
In this section we will prove that the property {\bf E} holds for $\rsl_{n+1}\mathbb R$ if $n\geq 5$.
We start by recalling several necessary facts about $\rsl_{n+1}\mathbb R$.
\begin{enumerate}
\item $\mathfrak a =\{ (a_1,\ldots,a_{n+1})\in \mathbb R^{n+1} \ | \ a_1 +\cdots +a_{n+1} =0 \}$
\item $\mathfrak a^+ =\{ (a_1,\ldots,a_{n+1}) \in \mathfrak a \ | \ a_1 > a_2 > \cdots > a_{n+1} \}$
\item $\Lambda^+ = \{ a_i^*- a_j^* \ | \ i<j \}$  and $\Delta^+ = \{a_1^*-a_2^*,\ldots, a_n^*-a_{n+1}^* \}$
\end{enumerate}
One can easily check that there are $n$ maximally singular vectors $w_1,\ldots,w_n$ in $\overline{\mathfrak a^+}$ where $w_i$ is the vector defined by
$$\text{ the $k$-th coordinate of $w_i$} = \left\{ \begin{array}{ll} i & \text{ if } 1\leq k \leq n-i+1 \\ i-n-1 & \text{ if } n-i+2 \leq k \leq n+1 \end{array} \right.$$
Note that any maximally singular vector in $\mathfrak a$ is conjugate to one of these by a Weyl symmetry.
From a straightforward computation, it follows that for each $w_i$ $$\dim(Q_{w_i})=in-i(i-1).$$
This implies that the minimum for $\dim(Q_{w_i})$ among the $w_i$s is $n$. Furthermore, all maximally singular vectors $w$ with $\dim(Q_w)=n$ can be listed up to sign and scaling as follows:
$$(1,\ldots,1,-n), (1,\ldots,1,-n,1),\ldots, (-n,1,\ldots,1).$$
If $w$ and $w'$ are two linearly independent maximally singular vectors in $\mathfrak a$ with $\dim(Q_w)=\dim(Q_{w'})=n$, then $\dim(Q_w \cap Q_{w'})=1$. Note that if $n\geq 5$, the second and third minimum for $\dim(Q_{w_i})$ are $2n-2$ and $3n-6$ respectively.

\begin{proposition}\label{pro:slnR}
Property $\bf E$ holds for $\rsl_{n+1}\mathbb R$ if $n\geq 5$.
\end{proposition}

\begin{proof}
Recalling Remark \ref{remark:propertyE}, it suffices to show the statement of property {\bf E} for any frame of maximally singular vectors. Hence we may assume that each $v_i$ is maximally singular.
Set $u_i=u_{v_i}$.
First of all, we reorder and relabel the rows of $A$ so that $|u_i| \leq |u_{i+1}|$ for all $i$.
As mentioned before, $|u_i|=n$ or else $|u_i| \geq 2n-2$ for each $i$ since $|u_i|=\dim(Q_{v_i})$.
We will first show that one algorithm, described below, succeeds on the $\{ u_i : |u_i|=n\}$ and then continue with another algorithm on the $\{ u_i : |u_i| = 2n-2\}$.
Now assume that there exists $p>0$ such that $|u_i|=n$ for each $1\leq i\leq p$.
For any $t$, let $N(i,t)$ denote the number of $1$'s left in the vector $u_i$ at the start of Stage $t$ of the algorithm.
Starting with $t=1$, perform Stage $t$ of the following algorithm described in \cite{CF14} on the row vectors $\{u_1,\ldots,u_p\}$.
\begin{enumerate}
\item Step 1: Reorder the rows so that for each $1\leq i \leq p-t$  $$N(i,t) \leq N(i+1,t).$$
\item Step 2: Choose three 1 entries of the top row; let $k_t, l_t, m_t$  be the column numbers of these three entries.
\item Step 3: Delete the top row and columns $k_t, l_t, m_t$, still calling the remaining vectors $u_j$ by their original names. Now increase the counter $t$ by 1 and go to Step 1.
\end{enumerate}

Let $u_{j_t}$ denote the top row after performing Step $1$ of Stage $t$. For instance $j_1=1$. Suppose that the algorithm fails at Stage $d$ for the first time.
Then $N(j_d,d)\leq 2$. It follows that $N(i,t)\geq N(i,t-1)-1$ since $|u_i \cap u_j |=1$ for any $1\leq i\neq j \leq p$. This implies that \begin{eqnarray}\label{dstage} 2 \geq N(j_d,d) \geq n-d+1 \end{eqnarray} and so $d \geq n-1$.
Hence Stage $n-1$ and $n$ are the only possible stages to fail.

Suppose that Stage $n-1$ fails for the first time. Then $N(j_{n-1},n-1)=2$, by (\ref{dstage}). A repeated application of the inequality $N(i,t)\geq N(i,t-1)-1$ gives that $N(j_k,2) \leq n-1$ for all $2\leq k \leq n-1$. This means that each of the $u_{j_2},\ldots, u_{j_{n-1}}$ must have a $1$ in one of the three columns removed from $u_1$. Since $N(1,1)=n\geq 5$ if $n\geq 5$, there exists a $1$ entry of $u_1$ not yet removed which does not overlap with any of $u_{j_2},\ldots,u_{j_{n-1}}$. Choose this entry to remove at Stage $1$ instead of choosing the $1$ entry of $u_1$ that overlaps with $u_{j_{n-1}}$. Then we give a $1$ back to $u_{j_{n-1}}$ so that $N(j_{n-1},n-1)=3$. This implies that the algorithm does not fail at Stage $n-1$. Now we may assume that the algorithm can not fail except at Stage $n$.
We prove a lemma before considering the case that Stage $n$ fails.

\begin{lemma}\label{lem:n}
Suppose that $|u_i|=n$ for each $1\leq i\leq n$. For $i\neq j$, there is exactly one $k$ with $a_{ik}=a_{jk}=1$. Furthermore there is the only one $1$ entry of $u_i$ that does not overlap with any other $u_j$. The other $1$ entry of $u_i$ overlaps with exactly one of the $u_j$ for $j \neq i$.
\end{lemma}

\begin{proof}[Proof of Lemma]
In the $\rsl_{n+1}\mathbb R$ case, the number of columns of the matrix $A$ is $n(n+1)/2$. Each column of $A$ has at least one $1$ entry by \cite{CF14}. Let $A_k$ be the submatrix of $A$ consisting of the first $k$ rows of $A$.
Let $c(k)$ be the number of columns of $A_k$ which have at least one entry with $1$.
Clearly $c(1)=n$ and $c(n)=n(n+1)/2$.

For $i\neq j$, $|u_i\cap u_j|=1$. Hence $c(2)-c(1) \geq n-1$. For a similar reason, $c(3)-c(2) \geq n-2$. Inductively, $$c(i+1)-c(i) \geq n-i$$ for $i=1,\ldots,n-1$.
Then $$c(n)=c(1) + \sum_{i=1}^{n-1} (c(i+1)-c(i)) \geq n+ \sum_{i=1}^{n-1} (n-i) = \frac{n(n+1)}{2}.$$
Since $c(n)=n(n+1)/2$, it follows that for $i=1,\ldots,n-1$,
\begin{eqnarray}\label{equality}c(i+1)-c(i) =n-i.\end{eqnarray}
It can be easily checked that (\ref{equality}) implies the lemma.
\end{proof}

Suppose that Stage $n$ fails.
By (\ref{dstage}), it follows that $N(j_n,n)=1$ or $2$. If $N(j_n,n)=1$, it can be easily seen that exactly one $1$ entry of $u_{j_n}$ is removed at each stage before Stage $n$. Noting that the number of columns of the matrix $A$ is $n(n+1)/2$ and
$$\frac{n(n+1)}{2} - 3(n-1) \geq 3 \text{ if } n\geq 5,$$
there must be at least three columns not yet removed at the start of Stage $n$.
Since $N(j_n,n)=1$, there are at least two columns not yet removed and not having a $1$ in row $u_{j_n}$. Denote two of these columns by $c_1$ and $c_2$.
Since each column must have at least one $1$ entry, we may assume that $c_1$ (respectively, $c_2$) has a $1$ entry of $u_p$ (respectively, $u_q$) for some $p,q \neq j_n$ (possibly $p=q$). When $u_p$ is the top row at some stage, one of the three columns removed from $u_p$ has a $1$ entry of $u_{j_n}$ as mentioned before. Put this column back and remove $c_1$ instead.
Then we can give a $1$ back to $u_{j_n}$.
If it is possible to choose $q\neq p$, then do the same thing with $u_p$ replaced by $u_q$ and $c_1$ replaced by $c_2$. Then we can give one more $1$ back to $u_{j_n}$ so that the algorithm does not fail at Stage $n$. If it is not possible to choose $p, q$ with $p \neq q$, each of $c_1$ and $c_2$ has only one $1$ entry of $u_p$ that does not overlap with any other row. This means that $u_p$ has two $1$'s that do not overlap with any other row. However this can not happen since each row has only one $1$ that does not overlap with any other row by Lemma \ref{lem:n}. Thus we are done in the case of $N(j_n,n)=1$.

If $N(j_n,n)=2$, a $1$ entry of $u_{j_n}$ must be removed at each stage before Stage $n$, except for only one Stage $p$ for some $1\leq p\leq n-1$. As shown before, there are at least three columns not yet removed at the beginning of Stage $n$. Since $N(j_n,n)=2$, there is a column $c_3$ not yet removed and not having a $1$ in row $u_{j_n}$.
If we can choose a $1$ entry of $u_q$ in $c_3$ for $q \neq p$, one of the three columns removed from $u_q$ has a $1$ entry of $u_{j_n}$. Put this column back and choose $c_3$ instead. Then $N(j_n,n)=3$ and so the algorithm does not fail at Stage $n$. If we can not choose a $q$ as above, this means that $c_3$ has the only one $1$ entry of $u_p$ that does not overlap with any other row. Then since each row has only one $1$ entry that does not overlap with any other row, any $1$ entry removed from $u_p$ overlaps with some other row. Let $c_4$ be a column removed from $u_p$ such that $c_4$ has a $1$ entry of $u_r$ for some $r \neq p, j_n$. Then first put the column $c_4$ back, and remove $c_3$ instead. Since $r\neq p$, one of the three columns removed from $u_r$ has a $1$ entry of $u_{j_n}$. Again put this column back and remove $c_4$ instead when we choose three $1$'s of $u_r$.
Then we can give a $1$ back to $u_{j_n}$ so that the algorithm succeeds at Stage $n$.

So far, we have shown that we can perform this algorithm successfully on the row vectors $\{u_i : |u_i|=n \}$ for $n\geq 5$.
Now we need to continue with an algorithm to be applied to the row vectors $\{ u_i : |u_i| = 2n-2 \}$.
However it seems hard to get a successful algorithm by applying the previous algorithm to these row vectors.
We need another algorithm. First, we need to observe something more about a maximally singular vector $v$ with $\dim(Q_v)=2n-2$.

Recall the list of maximally singular vectors $w$ with $\dim(Q_w)=n$ up to sign and scaling. The list consists of the vectors $z_1,\ldots,z_{n+1}$  in $\mathbb R^{n+1}$ where $z_i$ is the vector defined by
$$\text{ the $k$-th coordinate of $z_i$} = \left\{ \begin{array}{rl} -n & \text{if } k=n-i+2 \\ 1 & \text{Otherwise} \end{array} \right.$$
Note that if $v$ is a maximally singular vector in $\mathfrak a$ with $\dim(Q_v)=2n-2$, then $v=z_i + z_j$ for some $i\neq j$ up to sign and scaling.
Let $A$ be the matrix associated to a given frame $\{v_1,\ldots,v_n\}$ consisting of maximally singular vectors in $\mathfrak a$.
Assume that $|u_i|=\dim(Q_{v_i}) \leq 2n-2$ for each $1\leq i\leq p$. By assumption, $v_i$ is parallel to either one of the $z_i$ or a sum of two $z_i$s for each $1\leq i\leq p$. Since $A$ is invariant under the positive and negative scalings of each vector $v_i$, we may assume that $v_i$ is either one of the $z_i$s or a sum of two of them.
Let $\tilde v_i$ be the coordinate representation of $v_i$ with respect to the set $\{z_1,\ldots, z_{n+1} \}$. For example, $\tilde z_1 =(1,0,\ldots,0)$ and the coordinate representation of $z_1 +z_2$ is $(1,1,0,\ldots,0)$. Putting $\tilde v_i$ as the $i$th row vector of a matrix for each $i=1,\ldots,p$, we obtain a $p$  by $n+1$ matrix $\widetilde A_p$.

\begin{claim}
{\it For each $1\leq i \leq p$, there exists $1\leq k_i \leq n+1$ so that the $(i, k_i)$th entry of $\widetilde A_p$ is $1$ and $\cup_{i=1}^p \{k_i\}$ has cardinality $p$.}
\end{claim}
\begin{proof}[Proof of the claim]
It suffices to prove that there exists a $p$ by $p$ submatrix of $\widetilde A_p$ with nonzero determinant.
Let $c_i \in \mathbb R^p$ denote the $i$th column vector of $\widetilde A_p$ for $1\leq i \leq n+1$.
Suppose that every $p$ by $p$ submatrix of $\widetilde A_p$ has zero determinant.
Then it can be seen that the dimension of the vector space $V_p$ generated by $c_1,\ldots,c_{n+1}$ is less than $p$.
Hence by reordering and relabeling the columns of $\widetilde A_p$, we may assume that $\{c_1,\ldots, c_q\}$ is a basis of $V_p$ for $q <p$. Let $\widetilde B_p$ be the $p$ by $n$ matrix whose $i$th column vector consists of the $c_i-c_{n+1}$ for $i=1,\ldots,n$. Then since $z_1+\cdots+z_{n+1}=0$, the $j$th row of $\widetilde B_p$ is actually the coordinate representation of $v_j$ with respect to $\{z_1,\ldots,z_n\}$. Since $v_1,\ldots,v_p$ are linearly independent vectors and $\{z_1,\ldots,z_n\}$ is a basis of $\mathfrak a$, the rank of $\widetilde B_p$ should be $p$. However the vector space generated by $c_1-c_{n+1}, \ldots c_n-c_{n+1}$ has dimension less than $p$, which contradicts the fact that the rank of $\widetilde B_p$ is $p$. Therefore there exists a $p$ by $p$ submatrix of $\widetilde A_p$ with nonzero determinant.
\end{proof}

According to the claim above, by reordering and relabeling the $z_i$'s, we may assume that $v_i = z_i $ for $1\leq i \leq q$ and $v_i=z_i+z_{k_i}$ for $q+1\leq i \leq p$.
The associated row vector $u_i$ to $v_i$ for $q+1\leq i \leq p$ can be described as follows. Let $(u_i)_j$ denote the $j$th column of $u_i$. Then put $(u_i)_j=1$ if $b_j \in Q_{z_i} \cup Q_{z_{k_i}}- Q_{z_i}\cap Q_{z_{k_i}}$ and otherwise $(u_i)_j=0$.
Since $\dim(Q_{z_i} \cap Q_{z_{k_i}})=1$, there are $n-1$ entries of $1$ in $u_i$ which come from $Q_{z_i}$ and $n-1$ entries of $1$ in $u_i$ which come from $Q_{z_{k_i}}$. For example, if $v=z_1+z_2=(1,\cdots,1,1,-n)+ (1,\cdots,1,-n,1)$, then by $ Q_v =  \sum_{\substack{ \alpha \in \Lambda^+ \\ \alpha(v) \neq 0 }} \mathfrak p_\alpha$, the set of $\alpha$'s with $\alpha(v) \neq 0$ is $\{a_i^*- a_{n+1}^* : i=1,\cdots,n-1              \}$, which excludes $a_n^*-a_{n+1}^*$, together with $\{a_j^*-a_n^* : j=1,\cdots,n-1     \}$.

Now we are ready to apply a new algorithm to the row vectors $\{ u_i : |u_i| = 2n-2 \}$. Note that $|u_i|=n$ for $1\leq i \leq q$ and $|u_i|=2n-2$ for $q+1 \leq i \leq p$.
First apply the previous algorithm to the row vectors $\{ u_i : |u_i|=n \}$ and then, starting with $t=q+1$, apply Stage $t$ of the following algorithm to the row vectors $\{ u_i : |u_i|=2n-2 \}$.

\begin{enumerate}
\item Step 1: Reorder the rows of $\{ u_i : |u_i|=2n-2 \}$ so that $N(i,t) \leq N(i+1,t)$ for each $q+1\leq i\leq p-t$
\item Step 2: Let $u_{j_t}$ be the top row. Then choose three $1$ entries of $u_{j_t}$ which are contained in $Q_{z_{j_t}}$ and let $k_t,l_t, r_t$  be the column numbers of these three entries.
\item Step 3: Delete the top row and columns $k_t,l_t$ and $r_t$, still calling the remaining vectors $u_{j_t}$ by their original names. Now increase the counter $t$ by $1$ and go to Step $1$.
\end{enumerate}

The essential thing about the new algorithm is to choose $1$ entries of $u_{j_t}$ which are contained in $Q_{z_{j_t}}$ and $z_{j_1}, z_{j_2},\ldots,z_{j_p}$ are all distinct. One advantage of this algorithm is that the algorithm does not fail until Stage $n-2$.
This comes from the distinctness of $z_{j_1}, z_{j_2},\ldots,z_{j_p}$ as follows.
If $|u_i|=2n-2$, we can write $u_i = u_i' +u_{k_i}''$  where $u_i'$ (respectively. $u_{k_i}''$) is the associated vector to $z_i$(respectively, $z_{k_i}$) with the $1$ entry of $Q_{z_i} \cap Q_{z_{k_i}}$ deleted. Hence $u_i'$ and $u_{k_i}''$ do not overlap and $|u_i'|=|u_{k_i}''|=n-1$.
Note that at most one $1$ entry of $u_i'$ is removed from each of the other rows.
For a similar reason, at most one $1$ entry of $u_{k_i}''$ is removed from $u_j$ if $j \neq k_i$. Thus at most two $1$'s of $u_i$ are removed from $u_j$ if $j\neq k_i$.
If $j=k_i$, then three $1$ entries of $u_{k_i}''$ can be removed from $u_{k_i}$ and thus three $1$'s of $u_i$ can be removed from $u_j$. However the case of $j=k_i$ can only happen at most once while performing the algorithm since $z_1,\ldots,z_p$ are all distinct. Hence we have that \begin{eqnarray}\label{u2} N(i, t )\geq 2n-2-3-2(t-2)=2n-2t-1.\end{eqnarray}
It can be easily checked from (\ref{u2}) that $N(i,n-2) \geq 3$. This means that the algorithm can only fail at Stages $n-1$ or $n$.

\begin{lemma}\label{lem:n-1dim}
Let $\{v_1,\ldots,v_n\}$ be any frame of $\mathfrak a$. Then each row of the associated matrix $A$ has at most one $1$ entry that does not overlap with any other row.
\end{lemma}
\begin{proof}
Recalling (\ref{eqnQv}), the number of $1$'s of $u_i$ that do not overlap with any other row is equal to the number of positive roots $\alpha$ such that
\begin{eqnarray}\label{n-1dim} \alpha(v_k)=0 \text{ for }k\neq i \text{ and }\alpha(v_i)\neq 0.\end{eqnarray}
Let $V_i$ be the vector subspace of $\mathfrak a$ generated by $v_1,\ldots,v_{i-1},v_{i+1},\ldots,v_n$.
If $V_i$ contains a regular vector, then there are no positive roots $\alpha$ such that $\alpha(v_k)=0$ for $k\neq i$ and $\alpha(v_i)\neq 0$. Suppose there are two positive roots $\alpha$ and $\beta$ satisfying (\ref{n-1dim}). Then $\alpha=t \beta$ for some positive real number $t$. This is forbidden unless $t=1$ since the root system is reduced. This implies the lemma.
\end{proof}

Assume that Stage $n-1$ fails. By (\ref{u2}), we have that $N(j_{n-1},n-1) \geq 1$.
If $N(j_{n-1},n-1) =1$, then at least two $1$'s of $u_{j_{n-1}}$ are removed from each of $u_{j_1},\ldots,u_{j_{n-1}}$ since $|u_{j_{n-1}}|=2n-2$.
Consider the submatrix $A_{n-1}$ of $A$ consisting of the rows $u_{j_1},\ldots,u_{j_{n-1}}$.
Then Lemma \ref{lem:n-1dim} implies that $A_{n-1}$ has at least $n(n+1)/2-1$ nonzero columns.
Hence there are at least $n(n+1)/2-1- 3(n-2)$ nonzero columns of $A_{n-1}$ not yet removed at the start of Stage $n-1$.
Since $n(n+1)/2-1 - 3(n-2) \geq 5$ for $n\geq 5$, there are at least four columns not yet removed and having a $1$ in at least one of the rows $u_{j_1}, \ldots, u_{j_{n-2}}$.
Let $c_1$ be one of such columns having a $1$ in row $u_k$ for some $k \in \{j_1,\ldots,j_{n-2} \}$.
As mentioned before, two of the three columns removed from $u_k$ have a $1$ in row $u_{j_{n-1}}$.
Put one of them back and remove $c_1$ instead. Then this procedure gives a $1$ back to $u_{j_{n-1}}$.
Doing the same procedure with another column not yet removed and having a $1$ in at least one of the rows $u_{j_1}, \ldots, u_{j_{n-2}}$, we can give one more $1$ back to $u_{j_{n-1}}$ so that Stage $n-1$ does not fail.

Even in case $N(j_{n-1},n-1) =2$, it can be seen that the algorithm could be modified in a similar way so that Stage $n-1$ does not fail. We have thus shown that the algorithm always succeeds until Stage $n-1$.

Now assume that Stage $n$ fails.
Let $r_i$ be the number of $1$ entries of $u_{j_n}$ removed at Stage $i$.
If $N(j_n,n)=0$, then it can be easily seen that
\begin{equation}\label{n0} (r_1,\ldots, r_{n-1})= (3,2,\ldots,2,1) \text{  or } (2,2,\ldots,2)\end{equation} up to order. The number of the columns not yet removed at the start of Stage $n$ is $n(n+1)/2- 3(n-1)$.
Since $n(n+1)/2- 3(n-1) \geq 3$ for $n\geq 5$, there are at least three columns not yet removed and not having a $1$ in row $u_{j_n}$. Choose three of them, say $c_1$, $c_2$ and $c_3$. Then we claim that it is possible to replace three of the removed columns having a $1$ in row $u_{j_n}$ with $c_1$, $c_2$ and $c_3$ by the same procedure as described in the case that Stage $n-1$ fails. If this were not possible, it would follow from (\ref{n0}) that two of $c_1$, $c_2$ and $c_3$ would have  $1$'s only in the row from which only one $1$ of $u_{j_n}$ is removed. However this contradicts Lemma \ref{lem:n-1dim}, and hence the claim holds. Thus Stage $n$ does not fail.

If $N(j_n,n)=1$, it can be easily checked that
\begin{equation}\label{n1} (r_1,\ldots, r_{n-1})= (3,2,\ldots,2,0), (3,2,\ldots,2,1,1) \text{ or }(2,2,\ldots,2,1) \end{equation} up to order.
Since $n(n+1)/2- 3(n-1)-1 \geq 2$ for $n\geq 5$, there are at least two columns not yet removed and not having a $1$ in row $u_{j_n}$. Except for only one case, it is possible to replace two of the removed columns having a $1$ in row $u_{j_n}$ with these two columns using a similar procedure as before, so that the Stage $n$ does not fail. The exceptional case is when $n=5$, $(r_1,r_2,r_3,r_4)= (3,2,2,0)$ up to order and one of the remaining two columns, denoted by $c_1$, has  one $1$ only in  row $u_p$ for $p\neq j_5$ from which no $1$ of $u_{j_n}$ is  removed. Let $c_2$ be one of the three columns removed from $u_p$. Then since $u_p$ has $1$ in $c_1$ that does not overlap with any other row, it follows from Lemma \ref{lem:n-1dim} that each of the three columns removed from $u_p$ has a $1$ in some row different from $u_p$. Hence we may assume that $c_2$ has a $1$ in some row $u_q$ for $q \neq p, j_5$. As $q \neq p, j_5$, two of the three columns removed from $u_q$ have a $1$ in row $u_{j_5}$. Denote by $c_3$ one of these two columns.
First put $c_2$ back and remove $c_1$ instead. Then put $c_3$ back and remove $c_2$ instead. This process gives a $1$ back to $u_{j_5}$. Let $c_1'$ be another remaining column  not having a $1$ in row $u_{j_5}$.
Due to Lemma \ref{lem:n-1dim}, $c_1'$ has a $1$ in some row different from $u_p$. This allows us to replace one of the removed columns having a $1$ in row $u_{j_5}$ with $c_1'$ so that we give one more $1$ back to $u_{j_5}$. Thus Stage $5$ does not fail.

One can easily see that a similar argument also works in the case of $N(j_n,n)=2$ so that Stage $n$ does not fail.
To see this, one should observe that there is at least one column not yet removed and not having a $1$ in  $u_{j_n}$ and, $(r_1,\ldots, r_{n-1})$ is equal to one of
$$(3,2,\ldots,2,1,0), (3,2,\ldots,2,1,1,1), (2,2,\ldots,2,0),(2,2,\ldots,2,1,1)$$ up to order.

We have shown that one can apply the algorithm successfully to the row vectors $\{u_i: |u_i| \leq 2n-2\}$. We continue with the first algorithm on the remaining row vectors $\{u_i : |u_i| \geq 3n-6\}$. Since $3n-6 \geq 3(n-2)$, the algorithm can only fail at Stages $n-1$ or $n$. By a similar argument as in the proof of the case of $\{u_i : |u_i| =2n-2\}$, it can be seen that one can perform the algorithm successfully. This completes the proof.
\end{proof}

\begin{remark}
Note that property {\bf E} always fails in the excluded cases of Proposition \ref{pro:slnR} for the following reason.
For Property {\bf E} to hold, the following inequality must be satisfied.
$$\dim(X)-\mathrm{rank}(X) \geq 3 \mathrm{rank}(X) \iff \dim(X) \geq 4 \mathrm{rank}(X).$$
If $X=\rsl_{n+1}\mathbb R/\rso(n+1)$, it is easy to see that the above inequality is equivalent to $n\geq 5$.
Hence Proposition \ref{pro:slnR} is optimal.
\end{remark}

Even if property {\bf E} does not hold for $X$, Theorem \ref{thm:eigenmatching2} might be true.
If this happens, then $\dim(X)-\mathrm{rank}(X)-1 \geq 2 \mathrm{rank}(X)$.
Only one of the excluded cases in Proposition \ref{pro:slnR} satisfies this inequality.
That one case is the $\rsl_5\mathbb R$ case. Indeed we prove that Theorem \ref{thm:eigenmatching2} holds for $\rsl_5\mathbb R$ without property {\bf E}

\begin{proposition}\label{prop:sl5r}
Theorem \ref{thm:eigenmatching2} holds for $\mathrm{SL}_5\mathbb R/\mathrm{SO}(5)$.
\end{proposition}

\begin{proof}
Following the proof of Theorem \ref{thm:9.1}, it suffices to prove Theorem \ref{thm:eigenmatching2} for $\epsilon$-orthonormal bases of a maximal flat $\mathcal F$.
Let $\{v_1,v_2,v_3,v_4\}$ be an $\epsilon$-orthonormal basis of a maximal flat $\mathcal F$.
In the case, $\rsl_5\mathbb R$, it is not difficult to see that there exists a sufficiently small $\epsilon>0$ such that any $\epsilon$-orthonormal frame must contain a regular vector in $\mathcal F$ away from the $\epsilon_0$-neighborhood of any singular vector.
For this reason, we may assume that every entry of $u_4$ is $1$. Then we only need a successful algorithm on the other three rows $u_1,u_2,u_3$ corresponding to maximally singular vectors. Let $A_3$ denote the matrix whose rows consist of $u_1,u_2,u_3$. According to Lemma \ref{lem:n-1dim}, $A_3$ has at least $9$ nonzero columns. Note that $|u_i|=4$ or $6$ for $i\neq 4$. As observed before, $N(j,t)\geq 4-t+1$ if $|u_{j_3}|=4$ and $N(j,t) \geq 8-2t-1$ if $|u_{j_3}|=6$. Hence only Stage $3$ can  possibly fail. If Stage $3$ fails and $|u_{j_3}|=4$, then $N(j_3,3)=2$. This implies that $|u_1|=|u_2|=|u_3|=4$ and $u_2, u_3$ must have a $1$ in one of the three columns removed from $u_1$. Since $|u_1|=4$, there exists a $1$ entry of $u_1$ that does not overlap with $u_2, u_3$. Thus we can give a $1$ back to $u_{j_3}$ so that Stage $3$ does not fail.

If Stage $3$ fails and $|u_{j_3}|=6$, then $N(j_3,3)=1$ or $2$. If $N(j_3,3)=1$, there are $9-6-1=2$ nonzero columns in $A_3$ not yet removed and not having a $1$ in row $u_{j_3}$, say $c_1$ and $c_2$. Noting that at least two $1$'s of $u_{j_3}$ are removed at each stage before Stage $3$, it is possible to replace two of the removed columns having a $1$ in row $u_{j_3}$ with $c_1$ and $c_2$ so that Stage $3$ does not fail. Also, the case $N(j_3,3)=2$ can be done similarly. Hence we now conclude that one can carry out the algorithm successfully on the row $u_1, u_2, u_3$.

We have shown that there exist pairwise orthogonal $3$-dimensional subspaces $P_1$, $P_2$ and $P_3$ such that $P_i \subset Q_{v_i}$ for $i=1, 2, 3$. Since $\dim(W\cap P_i) \geq 2$, we can choose orthonormal vectors $v_i'$, $v_i''$ in $W\cap P_i$ for each $i=1,2,3$. Noting that $\dim(X)=14$, $\dim(\mathcal F)=4$ and $\dim(W)=13$, it follows that $\dim(W\cap \mathcal F^\bot) \geq 9$. Hence we can choose $v_4'$, $v_4''$ in $W \cap \mathcal F^\bot$ so that $\{v_1',v_1'',\ldots,v_4',v_4''\}$ is an orthonormal $8$-frame contained in $W$.
Since $Q_{v_4}=\mathcal F^\bot$, the chosen vectors $v_1',v_1'',\ldots,v_4',v_4''$ satisfy the angle inequality in Theorem \ref{thm:eigenmatching2}. This completes the proof.
\end{proof}

Theorem \ref{propertye} together with Proposition \ref{prop:sl5r} completes the proof of Theorem \ref{thm:eigenmatching2} and thus Theorem \ref{thm:codim1Jac} is true.

\section{Rank $2$ irreducible symmetric spaces}

In this section we will prove Theorems \ref{rank2jacobian} and \ref{rank2thm2}.

\begin{proof}[Proof of Theorem \ref{rank2jacobian}]
First of all, note that it suffices to show Theorem \ref{thm:eigenmatching2} for any $\epsilon$-orthonormal basis of a maximal flat $\mathcal F$ and any $d$-dimensional subspace $W$ of $T_xX$. Let $v_1, v_2$ be an $\epsilon$-orthonormal basis of a maximal flat $\mathcal F$. Then it can be seen that there exists a sufficiently small $\epsilon>0$ such that any $\epsilon$-orthonormal frame must contain a regular vector in $\mathcal F$ away from the $\epsilon_0$-neighborhood of any singular vector. Hence we may assume that $Q_{v_2}=\mathcal F^\bot$.
As explained before, it suffices to consider the case that $v_1$ is maximally singular.

It follows from \cite[Proposition 2.11.4]{Eb} that $$t_X=\max_{v\in \mathcal F} \dim(\mathrm{Stab}_K(v)\cdot \mathcal F).$$ In fact the maximum is realized by a maximally singular vector $v$. Let $v_1$ be a maximally singular vector such that $\dim(Q_{v})$ is minimal among  maximally singular vectors. Then $t_X=\dim(X)-|u_1|$. If $d \geq t_X +2$,
\begin{align*}
\dim(W\cap Q_{v_1}) &= \dim(W)+\dim(Q_{v_1})-\dim(\mathrm{span}(W\cup Q_{v_1})) \\
&\geq  d+\dim(X)-t_X -\dim(X) = 2.\end{align*}
In other words, it is possible to choose an orthonormal $2$-frame given by vectors $v_1', v_1''$ in $W\cap Q_{v_1}$.
On the other hand if $\dim(W\cap \mathcal F^\bot)\geq 4$, then it is possible to choose another orthonormal $2$-frame given by vectors $v_2', v_2''$ in $W\cap \mathcal F^\bot$ such that $\{v_1', v_1'', v_2',v_2''\}$ is an orthonormal $4$-frame of $W$.
Since $$\dim(W\cap \mathcal F^\bot)\geq \dim(W)+\dim(\mathcal F^\bot)-\dim(X) =d-2,$$ we have that $\dim(W\cap \mathcal F^\bot)\geq 4$ if $d\geq 6$.
Therefore if $d\geq \max\{6, t_X+2\}$, we can conclude that Theorem \ref{thm:eigenmatching2} is still true even if we replace $W$ with any $d$-dimensional subspace.
\end{proof}

It can be easily checked that $t_X \geq 4$ for all rank $2$ irreducible symmetric spaces except the symmetric spaces associated to $\rsl_3\mathbb R$, $\rsp_4\mathbb R$ and $G_{2(2)}$. In the exceptional three cases, it holds that $t_X=3$.

\begin{proof}[Proof of Theorem \ref{rank2thm2}]
Let $\Omega^*(X,\mathbb R)$ be the space of $\mathbb R$-valued differential forms on $X$.
For $\omega \in \Omega^k(X,\mathbb R)$, define its sup norm by
$$\| \omega \|_\infty = \sup_{x\in X} \sup_{v_1,\ldots,v_k \in T^1_xX} |\omega_x(v_1,\ldots,v_k)|,$$
where $T^1_x X$ is the unit tangent sphere of $X$ at $x\in X$.

According to the van Est isomorphism, there is an isomorphism $$\mathcal J : \Omega^*(X,\mathbb R)^G \rightarrow H^*_c(G,\mathbb R).$$
Recall that Dupont \cite{Du76} gave  an explicit description for the isomorphism at the cochain level:
For a $G$-invariant $k$-form $\omega$ on $X$,
$$\mathcal J(\omega)(g_0,\ldots,g_k) =\int_{[g_0, \ldots, g_k]} \omega.$$
On the other hand, by using the barycentric straightening, define another map $\mathcal J_b(\omega) : G^{k+1} \rightarrow \mathbb R$ by
$$\mathcal J_b(\omega)(g_0,\ldots,g_k) =\int_{st_{k}(g_0, \ldots, g_k)} \omega$$
where $st_k(g_0,\ldots,g_k)$ is the barycentrically straightened simplex with vertices $g_0x,\ldots,g_kx$ for a fixed point $x \in X$.
Since $st_* : C_*(X) \rightarrow C_*(X)$ is $G$-equivariant and $G$-equivariantly homotopic to the identity, it easily follows that
$\mathcal J_b(\omega)$ is a continuous $G$-invariant cocycle and represents the same continuous cohomology class as $\mathcal J(\omega)$.
Furthermore since every $G$-invariant form is bounded and the volumes of straightened simplices of dimension $d\geq \max\{6, t_X+2\}$ are uniformly bounded from above by Theorem \ref{rank2jacobian},
$\mathcal J_b(\omega)$ is actually bounded if $\omega$ is a $G$-invariant differential form of degree $d\geq \max\{6, t_X+2\}$. Therefore the first statement of the theorem is completed.

Theorem \ref{rank2jacobian} implies that for $d\geq \max\{6,t_X+2\}$, the volumes of straightened $d$-simplices are uniformly bounded from above. Following the proof of Theorem \ref{codim1bounded}, the theorem immediately follows.
\end{proof}

\appendix

\section{Proof of Theorem \ref{propertye}}

This appendix provides the proof of Theorem \ref{propertye} for all the other higher rank irreducible symmetric spaces of noncompact type except the case of $\rsl_{n+1}\mathbb R/\mathrm{SO}(n+1)$.

\subsection{Classical simple Lie groups}

\subsubsection{$\rsl_{n+1}\mathbb C$}
 Recall that $\rsl_{n+1}\mathbb C=\{ g \in \mathrm{M}_{n+1} \mathbb C \ | \ \det(g)=1 \}$ where $\mathrm{M}_{n+1} \mathbb C$ is the set of square complex matrices of size $n+1$. Then the following are well known facts about $\rsl_{n+1}\mathbb C$.
\begin{enumerate}
\item $K=\mathrm{SU}(n+1)$
\item $\mathfrak g = \{  a \in \mathrm{M}_{n+1} \mathbb C \ | \ tr(a)=0 \}$
\item $\mathfrak a = \{ \mathrm{Diag}(a_1,\ldots,a_{n+1}) \ | \ \sum_{=1}^{n+1}a_i=0 \}$
\item $\Pi = \{ a_1^*-a_2^*, \ldots, a_n^* - a_{n+1}^* \}$
\end{enumerate}
It can be easily checked that the stabilizer group of $v=\mathrm{Diag}(1,\ldots,1,-n)$ in $K$ is $\mathrm{Stab}_K(v)=\mathrm{S}(\mathrm{U}(n)\oplus \mathrm{U}(1))$ and $v$ is a maximally singular vector in $\mathfrak a$ such that the dimension of $Q_v = (\mathrm{Stab}_K(v) \cdot \mathcal F)^\bot$ is minimal among the maximally singular vectors. The dimension of $Q_v$ is computed as follows:
$$\dim(Q_v)= \dim(\mathrm{SU}(n+1)) -\dim(\mathrm{S}(\mathrm{U}(n)\oplus \mathrm{U}(1)))=2n.$$
For the other maximally singular vector $w$ in $\mathfrak a$, it can be seen in a similar way that $\dim (Q_w) \geq  4n-4$. Note that for any linearly independent two vectors $v, v'$ in $\mathfrak a$ with $\dim(Q_v)=\dim(Q_{v'})=2n$,
the intersection $K_v \cap K_{v'}$ is isomorphic to $\mathrm{S}(\mathrm{U}(n) \oplus \mathrm{U}(1) \oplus \mathrm{U}(1))$. This implies that
\begin{eqnarray}\label{slC}
\dim(Q_v \cap Q_{v'}) &=& \dim(\mathrm{SU}(n+1)) -2\dim(\mathrm{S}(\mathrm{U}(n)\oplus \mathrm{U}(1))) \\ & &+\dim(\mathrm{S}(\mathrm{U}(n) \oplus \mathrm{U}(1) \oplus \mathrm{U}(1)))=2. \nonumber
\end{eqnarray}

First consider the case that there exists $p>0$ so that $|u_i|=2n$ for each $1\leq i \leq p$.
By (\ref{slC}), we have that $$N(i,t)\geq N(i,t-1)-2 \text{ and } N(j_d,d)\geq 2n-2(d-1).$$
Hence Stage $n$ it the only one that can fail, and $N(j_n,n)=2$. Assume that Stage $n$ fails. The fact of $N(j_n,n)=2$ and the repeated inequality of $N(i,t)\geq N(i,t-1)-2$ gives that
\begin{align}\label{ind}
N(j_k,2) &\leq  N(j_k,k)+2(k-2) \\
&\leq N(j_n,k)+2(k-2)\leq N(j_n,n)+2n-4=2n-2\nonumber
\end{align}
 for $k=2,\cdots,n$. This means that each $u_{j_2},\ldots,u_{j_n}$ must have two $1$'s in two of the three columns removed from $u_{j_1}$ during Step $2$ of Stage $1$. Since $2n>3$ for $n\geq 2$, there is an entry of $u_{j_1}$ that does not overlap with any other $u_{j_k}$. If we choose this entry of $u_{j_1}$ to remove at Step $2$ of Stage $1$ instead of the $1$ entry of $u_{j_1}$ that overlaps with $u_{j_n}$, the algorithm succeeds.

After carrying out the algorithm successfully on the $\{u_i : |u_i|=2n \}$, continue with the algorithm on the remaining row vectors $\{u_i : |u_i| \geq 4n-4 \}$. Since we remove three columns at each stage, the algorithm always succeeds if $4n-4 \geq 3n$ i.e., $n\geq 4$. Hence we only need to consider the cases of $\rsl_3 \mathbb C$ and $\rsl_4 \mathbb C$.

If $\mathrm{SL}_3 \mathbb C$, it is easily checked that $|u_1|=|u_2|=4$ and there are exactly two $k$'s with $a_{1k}=a_{2k}=1$. Hence the matrix $A$ can be written as
$$\left( \begin{array}{cccccc} 1 & 1&1&1&0&0 \\0&0&1&1&1&1 \end{array} \right).$$
If we choose the first $1$'s of $u_1$ and the last three $1$'s of $u_2$, the algorithm succeeds.

Now consider $\mathrm{SL}_4\mathbb C$. The number of columns of $A$ is $12$ and for each $i$ either $|u_i|=6$ or else $|u_i| \geq 8$. Thus the algorithm on the row vectors $\{u_1 : |u_i| \geq 8\}$ can only fail at Stage $3$.
Assume that Stage $3$ fails for the first time. Then $N(j_3,3)=2$. Since $|u_{j_3}|=8$, all the six columns removed at Stage $1$ and $2$ must have $1$ entries of $u_{j_3}$.
Note that there are $12-6=6$ columns not yet removed and not having a $1$ in row $u_{j_3}$.
Denote one of them by $c_1$. Then we may assume that $c_1$ has a $1$ in row $u_k$ for some $k\neq j_3$. Put back one of the three columns removed from $u_k$ and remove $c_1$ instead. Then we give a $1$ back to $u_{j_3}$ so that the algorithm does not fail at Stage $3$. In conclusion, we complete the proof in the case of $n\geq 2$.

\subsubsection{$\mathrm{Sp}_{2n}\mathbb C$}
The simple Lie group $\mathrm{Sp}_{2n}\mathbb C$ is defined by
$$\mathrm{Sp}_{2n}\mathbb C =\{ g \in \mathrm{SL}_{2n}\mathbb C \ | \ g^t J_{n,n}g=J_{n,n} \}$$
where $J_{n,n} =\left( \begin{array}{cc} O_n & I_n \\ -I_n & O_n \end{array} \right)$. Then we have the following facts.
\begin{enumerate}
\item $K=\mathrm{Sp}_{2n}\mathbb C \cap \mathrm{U}(2n) \cong \mathrm{Sp}(n)$
\item $\mathfrak a=\{ \mathrm{Diag}(a_1,\ldots,a_n)\oplus \mathrm{Diag}(-a_1,\ldots,-a_n) \ | \ a_i \in \mathbb R \}$
\item $\Pi = \{a_1^*-a_2^*,\ldots, a_{n-1}^*- a_n^*, 2a_n^* \}$.
\end{enumerate}
Denote $\mathrm{Diag}(a_1,\ldots,a_n)\oplus \mathrm{Diag}(-a_1,\ldots,-a_n)$ by $(a_1,\ldots,a_n)$.
Among the maximally singular vectors in $\mathfrak a$, the dimension of the stabilizer subgroup of $(1,0,\ldots,0)$ in $K$ is maximal. In fact, a straightforward computation shows that the stabilizer group is isomorphic to $\mathrm{U}(1) \times \mathrm{Sp}(n-1)$. Hence for any maximally singular vector $v$ in $\mathfrak a$, we have $$\dim (Q_v) \geq \dim (\mathrm{Sp}(n))-\dim(\mathrm{U}(1) \times \mathrm{Sp}(n-1))= 4n-2.$$
Since $4n-2 \geq 3n$ for $n\geq 2$, the algorithm always succeeds.

\subsubsection{$\mathrm{SO}_{2n}\mathbb C$}
The group $\rso_m \mathbb C$ is defined by
$$\rso_m \mathbb C=\{ g \in \mathrm{SL}_m\mathbb C \ | \ g^t g= I_{m} \}.$$
Recall the following well-known facts in the case of $m=2n$.
\begin{enumerate}
\item $K= \mathrm{SO}(2n)$
\item $\mathfrak a = \{ a_1 E \oplus \cdots \oplus a_n  E \ | \ a_i \in \mathbb R \}$ where $E=\left( \begin{array}{cc} 0 & i \\ -i & 0 \end{array} \right)$.
\item $\Pi= \{a_1^*- a_2^*,\ldots,a_{n-1}^*-a_n^*, a_{n-1}^*+a_n^* \}$
\end{enumerate}
Let $(a_1,\ldots,a_n)$ denote $ a_1 E \oplus \cdots \oplus a_n  E $ in $\mathfrak a$. It can be easily checked that the dimension of the stabilizer subgroup of $(1,0,\ldots,0)$ in $K$ is maximal among the maximally singular vectors. Noting that the stabilizer subgroup of $(1,0,\ldots,0)$ in $K$ is $\mathrm{SO}(2)\oplus \mathrm{SO}(2n-2)$, we have that for any maximally singular vector $v$ in $\mathfrak a$, $$\dim(Q_v) \geq \dim(\mathrm{SO}(2n))-\dim(\mathrm{SO}(2)\oplus \mathrm{SO}(2n-2))=4n-4.$$
If $4n-4 \geq 3n$ i.e., $n\geq 4$, then the algorithm always succeeds. We only need to show that the algorithm succeeds for $n=2, 3$. Notice that $\mathrm{SO}_4\mathbb C$ and $\mathrm{SO}_6\mathbb C$ are isogenous to $\mathrm{SL}_2\mathbb C \times \mathrm{SL}_2\mathbb C$ and $\mathrm{SL}_4\mathbb C$ respectively.
Therefore we conclude that property {\bf E} holds for  $n \geq3$.

\subsubsection{$\mathrm{SO}_{2n+1}\mathbb C$}
As seen in the previous section, a maximal abelian subalgebra of $\mathrm{SO}_{2n+1}\mathbb C$ can be written as
$$\mathfrak a =\{ a_1 E \oplus \cdots \oplus a_n  E \oplus 0 \ | \ a_i \in \mathbb R \}.$$
As in the previous case, it can be seen that for any maximally singular vector $v$,
$$\dim(Q_v) \geq \dim(\mathrm{SO}(2n+1))-\dim(\mathrm{SO}(2)\oplus \mathrm{SO}(2n-1))=4n-2.$$
Since $4n-2 \geq 3n$ for $n\geq 2$, the algorithm always succeeds.

\subsubsection{$\rsp_{2n}\mathbb R$}
The reduced root system of $\rsp_{2n}\mathbb R$ can be described as follows:
\begin{enumerate}
\item $\mathfrak a=\{ (a_1,\ldots,a_n) \ | \ a_i \in \mathbb R \}$
\item $\Delta^+ =\{ a_i^* \pm a_j^*, \ 2a_k^* \ | \ i<j, \ 1\leq k \leq n \}.$
\item $\Pi = \{ a_1^*-a_2^*,\ldots, a_{n-1}^*-a_n^*, 2a_n^*\}.$
\end{enumerate}

Then a maximally singular vector in the closed positive Weyl chamber $\overline{\mathfrak a^+}$ is one of the following.
$$w_1=(1,0,\ldots,0),  \ w_2=(1,1,0,\ldots,0), \ldots, \ w_n=(1,\ldots,1).$$
Since every root space has dimension $1$, it follows from (\ref{eqnQv}) that $\dim(Q_{v})$ is equal to the number of positive roots $\alpha$ with $\alpha(v)\neq 0$. Hence
$$\dim(Q_{w_1}) =2n-1, \  \dim(Q_{w_2})= 4n-5, \ldots, \ \dim(Q_{w_n})= n(n+1)/2.$$
Note that $A$ is an $n$ by $n^2$ matrix.

We will first show that one can carry out the algorithm successfully on the row vectors $\{u_i : |u_i|=2n-1\}$.
Assume that there is a $p >0$ such that $|u_i|=2n-1$ for each $1\leq i\leq p$. For $1\leq i < j\leq p$, there are exactly two $k$'s with $a_{ik}=a_{jk}=1$. This can be shown by counting the number of positive roots $\alpha$ with $\alpha(w_i) \neq 0$ and $\alpha(w_j) \neq 0$. For instance, for $w_1 = (1,0,\ldots,0)$ and $w_2=(0,1,0,\ldots,0)$, the set of positive roots $\alpha$ with $\alpha(w_i) \neq 0$ and $\alpha(w_j) \neq 0$ coincides with $\{ a_1^*\pm a_2^* \}$. Thus we have
\begin{eqnarray}\label{spn} N(i,t)\geq N(i,t-1)-2.\end{eqnarray}

Assume that the algorithm fails at Stage $d$ for the first time. Then (\ref{spn}) implies that
$$(2n-1)-2(d-1) \leq N(j_d,d)\leq 2$$
and thus  Stage  $n$ is the only possible stage at which the algorithm can fail.

Now we assume that Stage $n$ fails. Then clearly $N(j_n,n)=1$ or $2$.
If $N(j_n,n)=1$, a repeated application of (\ref{spn}) gives that $N(j,2)\leq 2n-3$ for all $2\leq j\leq n$. This means that each $u_{j_2},\ldots,u_{j_n}$ must have two $1$'s in two of the three columns removed from $u_{j_1}$.
Since $(2n-1)-3\geq 2$ for $n\geq 3$, there exist two $1$ entries of $u_{j_1}$ that do not overlap with any other row and have not yet been removed at the start of Stage $n$.
Choose these two $1$'s of $u_{j_1}$ at Stage $1$ instead of choosing the two $1$'s of $u_{j_1}$ that overlap with $u_{j_n}$. Then we give two $1$'s back to $u_{j_n}$ so that Stage $n$ does not fail.

If $N(j_n,n)=2$, there are $n^2-3n+3$ columns not yet removed. Since $n^2-3n+3 \geq 3$ for $n\geq 3$, there is a column $c_1$ not yet removed and not having a $1$ in row $u_{j_n}$. Hence we may assume that $c_1$ has a $1$ in $u_q$ for some $q\neq j_n$. It can be easily seen from $N(j_n,n)=2$ that $u_{j_n}$ must have a $1$ in one of the three columns removed from $u_q$. Put this column back and remove $c_1$ instead. Then Stage $n$ does not fail.

So far we have shown that one can perform the algorithm successfully on the row vectors $\{u_i : |u_i|=2n-1 \}$ if $n\geq 3$.
Indeed, it is impossible to perform the algorithm successfully in the case of $\rsp_4\mathbb R$ since the number of columns of $A$ is just $4 < 3\times 2=6$. We now continue with the algorithm on the remaining row vectors.
If $n\geq 5$, then each of the remaining row vectors has at least $3n$ entries of $1$. Thus the algorithm always succeeds.

Consider the case of $\rsp_6\mathbb R$. The matrix $A$ is a $3$ by $9$ matrix and $|u_i|= 5,6$ or $7$.
We only need to show that the algorithm succeeds on the remaining vectors $\{u_i : |u_i|\geq 6 \}$.
Since $|u_i|\geq 6$, only Stage $3$ can possibly fail.
Assume that Stage $3$ fails for the first time and $|u_i|\geq 6$. If $N(j_3,3)=0$, it is clear that $|u_i|=6$ and $u_{j_3}$ must have a $1$ in all of the six columns removed at Stage $1$ and $2$. Since there are three columns not yet removed and not having a $1$ in row $u_{j_3}$, it is possible to give three $1$'s back to $u_{j_3}$ in a similar way as before. Thus Stage $3$ does not fail.

If $N(j_3,3)=1$, then $u_{j_3}$ must have a $1$ in at least five of the six columns removed at Stage $1$ and $2$.
Two of the three columns not yet removed do not have a $1$ in row $u_{j_3}$.
This allows us to give two $1$'s back to $u_{j_3}$ in a similar way. Hence Stage $3$ does not fail.
Also the case of $N(j_3,3)=2$ can be done similarly, so that Stage $3$ does not fail.
Thus we are done for the case of $\rsp_6\mathbb R$.

Next consider the case of $\rsp_8\mathbb R$. It can be easily checked that $A$ is a $4$ by $16$ matrix and each $|u_i|$ is either $7, 10, 11$ or $12$. We only need a successful algorithm on on the row vectors $\{u_i :  |u_i|\geq 10\}$.
Since $|u_i|\geq 10$, Stage $4$ can only fail with $N(j_4,4)=1$ or $2$.
If Stage $4$ fails for the first time, then $u_{j_4}$ must have a $1$ in at least eight of the nine columns removed at Stage $1$, $2$ and $3$.
Note that there are at least five columns not yet removed and not having a $1$ in row $u_{j_4}$.
By using these five columns, one can modify the algorithm in a similar way as done before so that Stage $4$ does not fail. Therefore property {\bf E} holds for $\rsp_{2n}\mathbb R$ if $n\geq 3$.

\subsubsection{$\rsl_{n+1}\mathbb H$}
A maximal compact subgroup of $\rsl_{n+1}\mathbb H$ is $\mathrm{Sp}(n+1)$ and its maximal abelian subalgebra $\mathfrak a$ is identified with $$\mathfrak a=\left\{(a_1,\ldots,a_{n+1}) \ \bigg| \ \sum_{i=1}^{n+1} a_i =0 \right\}$$
with $\Pi = \{a_1^*-a_2^*,\ldots,a_n^*-a_{n+1}^* \}$. It can be easily checked that the dimension of the stabilizer subgroup of $(1,\ldots,1,-n)$ in $\rsp(n+1)$ is $\mathrm{Sp}(n)\oplus \mathrm{Sp}(1)$ and moreover is maximal among the maximally singular vectors in $\mathfrak a$. Hence it follows that for any maximally singular vector $v$ in $\mathfrak a$,
$$\dim(Q_v) \geq \dim(\mathrm{Sp}(n+1))-\dim(\mathrm{Sp}(n)\oplus \mathrm{Sp}(1))=4n.$$
Since $4n >3n$, the algorithm always succeeds.

\subsubsection{$\mathrm{SU}(m,n)$} Assume that $m\geq n \geq 2$. Then we have that
\begin{enumerate}
\item $K=\mathrm{S}(\mathrm{U}(m)\oplus \mathrm{U}(n))$
\item $\mathfrak a=\{ (a_1,\ldots,a_n) \ | \ a_i \in \mathbb R \}$
\item $\Pi = \{a_1^*-a_2^*,\ldots,a_{n-1}^*-a_n^*, 2 a_n^*\}.$
\end{enumerate}
Except for the $\su(2,2)$ case, the dimension of the stabilizer subgroup of $(1,0,\ldots,0)$ in $K$ is maximal among the maximally singular vectors in $\mathfrak a$ and the stabilizer subgroup is isomorphic to $\mathrm{S}(\mathrm{U}(1)\times \mathrm{U}(m-1)\times \mathrm{U}(n-1))$. We remark that in the $\su(2,2)$ case, the stabilizer subgroup of $(1,1)$ has maximal dimension $4$ among the maximally singular vectors. Hence in all cases except $\su(2,2)$,
\begin{eqnarray*}
\dim(Q_v) &\geq& \dim(\mathrm{S}(\mathrm{U}(m)\oplus \mathrm{U}(n)))-\dim(\mathrm{S}(\mathrm{U}(1)\times \mathrm{U}(m-1)\times \mathrm{U}(n-1))) \\ &=& 2(m+n)-3
\end{eqnarray*}
for any maximally singular vector $v$ in $\mathfrak a$. Since $2(m+n)-3 \geq 3n$ for all $(m,n)\neq (2,2)$, the algorithm always succeeds.
We will deal with the remaining case of $\su(2,2)$ in next section. Note that $\su(2,2)$ is isogenous to $\rso(4,2)$.

\subsubsection{$\mathrm{SO}(m,n)$, $m\geq n \geq 2$} The simple Lie group $\mathrm{SO}(m,n)$ is defined by
$$\mathrm{SO}(m,n) = \left\{ g\in \rsl_{m+n}\mathbb R \  \bigg| \ g^t \left( \begin{array}{cc} I_m & O_{m\times n} \\ O_{n\times m} & -I_n \end{array} \right) g =  \left( \begin{array}{cc} I_m & O_{m\times n} \\ O_{n\times m} & -I_n \end{array} \right) \right\}$$
Its maximal abelian subalgebra $\mathfrak a =\{ (a_1,\ldots,a_n) \ | \ a_i \in \mathbb R \}$ can be described as follows.
To each $(a_1,\ldots, a_n) \in \mathfrak a$, one can associate a square matrix $a$ of size $(m+n)$ with all $a_{i,j}=0$ but $a_{m-i+1,m+i}=a_{m+i,m-i+1}=a_i$ for $i=1,\ldots,n$. Then it can be seen that the minimum of $\dim(Q_v)$ among the maximally singular vectors in $\mathfrak a$ is $m+n-2$ and the second minimum for $\dim(Q_v)$ is either $2n+2m-7$ or $n(2m-n-1)/2$ depending on $m, n$. For example, for $w_1=(1,0,\ldots,0)$, $w_2=(1,1,0,\ldots,0)$ and $w_3 = (1,\ldots,1)$, one can check that $\dim(Q_{w_1}) = m+n-2$, $\dim(Q_{w_2})=2n+2m-7$ and $\dim(Q_{w_3})=n(2m-n-1)/2$  by computing the stabilizer subgroup of each $w_i$ in $K$.  Furthermore, note that $\dim(Q_v \cap Q_{v'})=2$ for any two linearly independent vectors $v, v'$ with $\dim(Q_v)=\dim(Q_{v'})=m+n-2$.
We consider two cases: $m\neq n$ or $m=n$.

\vspace{3mm}

{\bf Case I : $m\neq n$}

\smallskip

First consider the case that there exists $p>0$ such that $|u_i|=m+n-2$ for $1\leq i \leq p$. Since for $i\neq j$ there are exactly two $k$'s with $a_{ik}=a_{jk}=1$, we have \begin{eqnarray}\label{somn} N(i,t) \geq N(i,t-1) -2.\end{eqnarray}
If Stage $d$ fails for the first time, then $$m+n-2-2(d-1) \leq N(j_d,d) \leq 2.$$
This implies that Stage $n$ can only fail in the case that $m= n+1$ or $n+2$. In all other cases, the algorithm succeeds.

Consider the case that $m=n+1$ and the algorithm fails at Stage $n$. Then it is easy to see that $N(j_n,n)=1$ or $2$. If $N(j_n,n)=1$, one can notice that $A$ is an $n$ by $n^2$ matrix and $|u_i|=2n-1$ for $1\leq i\leq p$. This is exactly the same situation as with the $\rsp_{2n}\mathbb R$ case. Following the proof in the $\rsp_{2n}\mathbb R$ case, one can prove that the algorithm succeeds if $n\geq 3$.
If $n=2$, this is the case of $\rso(3,2)$. Since $\rso(3,2)$ is isogenous to $\rsp_4\mathbb R$, the algorithm always fails in the case of $\rso(3,2)$ as seen before.

We now consider the case that $m=n+2$ and Stage $n$ fails. Then $N(j_n,n)=2$ and $|u_i|=2n$ for $1\leq i\leq p$. Repeated application of (\ref{somn}) gives that $N(j,2)\leq 2n-2$ for all $2\leq j\leq n$. This means that each $u_2,\ldots,u_n$ must have two $1$'s in two of the three columns removed from $u_1$. Since $2n>3$ for $n\geq 2$, there is a $1$ entry of $u_1$ that does not overlap with any other $u_j$ and has not yet been removed. Choose this $1$ entry of $u_1$ at Step $2$ of Stage $1$ instead of the $1$ entry of $u_1$ that overlaps with $u_{j_n}$. Then Stage $n$ does not fail.

We have shown that one can apply the algorithm successfully to the row vectors $\{ u_i \ : \ |u_i|=m+n-2 \}$ if $(m,n)\neq(3,2)$. We now continue with the algorithm on the remaining rows $\{ u_i : |u_i| \geq \min \{ 2m+2n-7 \text{ or } n(2m-n-1)/2 \} \}$.
Note that if $(m,n) \neq (3,2), (4,2), (4,3), (5,4)$, $$\min \{ 2m+2n-7 \text{ or } n(2m-n-1)/2 \} \geq 3n.$$
Hence the algorithm always succeeds on the remaining row vectors.
We only need to check the cases of $\rso(4,2)$, $\rso(4,3)$ and $\rso(5,4)$.

In the case of $\rso(4,2)$, $A$ is a $2$ by $6$ matrix and, each $|u_i|=4$ or $5$. In this case, one can directly show that it is possible to pick three $1$ entries from each row of $A$ so that all of the $6$ choices are in different columns.

In the case of $\rso(4,3)$, $A$ is a $3$ by $9$ matrix and $|u_i| = 5, 6$ or $7$. The algorithm applied to the remaining row vectors $\{u_i : |u_i| \geq 6 \}$ can only fail at Stage $3$ with $|u_{j_3}|\geq 6$.
If $N(j_3,3)=0$, this means that $u_{j_3}$ must have a $1$ in all of the six columns removed at Stage $1$ and $2$.
There are three columns not yet removed and not having a $1$ in row $u_{j_3}$. Then these three columns allow us to give three $1$'s back to $u_{j_3}$ so that the algorithm does not fail at Stage $3$.
A similar argument works as well in the other case of $N(j_3,3)=1$ or $2$.

There only remains the case $\rso(5,4)$.
In this case, $A$ is a $4$ by $16$ matrix and either $|u_i|=7$ or $|u_i| \geq 10$ for $i=1,2,3,4$. The algorithm on the remaining row vectors $\{u_i : |u_i|\geq 10\}$ can only fail at Stage $4$ with $N(j_4,4) \geq 1$ and $|u_{j_4}| \geq 10$.
If $N(j_4,4)\leq 2$, then $u_{j_4}$ must have a $1$ in at least eight of the nine columns removed at Stage $1$, $2$ and $3$. Note that there are at least $16-9-2=5$ columns not yet removed and not having a $1$ in row $u_{j_4}$. As before, we can modify the algorithm so that Stage $4$ does not fail.
Finally we conclude that property {\bf E} holds for all $m> n$ except $(m,n)=(3,2)$.

\vspace{3mm}

{\bf Case II : $m= n$}

\smallskip

First of all, note that $\rso(2,2)$ and $\rso(3,3)$ are isogenous to $\rsl_2\mathbb R \times \rsl_2\mathbb R$ and $\rsl_4\mathbb R$ respectively. These are the cases in which the algorithm always fails and hence we assume that $n\geq 4$.
Recall that $A$ is an $n$ by $n(n-1)$ matrix and the minimum for $|u_i|$ is $2n-2$. The second minimum for $|u_i|$ is $n(n-1)/2$ if $n=4,5,6$ and otherwise it is $(4n-7)$. As in the case of $m\neq n$, we first consider the case that there exists $p>0$ such that $|u_i|=2n-2$ for $1\leq i \leq p$. If Stage $d$ fails, then we have \begin{eqnarray}\label{soon} 2n-2-2(d-1) \leq N(j_d, d) \leq 2 \end{eqnarray}
and so $d=n-1$ or $n$.

Assume that Stage $n-1$ fails for the first time. Then it is easy to see that $N(j_{n-1},n-1)=2$ by (\ref{soon}). This implies that each of $u_{j_2},\ldots, u_{j_{n-1}}$ must have two $1$'s in two of the three columns removed from $u_1$ during Step $2$ of Stage $1$.
Since $(2n-2)-3-2 \geq 1$ for $n\geq 4$, there exists at least one $1$ entry of $u_1$ which has not yet been removed and does not overlap with any other $u_{j_k}, k=2,\cdots,n-1$. Choose this entry to remove at Stage $1$ instead of choosing the $1$ entry of $u_1$ that overlaps with $u_{j_{n-1}}$.
Then the algorithm does not fail at Stage $n-1$.

Now we may assume that Stage $n$ fails. Then $0\leq N(j_n,n) \leq 2$. First consider the case of $N(j_n,n)=0$.
Then it can be seen that each of the $u_2,\ldots,u_n$ must have two $1$'s in two of the three columns removed from $u_1$.
Since $2n-2\geq 6$, there are at least three $1$'s of $u_1$ that do not overlap with any other row. On the other hand, note that Lemma \ref{lem:n-1dim} is true in the case of $\rso(n,n)$ since every root space has dimension $1$. Hence the case of $N(j_n,n)=0$ can never happen.

Next consider the case $N(j_n,n)=1$.
There are $n^2-4n+3$ columns not yet removed.
Since $n^2-4n+3 \geq 3$ for $n\geq 4$, there are at least two columns not yet removed and not having a $1$ in row $u_{j_n}$.
As before, it is not difficult to show that one can replace two of the columns already removed before the start of Stage $n$ and having a $1$ in row $u_{j_n}$ with the remaining two columns. This replacement gives two $1$'s back to $u_{j_n}$ and the algorithm does not fail at Stage $n$. A similar argument works well for the the case of $N(j_n,n)=2$. Therefore we conclude that one can perform the algorithm successfully on the row vectors $\{u_i : |u_i|=2n-2 \}$.

We now continue with the algorithm on the remaining row vectors $\{ u_i : |u_i| \geq \min\{ 4n-7, n(n-1)/2 \}\}$. If $n\geq 7$, then  $ \min\{ 4n-7, n(n-1)/2 \} \geq 3n$ and thus the algorithm always succeeds. Hence we only need a successful algorithm on the remaining rows when $n=4,5,6$.

First consider the case of $n=6$. Since $\min\{ 4n-7, n(n-1)/2 \} = 15$, only Stage $6$ can fail.
Assume that Stage $6$ fails for the first time and $|u_{j_6}|\geq 15$. Then $u_{j_6}$ must have a $1$ in at least thirteen of the fifteen columns removed at Stage $1,2,3,4$ and $5$. Note that there are $30-15=15$ columns not yet removed.
Since $N(j_6,6)\leq 2$, there are at least thirteen columns not yet removed and not having a $1$ in row $u_{j_6}$. Armed with Lemma \ref{lem:n-1dim}, this situation enables us to modify the algorithm so that Stage $6$ does not fail.

Next consider the case of $n=5$. One can check the following:
\begin{enumerate}
\item $|u_i| \in \{ 8, 10, 13, 14, 15\}$
\item For $i\neq j$, if $|u_i|=|u_j|=8$, then $|u_i \cap u_j | =2$.
\item For $i\neq j$, if $|u_i|=8$ and $|u_j|=10$, then $|u_i \cap u_j | =4$.
\item For $i\neq j$, if $|u_i|=10$ and $|u_j|=10$, then $|u_i \cap u_j | =6$.
\end{enumerate}
Since we are now carrying out the algorithm on the remaining rows $\{ u_i : |u_i| \geq 10 \}$, only Stage $4$ or $5$ can only fail. If Stage $4$ with $|u_{j_4}| \geq 10$ fails, it can be checked that $1\leq N(j_4,4) \leq 2$, $|u_{j_4}|=10$ and $|u_{j_k}|\leq 10$ for $k=1,2,3$. Moreover $u_{j_4}$ must have a $1$ in at least eight of the nine columns removed at Stages $1,2$ and $3$.
For $i\neq j$, if $|u_i| \leq 10$ and $|u_j|\leq 10$, then one can easily check that $|u_i|+|u_j|-|u_i\cap u_j|=14$.
This means that there are at least $14-9-2=3$ columns not yet removed and not having a $1$ in row $u_{j_4}$ and having a $1$ in row $u_{j_1}$ or $u_{j_2}$. Then it can be easily shown that it is possible to replace two of the columns removed before Stage $4$ and having a $1$ in row $u_{j_4}$ with two of the three columns as mentioned just before. This replacement gives two $1$'s back to $u_{j_4}$ and then the algorithm does not fail at Stage $4$.

Assume that Stage $5$ with $|u_{j_5}| \geq 10$ fails. If $N(j_5,5)=0$, then $|u_{j_5}|=10$, $|u_j| \leq 10$ for all $j$. Let $r_i$ be the number of $1$'s of $u_{j_5}$ removed from $u_{j_i}$. Then $(r_1,r_2,r_3,r_4) = (3,3,3,1) \text{ or } (3,3,2,2)$ up to order. There are $20-12=8$ columns not yet removed and not having a $1$ in row $u_{j_5}$.
A similar argument as before together with Lemma \ref{lem:n-1dim} enables us to replace three of the removed columns having a $1$ in row $u_{j_5}$ with three of the eight columns. Then the algorithm does not fail at Stage $5$.

If $N(j_5,5)=1$, then $|u_{j_5}|= 10$ or $13$. In both cases, it can be checked that $(r_1,r_2,r_3,r_4)$ is one of $(3,3,3,3), (3,3,3,0), (3,3,2,1) \text{  or }(3,2,2,2)$, up to order.
There are $8-1=7$ columns not yet removed and not having a $1$ in row $u_{j_5}$. Again one can replace three of the removed columns having a $1$ in row $u_{j_5}$ with three of the seven columns. Then Stage $5$ does not fail.

If $N(j_5,5)=2$, then $|u_{j_5}|= 10, 13$ or $14$.
Observing that $(r_1,r_2,r_3,r_4)$ is one of $(3,3,3,3),(3,3,3,2), (3,3,2,0), (3,3,1,1), (3,2,2,1) \text{ or } (2,2,2,2)$ up to order and
there are $8-2=6$ columns not yet removed and not having a $1$ in row $u_{j_5}$,
it can be easily seen that one of the removed columns having a $1$ in row $u_{j_5}$ can be replaced with one of the six columns. Then the replacement gives a $1$ back to $u_{j_5}$ so that the algorithm does not fail at Stage $5$. Therefore we are done for  $\rso(5,5)$ case.

Finally consider the case of $n=4$. Then $|u_i|= 6$ or $9$. The algorithm on the remaining row vectors $\{u_i : |u_i|=9 \}$ can only fail at Stage $4$. If the algorithm fails at Stage $4$ for the first time and $N(j_4,4) =0$, then $u_{j_4}$ must have a $1$ in all of the nine columns removed before the start of Stage $4$. There are three columns not yet removed and none of them has a $1$ in row $u_{j_4}$. Then it is possible to replace three of the removed columns with the remaining three columns. The replacement gives three $1$'s back to $u_{j_4}$ so that the algorithm does not fail at Stage $4$. A similar argument can be made for the other cases of $N(j_4,4)=1, 2$ so that Stage $4$ does not fail. This completes the proof for $n\geq 4$.

\subsubsection{$\mathrm{Sp}(m,n)$}
Similarly to the case of $\mathrm{SU}(m,n)$, one can show that if $(m,n)\neq (2,2)$,
\begin{eqnarray*}
\dim(Q_v) &\geq& \dim(\mathrm{Sp}(m)\oplus \mathrm{Sp}(n))-\dim(\mathrm{Sp}(1)\times \mathrm{Sp}(m-1)\times \mathrm{Sp}(n-1)) \\ &=& 4(m+n)-5
\end{eqnarray*}
for any maximally singular vector $v$. Since $4(m+n)-5 \geq 3n$ for any $m\geq n\geq 2$, the algorithm always succeeds.
Even in the case of $(m,n)=(2,2)$, it holds that for any maximally singular vector $v$, $\dim(Q_v) \geq 10 > 3 \times 2$.
Therefore the algorithm always succeeds in any case.

\subsubsection{$\mathrm{SO}^*(2n)$} The simple Lie group $\mathrm{SO}^*(2n)$ is defined by
$$\mathrm{SO}^*(2n) = \{ g \in \mathrm{SU}(n,n) \ | \ g^t I_{n,n} J_{n,n} g= I_{n,n} J_{n,n} \}$$
where $I_{n,n}= \left( \begin{array}{cc} I_n & O_n \\ O_n  & -I_n \end{array} \right)$ and  $J_{n,n}= \left( \begin{array}{cc}  O_n & I_n \\ -I_n & O_n\end{array} \right)$.
Then its maximal compact subgroup is
$$K=\mathrm{SO}^*(2n) \cap \s (\ru(n)\oplus \ru(n))=\left\{ \left( \begin{array}{cc} B & O_n \\ O_n & \overline B \end{array} \right) \ \bigg| \ B \in \ru(n) \right\}.$$
Now we assume that $n=2k$. Note that the rank of the symmetric space associated to $\rso^*(2n)$ is $[n/2]=k$ and its maximal abelian subalgebra is
$$\mathfrak a = \left\{ \left( \begin{array}{cc} O_n & C \\ -C & O_n \end{array} \right) \ \bigg| \ C=a_1 \left( \begin{array}{cc} 0 & 1 \\ -1 &0 \end{array} \right) \oplus \cdots \oplus a_k \left( \begin{array}{cc} 0 & 1 \\ -1 &0 \end{array} \right)   \right\}.$$

First note that $\rso^*(8)$ is isogenous to $\rso(6,2)$ for which we are already done. Hence from now on we assume $k\geq 3$.
One can check that if $k \geq 3$, the dimension of the stabilizer subgroup of $(1,0,\ldots,0)$ is maximal among the maximally singular vectors and moreover the stabilizer subgroup is isomorphic to $\rsp(1) \times \ru (n-2)$. Hence for any maximally singular vector $v$, we have
$$\dim(Q_v) \geq \dim(\ru(n))-\dim( \rsp(1)\times \ru(n-2))= 4n-7 =8k-7.$$
Since $8k-7 \geq 3k $ for $k\geq 3$, the algorithm always succeeds.
In a similar way, it can be shown even in the case of $n=2k+1$ that any row of $A$ has at least $3k$ entries with $1$.
Thus we are done.

\subsection{Exceptional simple Lie groups}

\subsubsection{$E_6$}
First note that the real dimension of each root space is $2$.
Recall that the dimension of $Q_v$ for a maximally singular vector $v$ in $\mathfrak a$ is computed by counting positive roots $\alpha$ with $\alpha(v) \neq 0$. In this way, it can be seen that the minimum for $|\{\alpha \in \Lambda^+ \ | \ \alpha(v) \neq 0 \}|$ among the maximally singular vectors $v$ is $16$. Since the dimension of $\mathfrak p_\alpha $ is $2$, it follows that for any maximally singular vector $v$,
$$\dim(Q_v)\geq 32.$$
Since $32 > 6\times 3=18$, the algorithm always succeeds.

\subsubsection{$E_7$}
As for $E_6$, it can be seen that the minimum for $|\{\alpha \in \Lambda^+ \ | \ \alpha(v) \neq 0 \}|$ among the maximally singular vectors $v$ is $33$. Since the real dimension of each root space is $2$, we have that $$\dim(Q_v)\geq 66.$$
Since $66 > 7\times 3=21$, the algorithm always succeeds.

\subsubsection{$E_8$}
In the case of $E_8$, it turns out that $\dim(Q_v) \geq 114 > 8\times 3=24$ for any maximally singular vector $v$ and thus the algorithm always succeeds.

\subsubsection{$F_4$}
It can be checked that for any maximally singular vector $v$, $$\dim(Q_v) \geq 30 > 4\times 3=12$$ and hence the algorithm always succeeds.

\subsubsection{$G_2$}
In this case, one can easily show that $\dim(Q_v)=10 > 2\times 3=6$ for any maximally singular vector $v$. Thus the algorithm always succeeds.

\subsubsection{$E_{6(6)}$}
The real rank of $E_{6(6)}$ is $6$ and the number of columns of $A$ is $36$.
The reduced root system of $E_{6(6)}$ is $E_6$ and each root space has dimension $1$. As was seen in the $E_6$ case, for any maximally singular vector $v$ we have $\dim(Q_v) \geq 16$. Hence only Stage $6$ with $N(j_6,6 )\geq 1$ could fail.
Assume that Stage $6$ fails.
Observe that at least two $1$'s of $u_{j_6}$ are removed at each stage before Stage $6$, and there are at least $36-15-2=19$ columns not yet removed and not having a $1$ in row $u_{j_6}$. Then this situation allows us, as done before, to modify the algorithm so that Stage $6$ does not fail. Therefore we are done.

\subsubsection{$E_{6(2)}$} The reduced root system of $E_{6(2)}$ is $F_4$ with rank $4$ and thus
we have $\dim(Q_v)\geq 15 > 4 \times 3=12$ for any maximally singular vector $v$ (see the $F_4$ case). Thus the algorithm always succeeds.

\subsubsection{$E_{6(-14)}$} The rank of $E_{6(-14)}$ is $2$ and its maximal compact subgroup is isomorphic to $(\ru(1) \times \mathrm{Spin}(10))/\mathbb Z_4$. Hence it can be seen that $$\dim(Q_v) \geq 9 > 2\times 3=6.$$ This implies that the algorithm always succeeds.

\subsubsection{$E_{6(-26)}$} Recall that the reduced root system of $E_{6(-26)}$ is $A_2$ with rank $2$ and the dimension of the associated symmetric space is $26$. Hence $A$ is a $2$ by $24$ matrix. It can be easily seen that $|u_i| \geq 12 $ for $i=1, 2$. Therefore the algorithm always succeeds.

\subsubsection{$E_{7(7)}$} The simple Lie group $E_{7(7)}$ has the reduced root system of $E_7$. It is easy to see that each root space has dimension $1$. Hence as seen in the case of $E_7$, we have that for any maximally singular vector $v$, $$\dim(Q_v) \geq 33 > 7 \times 3 =21$$
and thus the algorithm always succeeds.

\subsubsection{$E_{7(-5)}$} Since the reduced root system of $E_{7(-5)}$ is $F_4$ and each root space has dimension at least $1$, it follows that for any maximally singular vector $v$, $\dim(Q_v) \geq 15 > 4 \times 3=12$.
Thus the algorithm always succeeds.

\subsubsection{$E_{7(-25)}$} The associated symmetric space to $E_{7(-25)}$ has dimension $54$ and rank $3$. Its reduced root system is $C_3$. Since the number of columns of $A$ is $51$, the dimension of $Q_v$ should be at least $17$. Hence the algorithm always succeeds.

\subsubsection{$E_{8(8)}$} The reduced root system of $E_{8(8)}$ is $E_8$ and each root space has dimension $1$. Hence it can be seen that $$\dim(Q_v) \geq 57 > 8\times 3=24$$ for any maximally singular vector $v$ (see the $E_8$ case). Thus the algorithm always succeeds.

\subsubsection{$E_{8(-24)}$} The associated symmetric space has dimension $112$ and rank $4$ and its reduced root system is $F_4$. Hence the dimension of $Q_v$ is at least $(112-4)/4=27$ for any maximally singular vector $v$. For this reason, the algorithm always succeeds.

\subsubsection{$F_{4(4)}$} The reduced root system of $F_{4(4)}$ is $F_4$ and each root space has dimension $1$. According to the proof in the case of $F_4$, we have that
$$\dim(Q_v) \geq 15 > 4\times 3 $$
for any maximally singular vector $v$. Thus the algorithm always succeeds.

\subsubsection{$G_{2(2)}$} The reduced root system of $G_{2(2)}$ is $G_2$ and each root space has dimension $1$. By the proof of the $G_2$ case, it follows that for any maximally singular vector $v$,
$\dim(Q_v)= 5$. Since $A$ is a $2$ by $6$ matrix, it is possible to choose a successful algorithm directly. Hence we are done.

{\bf Acknowledgements} We thank C. Connell and B. Farb for numerous discussions about weak eigenvalue matching theorem. We also
thank M. Bucher and  T. Hartnick for pointing out some inaccuracy in the earlier version of the paper.

\end{document}